\documentclass[11pt,reqno]{amsart}

\usepackage{epsfig}
\usepackage{graphicx}
\usepackage{mathrsfs}
\usepackage{subfigure}
\usepackage{verbatim}
\usepackage{amssymb}
\usepackage{latexsym}
\usepackage{amsmath}
\usepackage{amsfonts}
\usepackage{amsbsy}
\usepackage{amscd}
\usepackage{stmaryrd}
\usepackage[dvipsnames,usenames]{color} 
\usepackage{color}
\usepackage{epstopdf}
\usepackage{hyperref}
\usepackage{tikz}
\usetikzlibrary{arrows,shapes,chains}
\usepackage{amscd,cite} 

\newcommand{\R}{{\mathbb R}}

\newcommand{\e}{{\boldsymbol e}}

\newtheorem{prop}{Proposition}[section]
\newtheorem{lem}[prop]{Lemma}

\newtheorem{defi}[prop]{Definition}
\newtheorem{cor}[prop]{Corollary}
\newtheorem{thm}[prop]{Theorem}

\newtheorem{exam}[prop]{Example}
\newtheorem{rem}[prop]{Remark}

\begin{document}
\baselineskip=16pt

\title[$L^q$-spectrum of graph-directed self-similar measures]
{$L^q$-spectrum of graph-directed self-similar measures that have overlaps and are essentially of finite type}
\date{\today}

\author[Y. Xie]{Yuanyuan Xie}
\address{Institute of Science and Technology\\
Tianjin University of Finance and Economics\\
Tianjin, 300222\\
P. R. China.}\email{yyxiemath@163.com}

\thanks{The author is supported by the National Natural Science Foundation of China,
grant 12201453, and Excellent Young Scholars Support Program of Tianjin University of Finance and Economics.}



\begin{abstract}
For self-similar measures with overlaps, closed formulas of the $L^q$-spectrum have been obtained by Ngai and the author for measures that are essentially of finite type in [J. Aust. Math. Soc. \textbf{106} (2019), 56--103]. We extend the results of Ngai and the author \cite{Ngai-Xie_2019} to the graph-directed self-similar measures. For graph-directed self-similar measures satisfying the graph open set condition, the $L^q$-spectrum has been studied by Edgar and Mauldin \cite{Edgar-Mauldin_1992}. The main novelty of our results is that the graph-directed self-similar measures we consider do not need to satisfy the graph open set condition. For graph-directed self-similar measures $\mu$ on $\R^d$ ($d\ge1$), which could have overlaps but are essentially of finite type, we set up a framework for deriving a closed formula for the $L^q$-spectrum of $\mu$ for $q\ge 0$, and prove the differentiability of the $L^q$-spectrum. This framework allows us to include graph-directed self-similar measures that are strongly connected and not strongly connected and
those in higher dimension.
\tableofcontents
\end{abstract}
\maketitle

\section{Introduction}\label{S:introduction}
\setcounter{equation}{0}

Let $\mu$ be a bounded positive Borel measure on $\R^d$ ($d\ge1$) with compact support ${\rm supp}(\mu)$.
Define the \textit{$L^q$-spectrum} $\tau(q)$ of $\mu$ by
\begin{equation*}\label{eq:tau_q_on_Rd}
\tau(q):=\varliminf_{h\to 0^+}\dfrac{\ln\sup \sum_i \mu(B_h(x_i))^q}{\ln h}\qquad q\in\R,
\end{equation*}
where the supremum is taken over all such families whose elements are disjoint balls $B_h(x_i)$ with center $x_i\in{\rm supp}(\mu)$ and radius $h$. The function $\tau(q)$ plays a pivotal role in the theory of multifractal decomposition of measures.
The computation of the following \textit{dimension spectrum}:
$$
f(\alpha):=\dim_{\rm H}\left\{x\in{\rm supp}(\mu):\lim_{h\to 0^+}
        \dfrac{\ln\mu(B_{h}(x))}{\ln h} = \alpha\right\},
$$
is a major goal of the theory, where $\dim_{\rm H}$ is the Hausdorff dimension.
The multifractal formalism, originally a heuristic principle suggested by physicists (see \cite{Frisch-Parisi_1985, Halsey-Jensen-Kadanoff-Procaccia-Shraiman_1986} and the references therein), asserts that the dimension spectrum $f(\alpha)$
and the Legendre transform of $\tau(q)$ are equal, i.e.,
\begin{equation}\label{eq:multifractal-formalism}
f(\alpha)=\tau^*(\alpha):=\inf\{q\alpha-\tau(q):\ q\in \mathbb R\}.
\end{equation}

For self-similar measures satisfying the separated open set condition, the multifractal formalism \eqref{eq:multifractal-formalism} has been verified rigorously \cite{Cawley-Mauldin_1992, Edgar-Mauldin_1992}.
For self-similar measures defined by iterated function systems (IFSs) satisfying the weak separation condition, Lau and Ngai \cite{Lau-Ngai_1999} showed that \eqref{eq:multifractal-formalism} at the corresponding point holds, if $\tau^\ast$ is strictly concave. Feng and Lau \cite{Feng-Lau_2009} studied the validity of \eqref{eq:multifractal-formalism} for $q<0$.
For any self-similar measures without any separation conditions, Feng \cite{Feng_2007} proved that \eqref{eq:multifractal-formalism} still holds if $\alpha=\tau'(q)$ for some $q>1$.

The computation of $L^q$-spectrum plays a important role in the theory of multifractal measures.
For self-similar measures satisfying the open set condition (OSC), Cawley and Mauldin \cite{Cawley-Mauldin_1992} computed $\tau(q)$.
An IFS or a graph-directed iterated function system (GIFS), as well as any associated self-similar measure or graph-directed self-similar measure, is said to have \textit{overlaps}, if (OSC) or the graph open set condition (GOSC) (see Section \ref{SS:GIFS}) fails. In this case, the computation of $\tau(q)$ is much more difficult. For the infinite Bernoulli convolution associated with golden ratio and a class of convolutions of Cantor measures \cite{Lau-Ngai_1998, Lau-Ngai_2000}, using Strichartz's second-order self-similar identities, Lau and Ngai computed $\tau(q)$, $q\ge0$. For the infinite Bernoulli convolutions associated with a class of Pisot numbers, Feng \cite{Feng_2005} obtained $\tau(q)$.
For some self-similar measures with overlaps, especially those defined by similitudes with different contraction ratios, Deng and Ngai \cite{Deng-Ngai} proved the differentiability of the $L^q$-spectrum and established the multifractal formalism.
For the class of self-similar measures with overlaps and satisfying the essentially of finite type condition (EFT), a condition introduced in \cite{Ngai-Tang-Xie_2018}, Ngai and the author \cite{Ngai-Xie_2019} derived a closed formula for the $L^q$-spectrum of the measure for $q\ge 0$.

The main motivation of this paper is to extend the results of \cite{Ngai-Xie_2019} to the graph-directed self-similar measures. For graph-directed self-similar measures satisfying (GOSC),
Edgar and Mauldin \cite{Edgar-Mauldin_1992} computed $\tau(q)$. Differ from \cite{Edgar-Mauldin_1992}, the graph-directed self-similar measures we consider do not need to satisfy the (GOSC). For such measures with overlaps and satisfying (EFT), we derive a closed formula for $\tau(q)$, $q\ge0$, and prove the differentiability of $\tau(q)$.

Throughout this paper a GIFS is defined on a compact subset $X\subseteq \R^d$.
Based on the following equivalent definition:
\begin{eqnarray}\label{e:equi_defi_of_Lq_spectrum}
\tau(q)&=&\inf\Big\{\alpha\ge 0:\varlimsup_{h\to 0^+}\frac{1}{h^{d+\alpha}}\int_X\mu(B_h(x))^q\,dx>0\Big\}\notag\\
&=&\sup\{\alpha\ge 0:\varlimsup_{h\to 0^+}\frac{1}{h^{d+\alpha}}\int_X\mu(B_h(x))^q\,dx<\infty\Big\}\qquad\mbox{ for }q\ge 0
\end{eqnarray}
(see \cite{Lau_1992, Lau-Ngai_1998} and \cite[Proposition 3.1]{Lau-Ngai_1999}), we give the derivation of $\tau(q)$ in this paper.

Let $\mu=\sum_{i=1}^{M}\mu_i$ be the graph-directed self-similar measure defined by a GIFS $G=(V, E)$ on $\R^d$, where $V=\{1,\ldots, M\}$ and $E$ are the set of \textit{vertices} and \textit{directed edges} with each edge beginning and ending at a vertex, respectively. The measure $\mu$ is said to satisfies (EFT) (see Definition \ref{defi:EFT}) if there exist a family of bounded bounded open subset $\{\Omega_i\}_{i=1}^{M}$ with $\Omega_i\subseteq\R^d$, ${\rm supp}(\mu_i)\subseteq\overline{\Omega_i}$, and $\mu(\Omega_i)>0$, and a finite family ${\bf B}:=\{B_{1,\ell}: \ell\in\Gamma\}$ of measure disjoint cells, $B_{1,\ell}\subseteq\Omega_{i_\ell}$ for some $i_\ell\in V$, such that for any $\ell\in\Gamma$, there is a family of $\mu$-partitions $\mu$-partitions $\{{\bf P}_{k,\ell}\}_{k\ge1}$ of $B_{1,\ell}$ satisfying the following three conditions:
(1) ${\bf P}_{1,\ell}=\{B_{1,\ell}\}$, and there exists some $B\in {\bf P}'_{2,\ell}$ such that $B\neq B_{1,\ell}$;
(2) for any $k\ge2$, ${\bf P}'_{k+1,\ell}$ encompasses all cells in ${\bf P}'_{k,\ell}$;
(3) the sum of the $\mu$-measures of those cells $B\in {\bf P}^\ast_{k,\ell}$ tends to $0$ as $k\to\infty$,
where ${\bf P}'_{k,\ell}$, $k\ge2$, is the collection of all cells $B\in{\bf P}_{k,\ell}$ that are $\mu$-equivalent to some $B_{1,\ell}$ for $\ell\in\Gamma$, and ${\bf P}^\ast_{k,\ell}$, $k\ge2$, is the collection of all $B\in {\bf P}_{k,\ell}$ that are not $\mu$-equivalent to any cell in ${\bf B}$ (see \eqref{eq:P_1_2}). In this case, we call $\{\Omega_i\}_{i=1}^M$ an \textit{EFT-family}, ${\bf B}$ a \textit{basic family of cells}, and $({\bf B}, {\bf P}):=(\{B_{1,\ell}\},\{{\bf P}_{k,\ell}\}_{k\ge1})_{\ell\in\Gamma}$ a \textit{basic pair}. $({\bf B},{\bf P})$ is said to be \textit{weakly regular} if for any $\ell\in\Gamma$, there exist some similitude $\sigma_\ell$ and some $\Omega_{j_\ell}$ such that $\sigma_\ell(\Omega_{j_\ell})\subseteq B_{1,\ell}$.
Weakly regular in this paper is weaker than regular (see \cite{Ngai-Tang-Xie_2018, Ngai-Xie_2020}).

Let $\mu=\sum_{i=1}^M \mu_i$ be the graph-directed self-similar measure defined by a GIFS $G=(V,E)$ on $\R^d$.
Assume that $G$ has $\eta=\eta(G)$ strongly connected components (SCC), and
$\mu$ satisfies (EFT) with $\{\Omega_i\}_{i=1}^M$ being an EFT-family and assume that there exists a weakly regular basic pair $(\mathbf{B},\mathbf{P}):=(\{B_{1,\ell}\},\{\mathbf{P}_{k,\ell}\}_{k\ge1})_{\ell\in\Gamma}$. Let
\begin{equation}\label{eq:SC_theta}
\mathcal{SC}_m:=\{i\in V: i \mbox{ is contained in the }m\mbox{-th} \mbox{ SCC}\}\qquad\mbox{ for }m=1,\ldots, \eta,
\end{equation}
and
\begin{equation}\label{e:defi_Gamma_i}
\Gamma_i:=\{\ell\in\Gamma: B_{1,\ell}\subseteq\Omega_i\}\qquad\mbox{ for }i\in\mathcal{SC}_m.
\end{equation}
Fix $q\ge 0$, for $m=1,\ldots, \eta$, $i\in \mathcal{SC}_m$ and $\ell\in\Gamma_i$, define
\begin{equation}\label{e:defi_varphi_and_Phi}
\varphi_\ell(h):=\int_{B_{1,\ell}}\mu(B_h(x))^q\,dx,\qquad \Phi_\ell^{(\alpha_m)}(h):=h^{-(d+\alpha_m)}\varphi_\ell(h).
\end{equation}
Based on these assumptions, we can derive renewal equations for $\Phi_\ell^{(\alpha_m)}(h)$, and express them in vector form as:
\begin{equation*}
{\bf f}={\bf f}\ast{\bf M}_{\boldsymbol\alpha}+{\bf z},
\end{equation*}
where
\begin{eqnarray}\label{e:defi_f_M_z}
\begin{aligned}
&\boldsymbol\alpha=(\alpha_1,\ldots, \alpha_\eta), \quad \alpha_m\in\R \mbox{ for } m=1,\ldots,\eta;\\
&{\bf f}={\bf f}^{(\boldsymbol\alpha)}(x)=[f_\ell^{(\alpha_m)}(x)]_{\ell\in\Gamma}, \quad x\in \R;\\
&f_\ell^{(\alpha_m)}(x):=\Phi_\ell^{(\alpha_m)}(e^{-x})\quad \mbox{for }\ell\in\Gamma;\\
&{\bf M}_{\boldsymbol\alpha}=[\mu_{\ell \ell'}^{(\alpha_m)}]_{\ell,\ell'\in\Gamma}\quad\mbox{is a finite matrix of Borel measures on }\R;\\
&{\bf z}={\bf z}^{(\boldsymbol\alpha)}(x)=[z^{(\alpha_m)}_\ell(x)]_{\ell\in\Gamma}\quad\mbox{is a vector of error functions}.
\end{aligned}
\end{eqnarray}
Let
\begin{equation}\label{e:defi_M_alpha_infty}
{\bf M}(\boldsymbol\alpha;\infty):=\big[\mu_{\ell \ell'}^{(\alpha_m)}(\R)\big]_{\ell,\ell'\in\Gamma}.
\end{equation}
For each $\ell\in\Gamma_i$ and $\alpha_m\in\R$, define
\begin{equation}\label{e:defi_F_ell_and_D_ell}
F_\ell(\alpha_m):=\sum_{\ell'\in\Gamma}\mu_{\ell\ell'}^{(\alpha_m)}(\R),
\quad D_\ell:=\{\alpha_m\in\R: F_\ell(\alpha_m)<\infty\}.
\end{equation}
If the error functions decay exponentially to $0$ as $x\to\infty$, then the $L^q$-spectrum of $\mu$ is given by the unique $\alpha$ associated with $\alpha_m$, $m=1,\ldots,\eta$, such that the spectral radii of ${\bf M}(\boldsymbol\alpha;\infty)$ and all of the classes are equal to $1$, where $\boldsymbol\alpha=(\alpha_1,\ldots,\alpha_\eta)$.

Define the \textit{convolution} of a function $a$ with a measure $b$ as
\begin{equation*}
b\ast a(x)=a\ast b(x)=\int_0^t a(x-s) b(ds);
\end{equation*}
if both $a$ and $b$ are measures, we perform convolution of the distribution function of $a$ with the measure $b$.
For two matrices $A, B$ of measures, let $c_{jj'}(x)=\sum_{k} f_{jk}\ast g_{kj'}(x)$ be the \textit{$jj'$-th element}
of $C(x)= A\ast B(x)$.

For $n\ge2$ and $\ell_j\in\Gamma$, $j=1,\ldots, n$, let $\gamma=(\ell_1,\ldots,\ell_n)$ be \textit{a path} (or \textit{$\gamma$-path}) from $\ell_1$ to $\ell_n$. We call a $\gamma$ a \textit{cycle} if $\ell_1=\ell_n$, and a \textit{simple cycle} if it is a cycle and all $\ell_1,\ldots,\ell_{n-1}$ are distinct. For any path $\gamma=(\ell_1,\ldots,\ell_n)$, let $\ell_j\in\Gamma_i$ and $i\in\mathcal{SC}_{m_j}$ for $i\in V$, $j=1,\ldots, n$, and $m_j\in \{1,\ldots,\eta\}$. Define
\begin{equation*}
\mu_\gamma=\mu^{(\alpha_{m_1})}_{\ell_1\ell_2}\ast\mu^{(\alpha_{m_2})}_{\ell_2\ell_3}\ast\ldots\ast\mu^{(\alpha_{m_{n-1}})}_{\ell_{n-1}\ell_n}.
\end{equation*}

For $\ell,\ell'\in\Gamma$, let ${\bf M}^{\boldsymbol\alpha}_{\ell\ell'}$ denote the submatrix of ${\bf M}_{\boldsymbol\alpha}$ obtained by deleting the $\ell$-th row and $\ell'$-th column of ${\bf M}_{\boldsymbol\alpha}$, $\mu^{(\alpha)}_{\widehat{\ell}\ell'}$ be the $\ell'$-th column of the matrix ${\bf M}_{\boldsymbol\alpha}$ with the $\ell$-th element removed, and $\mu^{(\alpha)}_{\ell\widehat{\ell'}}$ for the $\ell$-th row of ${\bf M}_{\boldsymbol\alpha}$ with $\ell'$-th element removed. Denote
\begin{equation*}
\nu_\ell=\mu_{\ell\ell}^{(\alpha)}+\mu_{\ell\widehat{\ell}}^{(\alpha)}
+\mu_{\ell\widehat{\ell}}^{(\alpha)}\ast\sum_{n=0}^\infty({\bf M}^{\boldsymbol\alpha}_{\ell\ell})^{\ast n}\ast\mu_{\widehat{\ell}\ell}^{(\alpha)}\quad\mbox{ for }\ell\in\Gamma.
\end{equation*}
As we known, if ${\bf M}(\boldsymbol\alpha;\infty)$ has maximal eigenvalue $1$ and is irreducible, then
$\nu_\ell$ is a probability measure with support given by $\cup\{\mbox{supp}(\mu_\gamma):\gamma\mbox{ is a simple cycle in }G\}$ (see \cite[Lemma 2.3]{Lau-Wang-Chu_1995}). If $\mbox{supp}(\nu_\ell)$ is contained in a discrete subgroup of $\R$, then $\nu_\ell$ is called \textit{lattice}; otherwise it is called \textit{non-lattice}. The irreducibility implies that $\nu_\ell$ is non-lattice for all $\ell$ if $\nu_1$ is non-lattice.

For $\ell,\ell'\in\Gamma$, we say $\ell$ is \textit{accessible} to $\ell'$ (or $\ell'$ is \textit{accessible} from $\ell$) if there is a path from $\ell$ to $\ell'$. Further we say $\ell$ \textit{communicates} $\ell'$, if they are accessible to each other. We can partition all $\ell\in \Gamma$ into equivalence classes by the communication relation. The spectral radius of a class is the spectral radius of the matrix that results from restricting ${\bf M}(\boldsymbol\alpha;\infty)$ to that class. We say a class is \textit{basic} if its spectral radius is equal to that of ${\bf M}(\boldsymbol\alpha;\infty)$. A class that is not basic is referred to as \textit{non-basic}.
A class $J$ is \textit{final} if $J$ is not accessible to other classes. A \textit{chain} of classes is a collection of classes such that each class is accessible to or from another in the collection. The \textit{length} of a chain is the number of basic classes that it contains. The \textit{height} of a basic class $C$ is defined as the length of the longest chain of classes which have access to $C$. Let $\mathcal{S}_m$ denote the union of basic classes with a height of $m+1$ for $m\ge0$.

Our main results are as follows:

\begin{thm}\label{T:main_thm_str_con}
Let $\mu=\sum_{i=1}^{M}\mu_i$ be a graph-directed self-similar measure defined by a strongly connected GIFS $G=(V, E)$ on $\R^d$. Assume that $\mu$ satisfies (EFT) with $\{\Omega_i\}_{i=1}^{M}$ being an EFT-family and
$({\bf B}, {\bf P}):=(\{B_{1,\ell}\}, \{{\bf P}_{k,\ell}\}_{k\ge1})_{\ell\in\Gamma}$ being a weakly regular basic pair with respect to $\Omega=\bigcup_{i=1}^{M}\Omega_i$. Let ${\bf M}(\boldsymbol\alpha;\infty)$, $F_\ell(\alpha_m)$ be defined as in \eqref{e:defi_M_alpha_infty} and \eqref{e:defi_F_ell_and_D_ell}.
\begin{enumerate}
\item[(a)] There exists a unique $\alpha\in\R$ such that the spectral radius of ${\bf M}(\boldsymbol\alpha; \infty)$ is equal to $1$, where $\boldsymbol\alpha:=(\alpha)$ is composed solely of a single component.
\item[(b)] If we assume, in addition, that for $\ell\in\Gamma$ and the unique $\alpha$ in (a), there exists $\epsilon>0$ such that for all $\ell\in \Gamma$, $z_\ell^{(\alpha)}(x)=o(e^{-\epsilon x})$ as $x\to \infty$, then $\tau(q)=\alpha$ for $q\ge 0$.
\item[(c)]
Let $i\in V$ and $\ell\in \Gamma_i$. If $\nu_1$ is non-lattice, then
\begin{eqnarray*}
\lim_{x\to\infty}e^{(d+\alpha)x}\varphi_\ell(e^{-x})=c_\ell\quad\mbox{ for some constant }c_\ell\ge0;
\end{eqnarray*}
if $\nu_1$ is lattice, then
\begin{eqnarray*}
\lim_{x\to\infty}\big(e^{(d+\alpha)x}\varphi_\ell(e^{-x})-q_\ell(x)\big)=0\quad\mbox{ for some periodic function }q_\ell.
\end{eqnarray*}
\end{enumerate}
\end{thm}

In Section \ref{S:str_con_GIFSs}, we illustrate Theorem \ref{T:main_thm_str_con} by the following Examples on $\R$ and $\R^2$.

\begin{exam}\label{E:exam_str_con_R}
Consider the strongly connected GIFS $G=(V, E)$ with $V={1,2}$ and $E=\{e_i: 1\le i\le 5\}$, where $e_i\in E^{1,1}$ for $i=1,3$, $e_2\in E^{1,2}$, $e_4\in E^{2,2}$, and $e_5\in E^{2,1}$. The five similitudes are given by
\begin{eqnarray}\label{e:exam_str_con_R_similitudes}
\begin{aligned}
&S_{e_i}(x)=\rho x,\quad S_{e_2}(x)=rx+\rho (1-r),\quad S_{e_{i'}}(x)=rx+(1-r)\quad\mbox{ for }i=1,5\mbox{ and }i'=3,4,
\end{aligned}
\end{eqnarray}
with $\rho+2r-\rho r\le 1$, i.e., $S_{e_2}(1)\le S_{e_3}(0)$ (see Figure \ref{F:fig1_str_con_R}).

\begin{center}
\begin{figure}[h]

\begin{picture}(550,50)
\unitlength=0.06cm
\thicklines
\put(15,20){\line(1,0){100}}
\put(14,22){$0$}
\put(113,22){$1$}
\put(63,22){$\Omega_1$}
\put(15,20){\circle*{1.5}}

\put(15,5){\line(1,0){33.33}}
\put(23,8){$S_{e_1}(\Omega_1)$}
\put(15,5){\circle*{1.5}}
\put(38.81,2){\line(1,0){28.57}}
\put(48.5,5){$S_{e_2}(\Omega_2)$}
\put(38.81,2){\circle{1.5}}

\put(86.43,5){\line(1,0){28.57}}
\put(90,8){$S_{e_3}(\Omega_1)$}
\put(86.43,5){\circle*{1.5}}

\put(145,20){\line(1,0){100}}
\put(144,22){$0$}
\put(243,22){$1$}
\put(193,22){$\Omega_2$}
\put(145,20){\circle{1.5}}

\put(145,5){\line(1,0){33.33}}
\put(153,8){$S_{e_5}(\Omega_1)$}
\put(145,5){\circle*{1.5}}

\put(216.43,5){\line(1,0){28.57}}
\put(220,8){$S_{e_4}(\Omega_2)$}
\put(216.43,5){\circle{1.5}}
\end{picture}

\caption{The first iteration of the GIFS defined in \eqref{e:exam_str_con_R_similitudes}, where $\Omega_i=(0,1)$ for $i=1,2$. The figure is illustrated using $\rho=1/3$ and $r=2/7$.}
\label{F:fig1_str_con_R}
\end{figure}
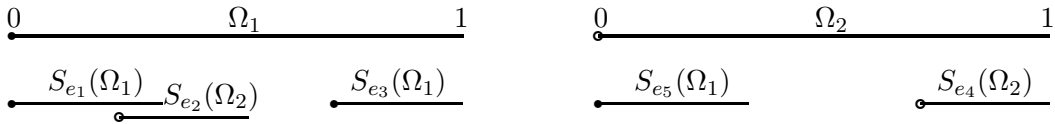
\end{center}
\end{exam}
In \cite[Section 5.1]{Ngai-Xie_2020}, Ngai and the author have proved that the graph-directed self-similar measure $\mu$ defined by a GIFSs $G=(V, E)$ in Example \ref{E:exam_str_con_R} satisfies (EFT) with $\{\Omega_i\}_{i=1}^2=\{(0,1), (0,1)\}$ being an EFT-family and there exists a regular basic pair.

\begin{cor}\label{C:cor_str_con_R}
Let $\mu=\sum_{i=1}^2\mu_i$ be a graph-directed self-similar measure defined by a GIFS $G=(V,E)$ in Example \ref{E:exam_str_con_R} together with a probability matrix $(p_e)_{e\in E}$. Then for $q\ge 0$, there exists a unique real number $\alpha:=\alpha(q)$ satisfying $H(q,\alpha)=0$ with
$$H(q,\alpha):=\overline{Q}_{4,r}\big(\overline{Q}_{1,\rho}\overline{Q}_{3,r}-Q_{1,3,2,5,\rho,r}\big)
-Q_{2,r}Q_{4,r}Q_{5,\rho},$$
where $Q_{1,3,2,5,\rho,r}=Q_{1,3,2,5,\rho,r,q,\alpha}:=(p_{e_1e_3}+p_{e_2e_5})^q(\rho r)^{-\alpha}$ and for $i=1,5$ and $i'=2,3,4$,
\begin{eqnarray*}
\begin{aligned}
&Q_{i,\rho}=Q_{i,\rho,q,\alpha}:=p_{e_i}^q\rho^{-\alpha},\quad&&Q_{i',r}=Q_{i',r,q,\alpha}=p_{e_{i'}}^qr^{-\alpha},\\
&\overline{Q}_{i,\rho}=1-Q_{i,\rho},\quad&&\overline{Q}_{i',r}=1-Q_{i',r}.
\end{aligned}
\end{eqnarray*}
Hence $\tau(q)=\alpha$. In particular, if $H_\alpha(q,\alpha)\neq 0$ for $(q,\alpha)\in(0,\infty)\times\R$, then $\tau$ is differentiable at $q$. Moreover,
\begin{eqnarray}\label{e:result_tau_q_1}
\begin{aligned}
\tau'(q)
=&\biggl\{Q_{4,r}\ln(p_{e_4})\Big(\overline{Q}_{1,\rho}\overline{Q}_{3,r}-Q_{1,3,2,5,\rho,r}\Big)
+Q_{2,r}Q_{4,r}Q_{5,\rho}\ln (p_{e_2e_4e_5})\\
&+\overline{Q}_{4,r}\Big(Q_{1,\rho}\overline{Q}_{3,r}\ln(p_{e_1})
+Q_{3,r}\overline{Q}_{1,\rho}\ln(p_{e_3})+Q_{1,3,2,5,\rho,r}\ln\big(p_{e_1e_3}+p_{e_2e_5}\big)\Big)\biggr\}\\
\cdot&\biggl\{Q_{4,r}\ln(r)
\Big(\overline{Q}_{1,\rho}\overline{Q}_{3,r}-Q_{1,3,2,5,\rho,r}\Big)+Q_{2,r}Q_{4,r}Q_{5,\rho}\ln\big(\rho r^2\big)\\
&+\overline{Q}_{4,r}\Big(Q_{1,\rho}\overline{Q}_{3,r}\ln(\rho)
+Q_{3,r}\overline{Q}_{1,\rho}\ln(r)+Q_{1,3,2,5,\rho,r}\ln(\rho r)\Big)\biggr\}^{-1}.
\end{aligned}
\end{eqnarray}
If $\nu_1$ is non-lattice, then there exists a non-negative constant $c_\ell$ such that
\begin{eqnarray*}
\lim_{x\to\infty}e^{(1+\alpha)x}\varphi_\ell(e^{-x})=c_\ell\quad\mbox{ for }\ell\in\Gamma_i, \ell=1,3,4 \mbox{ and }i=1,2;
\end{eqnarray*}
if $\nu_1$ is lattice, then there exists a periodic function $q_\ell$ such that
\begin{eqnarray*}
\lim_{x\to\infty}\big(e^{(1+\alpha)x}\varphi_\ell(e^{-x})-q_\ell(x)\big)=0\quad\mbox{ for }\ell\in\Gamma_i, \ell=1,3,4 \mbox{ and }i=1,2.
\end{eqnarray*}
\end{cor}

\begin{exam}\label{E:exam_str_con_R2}
Let $G=(V,E)$ be a strongly connected GIFS with $V=1,2$ and $E=\{e_i: i=1,\ldots,8\}$, where $e_i\in E^{1,1}$, $e_4\in E^{1,2}$, $e_j\in E^{2,1}$ and $e_k\in E^{2,2}$ for $i=1,2,3$, $j=5,6$ and $k=7,8$ (see Figure \ref{F:fig1_str_con_R2}). Let
\begin{equation*}
{\bf x}=\left(
                 \begin{array}{c}
                   x \\
                   y \\
                 \end{array}
               \right)
,
\quad \rho=\frac{\sqrt{5}-1}{2},
\quad R(\theta)=\left(
                  \begin{array}{cc}
                    \cos\theta & -\sin\theta \\
                    \sin\theta & \cos\theta \\
                  \end{array}
                \right),
\quad S=\left(
          \begin{array}{cc}
            \rho^2 & 0 \\
            0 & \rho^2 \\
          \end{array}
        \right).
\end{equation*}
The $8$ similitudes are given by
\begin{eqnarray}\label{e:exam_str_con_R2_similitudes}
\begin{aligned}
&S_{e_1}({\bf x})=S{\bf x}+(\rho^3, \rho^3)^T,\qquad&&S_{e_2}({\bf x})=S{\bf x}+(\rho, 0)^T,\\
&S_{e_3}({\bf x})=S{\bf x}+(0, \rho)^T,\qquad&&S_{e_4}({\bf x})=S{\bf x}+(0, 0)^T,\\
&S_{e_5}({\bf x})=R(-\pi/2)S{\bf x}+(0, 1)^T,\qquad&&S_{e_6}({\bf x})=R(\pi/2)S{\bf x}+(1, 0)^T,\\
&S_{e_7}({\bf x})=S{\bf x}+(0, 0)^T,\qquad&&S_{e_8}({\bf x})=S{\bf x}+(\rho, \rho)^T.
\end{aligned}
\end{eqnarray}

\begin{center}
\begin{figure}[h]

\begin{picture}(500,200)
\unitlength=0.3cm

\thicklines
\put(10,13.5){\line(1,0){10}}
\put(10,13.5){\line(0,1){10}}
\put(20,13.5){\line(-1,1){10}}
\put(12,18){$\Omega_1$}


\put(30,13.5){\line(1,0){10}}
\put(30,13.5){\line(0,1){10}}
\put(40,13.5){\line(0,1){10}}
\put(30,23.5){\line(1,0){10}}
\put(34.5,18){$\Omega_2$}

\put(10,1){\line(0,1){3.8}}
\put(10,1){\line(1,0){3.8}}
\put(10,4.8){\line(1,0){3.8}}
\put(13.8,1){\line(0,1){3.8}}

\put(16.2,1){\line(1,0){3.8}}
\put(16.2,1){\line(0,1){3.8}}
\put(20,1){\line(-1,1){3.8}}

\put(10,7.2){\line(0,1){3.8}}
\put(10,11){\line(1,-1){3.8}}
\put(10,7.2){\line(1,0){3.8}}
\multiput(13.8,1)(0.25,0){10}{\line(1,0){0.125}}
\multiput(10,4.8)(0,0.25){10}{\line(0,1){0.125}}
\multiput(16.2,4.8)(-0.25,0.25){10}{\line(0,1){0.125}}

\put(13.8,3.4){\line(1,0){2.4}}
\put(12.4,4.8){\line(0,1){2.4}}
\put(12.4,7.2){\line(1,-1){3.8}}

\put(14.2,5.5){\vector(1,0){2}}
\put(16.5,5){$S_{\e_1}(\Omega_1)$}

\put(10,3.1){\vector(-1,0){1}}
\put(4.5,2.8){$S_{\e_4}(\Omega_2)$}
\put(18,3){\vector(1,0){1}}
\put(19,2.6){$S_{\e_2}(\Omega_1)$}
\put(10,9){\vector(-1,0){1}}
\put(4.5,8.7){$S_{\e_3}(\Omega_1)$}

\put(30,1){\line(1,0){3.8}}
\put(30,1){\line(0,1){3.8}}
\put(33.8,1){\line(0,1){3.8}}
\put(30,4.8){\line(1,0){3.8}}

\put(40,11){\line(0,-1){3.8}}
\put(40,11){\line(-1,0){3.8}}
\put(36.2,7.2){\line(0,1){3.8}}
\put(36.2,7.2){\line(1,0){3.8}}

\put(36.2,1){\line(1,0){3.8}}
\put(36.2,1){\line(1,1){3.8}}
\put(40,1){\line(0,1){3.8}}

\put(30,7.2){\line(0,1){3.8}}
\put(30,11){\line(1,0){3.8}}
\put(30,7.2){\line(1,1){3.8}}

\multiput(33.8,1)(0.25,0){16}{\line(1,0){0.125}}
\multiput(33.8,11)(0.25,0){16}{\line(1,0){0.125}}
\multiput(30,4.8)(0,0.25){16}{\line(0,1){0.125}}
\multiput(40,4.8)(0,0.25){16}{\line(0,1){0.125}}

\put(30,3){\vector(-1,0){1}}
\put(24.5,2.5){$S_{\e_7}(\Omega_2)$}
\put(40,9){\vector(1,0){1}}
\put(41,8.5){$S_{\e_8}(\Omega_2)$}
\put(30,9){\vector(-1,0){1}}
\put(24.5,8.5){$S_{\e_5}(\Omega_1)$}
\put(40,3){\vector(1,0){1}}
\put(41,2.5){$S_{\e_6}(\Omega_1)$}
\end{picture}

\caption{The first iteration of the GIFS defined in \eqref{e:exam_str_con_R2_similitudes} with  $\Omega_1=\bigcup_{x\in(0,1)}(0,1)\times(x,1-x)$ and $\Omega_2=(0,1)\times(0,1)$.}\label{F:fig1_str_con_R2}
\end{figure}
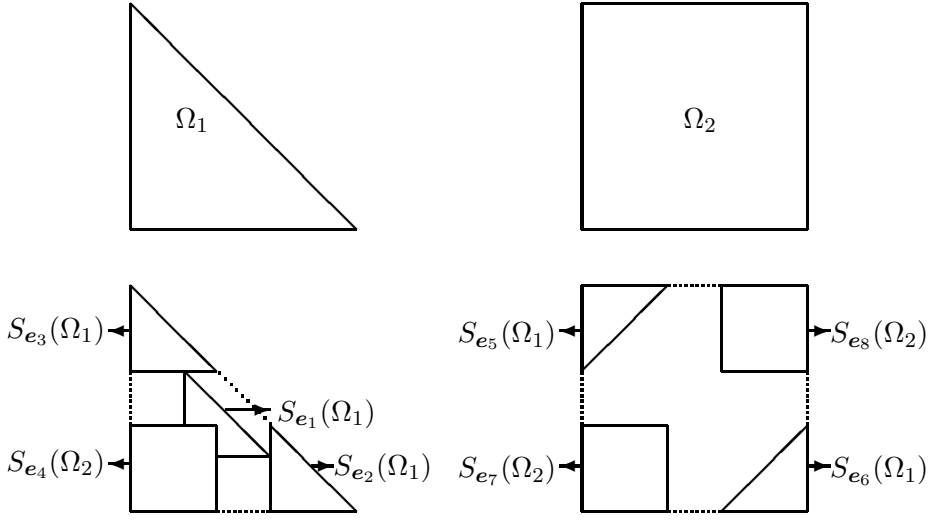
\end{center}
\end{exam}

Ngai and Das \cite{Das-Ngai_2004} computed the Hausdorff dimension of the graph self-similar set defined by a GIFSs $G=(V, E)$ in Example \ref{E:exam_str_con_R2}.

\begin{cor}\label{C:cor_GIFS_str_con_R2}
Let $\mu=\sum_{i=1}^2\mu_i$ be a graph-directed self-similar measure defined by a GIFS $G=(V,E)$ in Example \ref{E:exam_str_con_R2} together with a probability matrix $(p_e)_{e\in E}$. Then for $q\ge 0$, there exists a unique real number $\alpha=\alpha(q)$ satisfying $H(q,\alpha)=0$ with
\begin{eqnarray*}
\begin{aligned}
H(q,\alpha)
:=&\Big[1-\Big(\sum_{i=7}^8 Q_{i,\rho}\Big)\Big]\Big[1-\Big(\sum_{i=1}^3 Q_{i,\rho}\Big)\Big]
-\Big(\prod_{i=1,8}\overline{Q}_{i,\rho}\Big)\Big(\sum_{i=5}^6 Q_{i,\rho}\Big)
\Big(\sum_{k=0}^\infty Q^w_{k,\rho}\Big),
\end{aligned}
\end{eqnarray*}
where $Q^w_{k,\rho}=Q^w_{k,\rho,q,\alpha}:=w^q(k)\rho^{-2\alpha(k+1)}$ and
\begin{eqnarray*}
Q_{i,\rho}=Q_{i,\rho, q,\alpha}:=p^q_{e_i}\rho^{-2\alpha}\quad \mbox{and}\quad
\overline{Q}_{i,\rho}=1-Q_{i,\rho}\quad\mbox{ for }i=1,2,3,5,6,7,8.
\end{eqnarray*}
Hence $\tau(q)=\alpha$. Particularly, if $H_\alpha(q,\alpha)\neq 0$ for $(q,\alpha)\in(0,\infty)\times\R$, then $\tau$ is differentiable at $q$. Moreover,
\begin{eqnarray}\label{e:result_tau_q_2}
\begin{aligned}
&\tau'(q)\\
=&\biggl\{\Big(\sum_{i=7}^8 Q_{i,\rho}\ln\big(p_{e_i}\big)\Big)\Big[1-\Big(\sum_{i=1}^3 Q_{i,\rho}\Big)\Big]
+\Big[1-\Big(\sum_{i=7}^8 Q_{i,\rho}\Big)\Big]\Big(\sum_{i=1}^3 Q_{i,\rho}\ln(p_{e_i})\Big)\\
&-\Big(Q_{1,\rho}\overline{Q}_{8,\rho}\ln(p_{e_1})+Q_{8,\rho}\overline{Q}_{1,\rho}\ln(p_{e_8})\Big)
\Big(\sum_{i=5}^6 Q_{i,\rho}\Big)\Big(\sum_{k=0}^\infty Q^w_{k,\rho}\Big)\\
&+\Big(\prod_{i=1,8}\overline{Q}_{i,\rho}\Big)
\Big[\Big(\sum_{i=5}^6 Q_{i,\rho}\ln(p_{e_i})\Big)\Big(\sum_{k=0}^\infty Q^w_{k,\rho}\Big)
+\Big(\sum_{i=5}^6 Q_{i,\rho}\Big)\Big(\sum_{k=0}^\infty Q^w_{k,\rho}\ln(w(k))\Big)\Big]\biggr\}\\
&\cdot\biggl\{
(k+2)\Big(\prod_{i=1,8}\overline{Q}_{i,\rho}\Big)
\Big(\sum_{i=5}^6 Q_{i,\rho}\Big)\Big(\sum_{k=0}^\infty Q^w_{k,\rho}\Big)+
\Big(\sum_{i=7}^8 Q_{i,\rho}\Big)\Big[1-\Big(\sum_{i=1}^3 Q_{i,\rho}\Big)\Big]\\
&\quad+\Big[1-\Big(\sum_{i=7}^8 Q_{i,\rho}\Big)\Big]\Big(\sum_{i=1}^3 Q_{i,\rho}\Big)
-\Big(Q_{1,\rho}\overline{Q}_{8,\rho}+Q_{8,\rho}\overline{Q}_{1,\rho}\Big)
\Big(\sum_{i=5}^6 Q_{i,\rho}\Big)\Big(\sum_{k=0}^\infty Q^w_{k,\rho}\Big)\biggr\}^{-1}
\cdot \frac{1}{2}\ln(\rho).
\end{aligned}
\end{eqnarray}
If $\nu_1$ is non-lattice, then there exists a non-negative constant $c_\ell$ such that
\begin{eqnarray*}
\lim_{x\to\infty}e^{(2+\alpha)x}\varphi_\ell(e^{-x})=c_\ell\quad\mbox{ for }\ell\in\Gamma_i, \ell=1,2,\ldots,7, \mbox{ and }i=1,2;
\end{eqnarray*}
if $\nu_1$ is lattice, then there exists a periodic function $q_\ell$ such that
\begin{eqnarray*}
\lim_{x\to\infty}\big(e^{(2+\alpha)x}\varphi_\ell(e^{-x})-q_\ell(x)\big)=0\quad\mbox{ for }\ell\in\Gamma_i, \ell=1,2,\ldots,7, \mbox{ and }i=1,2.
\end{eqnarray*}
\end{cor}

\begin{thm}\label{T:main_thm_not_str_con}
Let $\mu=\sum_{i=1}^M \mu_i$ be a graph-directed self-similar measure defined by a GIFS $G=(V, E)$ on $\R^d$, which is not strongly connected. Assume that $G$ has $\eta$ strongly connected components. Also assume that $\mu$ satisfies (EFT) with $\{\Omega_i\}_{i=1}^M$ being an EFT-family and $({\bf B}, {\bf P}):=(\{B_{1,\ell}\}, \{{\bf P}_{k,\ell}\}_{k\ge1})_{\ell\in\Gamma}$ being a weakly regular basic pair with respect to $\Omega=\bigcup_{i=1}^M \Omega_i$. Let ${\bf M}(\boldsymbol\alpha;\infty)$, $F_\ell(\alpha_m)$ be defined as in \eqref{e:defi_M_alpha_infty} and \eqref{e:defi_F_ell_and_D_ell}.
\begin{enumerate}
\item[(a)] There exists a unique set of real numbers $\alpha_1,\ldots,\alpha_\eta$ such that the spectral radius of ${\bf M}(\boldsymbol\alpha; \infty)$ and all the other classes equal $1$, where $\boldsymbol\alpha:=(\alpha_1,\ldots, \alpha_\eta)$.
\item[(b)] If we assume, in addition, that for $m=1,\ldots,\eta$, $i\in\mathcal{SC}_m$, $\ell\in\Gamma_i$ and the unique set $\{\alpha_1,\ldots,\alpha_\eta\}$ in (a), there exists $\epsilon>0$ such that for all $\ell\in \Gamma_i$, $z_\ell^{(\alpha)}(x)=o(e^{-\epsilon x})$ as $x\to \infty$, then $\tau(q)=\alpha$ for $q\ge 0$, where $\alpha:=\min\{\alpha_1,\ldots,\alpha_\eta\}$.
\item[(c)] Let $i\in V$ and $\ell\in\Gamma_i$.
  \begin{enumerate}
  \item[(1)] If $\ell\in\mathcal{S}_0$, then
  \begin{eqnarray*}
  \lim_{x\to\infty}\big(e^{(d+\alpha)x}\varphi_\ell(e^{-x})-q_\ell(x)\big)=0,
  \end{eqnarray*}
  where $q_\ell$ is either periodic or non-negative constant depending on whether $\nu_\ell$ is lattice or not.
  \item[(2)] If $m>0$ and $\ell\in\mathcal{S}_m$, then
  \begin{eqnarray*}
  \lim_{x\to\infty} x^{-m} e^{(d+\alpha)x}\varphi_\ell(e^{-x})=c_\ell\quad\mbox{ for some constant }c_\ell\ge0.
  \end{eqnarray*}
  \item[(3)] If $\ell\notin \mathcal{S}=\bigcup_{m\ge0}\mathcal{S}_m$ and there is no path from $\mathcal{S}$ to $\ell$, then
  \begin{eqnarray*}
  \lim_{x\to\infty} e^{(d+\alpha)x}\varphi_\ell(e^{-x})=0.
  \end{eqnarray*}
  \item[(4)] If $\ell\notin \mathcal{S}$ and there is a path from $\mathcal{S}_0$ to $\ell$, but no path to $\ell$ from $\mathcal{S}_k$ for any $k>0$, then
  \begin{eqnarray*}
  \lim_{x\to\infty}\big(e^{(d+\alpha)x}\varphi_\ell(e^{-x})-\widetilde{q}_\ell(x)\big)=0,
  \end{eqnarray*}
  for some $\widetilde{q}_\ell$ which is either non-negative constant or periodic.
  \item[(5)] If $\ell\notin \mathcal{S}$ and there is a path from $\mathcal{S}_m$ to $\ell$, but no path from $\mathcal{S}_k$, for any $k>m>0$, then
      \begin{eqnarray*}
      \lim_{x\to\infty} x^{-m}e^{(d+\alpha)x}\varphi_\ell(e^{-x})=\widetilde{c}_\ell\quad\mbox{ for some constant }\widetilde{c}_\ell\ge0.
      \end{eqnarray*}
  \end{enumerate}
\end{enumerate}
\end{thm}

In Section \ref{S:not_str_conn_R}, we illustrate Theorem \ref{T:main_thm_not_str_con} by the following Examples on $\R$ and $\R^2$.

\begin{exam}\label{E:exam_not_str_con_R_1}
Let $G=(V,E)$ be a GIFS that is not strongly connected with $V=\{1,2\}$ and $E=\{e_i: 1\le i\le 5\}$,
where $e_i\in E^{1,1}$ for $1,2,3$, $e_4\in E^{2,2}$, and $e_5\in E^{2,1}$. The five similitudes are defined by
\begin{equation}\label{E:exam_str_con_R_1_similitudes}
S_{e_i}(x)=\rho x,\quad S_{e_2}(x)=rx+\rho(1-r),\quad S_{e_{i'}}(x)=rx+(1-r)\mbox{  for }i=1,5\mbox{ and }i'=3,4,
\end{equation}
with $\rho+2r-\rho r\le 1$, i.e., $S_{e_2}(1)\le S_{e_3}(0)$ (see Figure~\ref{F:fig1_not_str_con_R_1}).
\begin{center}
\begin{figure}[h]

\begin{picture}(424,60)
\unitlength=0.07cm
\thicklines
\put(0,19){\line(1,0){100}}
\put(0,21){$0$}
\put(100,21){$1$}
\put(48,21){$\Omega_1$}
\put(0,19){\circle*{1.5}}

\put(10,10){$S_{e_1}(\Omega_1)$}
\put(0,8){\line(1,0){33.3}}
\put(0,8){\circle*{1.5}}

\put(34,7){$S_{e_2}(\Omega_1)$}
\put(23.81,5){\line(1,0){28.57}}
\put(23.81,5){\circle*{1.5}}

\put(80,10){$S_{e_3}(\Omega_1)$}
\put(71.43,8){\line(1,0){28.57}}
\put(71.43,8){\circle*{1.5}}

\put(110,19){\line(111,0){100}}
\put(110,21){$0$}
\put(210,21){$1$}
\put(158,21){$\Omega_2$}
\put(110,19){\circle{1.5}}

\put(120,10){$S_{e_5}(\Omega_1)$}
\put(110,8){\line(111,0){33.3}}
\put(110,8){\circle*{1.5}}

\put(190,7){$S_{e_4}(\Omega_2)$}
\put(181.43,5){\line(111,0){28.57}}
\put(181.43,5){\circle{1.5}}
\end{picture}

\caption{First iteration of the GIFS defined in \eqref{E:exam_str_con_R_1_similitudes}, where $\Omega_i=(0,1)$ for $i=1,2$. The figure is illustrated using $\rho=1/3$ and $r=2/7$.}
\label{F:fig1_not_str_con_R_1}
\end{figure}
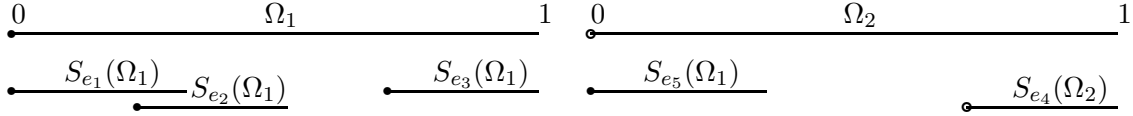
\end{center}
\end{exam}

\begin{cor}\label{C:cor_GIFS_not_str_con_R_1}
Let $\mu=\sum_{i=1}^2\mu_i$ be a graph-directed self-similar measure defined by the GIFS $G=(V, E)$ in Example \ref{E:exam_not_str_con_R_1} together with a probability matrix $(p_e)_{e\in E}$, and let $w_k$ be defined as in \eqref{eq:defi_w_k_GIFS_not_str_con_R_1}. Then there exist two real numbers $\alpha_i$, $i=1,2$, satisfying $H(q,\alpha_i)=0$ with
\begin{eqnarray*}
\begin{aligned}
H(q,\alpha)
:=\Big[\Big(\prod_{i=2}^3 \overline{Q}_{i,r}\Big)
\Big(1-\Big(\sum_{k=0}^\infty Q^w_{k,\rho,r}\Big)\Big)-\Big(\prod_{i=2}^3 Q_{i,r}\Big)\Big]\cdot\overline{Q}_{4,r},
\end{aligned}
\end{eqnarray*}
where $Q^w_{k,\rho,r}=Q^w_{k,\rho,r,q,\alpha}:=w_k^q (\rho r^k)^{-\alpha}$ and
\begin{eqnarray*}
Q_{i,r}=Q_{i,r,q,\alpha}:=p^q_{e_i} r^{-\alpha}\quad\mbox{and}\quad \overline{Q}_{i,r}=1-Q_{i,r}\quad\mbox{ for }i=2,3,4.
\end{eqnarray*}
Hence $\tau(q)=\min\{\alpha_1,\alpha_2\}$. In particular, if $q>0$ and $H_\alpha(q,\alpha)\neq 0$, then $\tau$ is differentiable at $q$. Moreover,
\begin{eqnarray}\label{e:result_tau_q_3}
\begin{aligned}
\tau'(q)
=&\biggl\{\Big[\Big(\prod_{i=2}^3 \overline{Q}_{i,r}\Big)
\Big(1-\Big(\sum_{k=0}^\infty Q^w_{k,\rho,r}\Big)\Big)-Q_{2,r}Q_{3,r}\Big]\cdot Q_{4,r}\ln(p_{e_4})\\
&+\Big[\Big(Q_{2,r}\overline{Q}_{3,r}\ln(p_{e_2})+Q_{3,r}\overline{Q}_{2,r}\ln(p_{e_3})\Big)
\Big(1-\Big(\sum_{k=0}^\infty Q^w_{k,\rho,r}\Big)\Big)\\
&~\quad+\Big(\prod_{i=2}^3\overline{Q}_{i,r}\Big)\Big(\sum_{k=0}^\infty Q^w_{k,\rho,r}\ln(w_k)\Big)
+Q_{2,r}Q_{3,r}\ln(p_{e_2e_3})\Big]\cdot\overline{Q}_{4,r}\biggr\}\\
&\cdot\biggl\{\Big[\Big(\prod_{i=2}^3\overline{Q}_{i,r}\Big)
\Big(1-\Big(\sum_{k=0}^\infty Q^w_{k,\rho,r}\Big)\Big)-Q_{2,r}Q_{3,r}\Big]\cdot Q_{4,r}\ln(r)\\
&~~+\Big[\Big(Q_{2,r}\overline{Q}_{3,r}\ln(r)+Q_{3,r}\overline{Q}_{2,r}\ln(r)\Big)
\Big(1-\Big(\sum_{k=0}^\infty Q^w_{k,\rho,r}\Big)\Big)\\
&~~~~\quad+\Big(\prod_{i=2}^3 \overline{Q}_{i,r}\Big)\Big(\sum_{k=0}^\infty w_k^q (\rho r^k)^{-\alpha}\ln (\rho r^k)\Big)
+2Q_{2,r}Q_{3,r}\ln(r)\Big]\cdot\overline{Q}_{4,r}\biggr\}^{-1}
\end{aligned}
\end{eqnarray}
If $\ell=1,2$, then
\begin{eqnarray*}
  \lim_{x\to\infty}\big(e^{(1+\alpha)x}\varphi_\ell(e^{-x})-q_\ell(x)\big)=0,
  \end{eqnarray*}
where $q_\ell$ is either periodic or non-negative constant depending on whether $\nu_\ell$ is lattice or not;
if $\ell=3$ or $4$, then
\begin{eqnarray*}
  \lim_{x\to\infty} e^{(1+\alpha)x}\varphi_\ell(e^{-x})=0.
  \end{eqnarray*}
\end{cor}

\begin{exam}\label{E:exam_GIFS_not_str_con_R_2}
Let $G=(V,E)$ be a GIFS that is not strongly connected with $V=\{1,\ldots,6\}$ and $E=\{e_i: 1\le i\le 17\}$, where $e_{i_1}\in E^{1,1}$, $e_4\in E^{2,1}$, $e_{i_2}\in E^{2,2}$, $e_{i_3}\in E^{3,3}$, $e_{10}\in E^{4,3}$, $e_{i_4}\in E^{4,4}$, $e_{13}\in E^{5,3}$, $e_{i_5}\in E^{5,5}$, $e_{16}\in E^{6,1}$, and $e_{17}\in E^{6,6}$ for $i_1=1,2,3$, $i_2=5,6$, $i_3=7,8,9$, $i_4=11,12$ and $i_5=14, 15$. The 17 similitudes are defined by
\begin{equation}\label{E:exam_str_con_R_2_similitudes}
\begin{aligned}
&S_{e_i}(x)=\rho x,\qquad\mbox{ for }i=1,4,7,10,13,16,\\
&S_{e_j}(x)=rx+\rho(1-r),\qquad\mbox{ for }j=2,5,8,11,14,\\
&S_{e_k}(x)=rx+(1-r),\qquad\mbox{ for }k=3,6,9,12,15,17,
\end{aligned}
\end{equation}
where $\rho+2r-\rho r\le 1$, i.e., for $j=2,5,8,11,14$, $S_{e_j}(1)\le S_{e_{j+1}}(0)$ (see Figure~\ref{F:fig1_not_str_con_R_2}). For a probability matrix $(p_e)_{e\in E}$ and $k\ge 0$, we define
\begin{equation}\label{defi:w_i(k)}
\begin{aligned}
&w_1(k):=p_{e_1}\sum_{j=0}^k p_{e_2}^{j}p_{e_3}^{k-j},\quad&&w_2(k):=p_{e_4}\sum_{j=0}^k p_{e_5}^{j}p_{e_3}^{k-j},
\quad&&w_3(k):=p_{e_7}\sum_{j=0}^k p_{e_8}^{j}p_{e_9}^{k-j},\\
&w_4(k):=p_{e_{10}}\sum_{j=0}^k p_{e_{11}}^{j}p_{e_9}^{k-j},\quad&&w_5(k):=p_{e_{13}}\sum_{j=0}^k p_{e_{14}}^{j}p_{e_9}^{k-j}.
\end{aligned}
\end{equation}

\begin{center}
\begin{figure}[h]

\begin{picture}(424,190)
\unitlength=0.07cm
\thicklines
\put(0,95){\line(1,0){100}}
\put(0,97){$0$}
\put(100,97){$1$}
\put(48,97){$\Omega_1$}

\put(10,86){$S_{e_1}(\Omega_1)$}
\put(0,84){\line(1,0){33.3}}

\put(34,83){$S_{e_2}(\Omega_1)$}
\put(23.81,81){\line(1,0){28.57}}

\put(80,86){$S_{e_3}(\Omega_1)$}
\put(71.43,84){\line(1,0){28.57}}

\put(110,95){\line(111,0){100}}
\put(110,97){$0$}
\put(210,97){$1$}
\put(158,97){$\Omega_2$}

\put(120,86){$S_{e_4}(\Omega_1)$}
\put(110,84){\line(111,0){33.3}}

\put(144,83){$S_{e_5}(\Omega_2)$}
\put(133.81,81){\line(111,0){28.57}}

\put(190,83){$S_{e_6}(\Omega_2)$}
\put(181.43,81){\line(111,0){28.57}}

\put(0,58){\line(1,0){100}}
\put(0,60){$0$}
\put(100,60){$1$}
\put(48,60){$\Omega_3$}

\put(10,49){$S_{e_7}(\Omega_3)$}
\put(0,47){\line(1,0){33.3}}

\put(34,46){$S_{e_8}(\Omega_3)$}
\put(23.81,44){\line(1,0){28.57}}

\put(80,49){$S_{e_9}(\Omega_3)$}
\put(71.43,47){\line(1,0){28.57}}

\put(110,58){\line(111,0){100}}
\put(110,60){$0$}
\put(210,60){$1$}
\put(158,60){$\Omega_4$}

\put(120,49){$S_{e_{10}}(\Omega_3)$}
\put(110,47){\line(111,0){33.3}}

\put(144,46){$S_{e_{11}}(\Omega_4)$}
\put(133.81,44){\line(111,0){28.57}}

\put(190,46){$S_{e_{12}}(\Omega_4)$}
\put(181.43,44){\line(111,0){28.57}}

\put(0,21){\line(1,0){100}}
\put(0,23){$0$}
\put(100,23){$1$}
\put(48,23){$\Omega_5$}

\put(10,12){$S_{e_{13}}(\Omega_3)$}
\put(0,10){\line(1,0){33.3}}

\put(34,9){$S_{e_{14}}(\Omega_5)$}
\put(23.81,7){\line(1,0){28.57}}

\put(80,12){$S_{e_{15}}(\Omega_5)$}
\put(71.43,10){\line(1,0){28.57}}

\put(110,21){\line(111,0){100}}
\put(110,23){$0$}
\put(210,23){$1$}
\put(158,23){$\Omega_6$}

\put(120,12){$S_{e_{16}}(\Omega_1)$}
\put(110,10){\line(111,0){33.3}}

\put(190,9){$S_{e_{17}}(\Omega_6)$}
\put(181.43,7){\line(111,0){28.57}}
\end{picture}

\caption{First iteration of the GIFS defined in \eqref{E:exam_str_con_R_2_similitudes}, where $\Omega_i=(0,1)$ for $i=1,\ldots, 6$. The figure is drawn with $\rho=1/3$ and $r=2/7$.}
\label{F:fig1_not_str_con_R_2}
\end{figure}
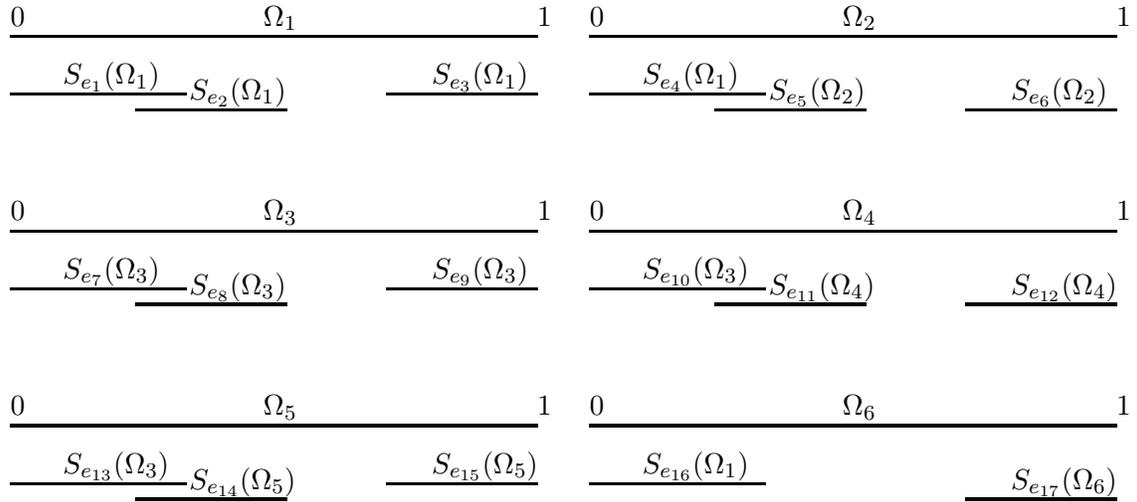
\end{center}
\end{exam}

\begin{cor}\label{C:cor_GIFS_not_str_con_R_2}
Let $\mu=\sum_{i=1}^6\mu_i$ be the graph-directed self-similar measure defined by the GIFS $G=(V, E)$ in Example \ref{E:exam_GIFS_not_str_con_R_2} together with a probability matrix $(p_e)_{e\in E}$. Then there exists $\alpha_i$, $i=1,2,\ldots,6$, satisfying $H(q,\alpha)=0$ with
\begin{eqnarray*}
\begin{aligned}
H(q,\alpha)
=&\Big[1-\Big(\sum_{i=2}^3 Q_{i,r}\Big)-\Big(\prod_{i=2}^3 \overline{Q}_{i,r}\Big)
\Big(\sum_{k=0}^\infty Q^{w_1}_{k,\rho,r}\Big)\Big]\Big[1-\Big(\sum_{i=5}^6 Q_{i,r}\Big)\Big]\\
&\Big[(1-\Big(\sum_{i=8}^9 Q_{i,r}\Big)-\Big(\prod_{i=8}^9 \overline{Q}_{i,r}\Big)
\Big(\sum_{k=0}^\infty Q^{w_3}_{k,\rho,r}\Big)\Big]\Big[1-\Big(\sum_{i=11}^{12} Q_{i,r}\Big)\Big]
\Big[1-\Big(\sum_{i=14}^{15} Q_{i,r}\Big)\Big] Q_{17,r},
\end{aligned}
\end{eqnarray*}
where
\begin{eqnarray*}
\begin{aligned}
&Q^{w_i}_{k,\rho,r}=Q^{w_i}_{k,\rho,r,q,\alpha}:=w_i^q(k) (\rho r^k)^{-\alpha}\quad\mbox{ for }i=1,3,\\
&Q_{i,r}=Q_{i,r,q,\alpha}:=p^q_{e_i} r^{-\alpha}\quad\mbox{and}\quad \overline{Q}_{i,r}=1-Q_{i,r}\quad\mbox{ for }i=2,3,5,6,8,9,11,12,14,15,17.
\end{aligned}
\end{eqnarray*}
Hence $\tau(q)=\min\{\alpha_i: i=1,\ldots, 6\}$.
In particular, if $q>0$ and $H_\alpha(q,\alpha)\neq 0$, then $\tau$ is differentiable at $q$. Moreover,
$\tau'=-H_q(q,\alpha)\cdot H^{-1}_\alpha(q,\alpha)$,
with
\begin{eqnarray}\label{e:result_tau_q_4_q}
\begin{aligned}
&H_q(q,\alpha)\\
=&\biggl\{\bigg[\Big(\prod_{i=2}^3 \overline{Q}_{i,r}\Big)\Big(\sum_{k=0}^\infty Q^{w_1}_{k,\rho,r}\ln\big(w_1(k)\big)\Big)
+\Big(Q_{2,r}\overline{Q}_{3,r}\ln(p_{e_2})+Q_{3,r}\overline{Q}_{2,r}\ln(p_{e_3})\Big)
\Big(\sum_{k=0}^\infty Q^{w_1}_{k,\rho,r}\Big)\\
&\quad-\Big(\sum_{i=2}^3 Q_{i,r} \ln(p_{e_i})\Big)\bigg]
\bigg[1-\Big(\sum_{i=5}^6 Q_{i,r}\Big)\bigg]
\bigg[1-\Big(\sum_{i=8}^9 Q_{i,r}\Big)-\Big(\prod_{i=8}^9 \overline{Q}_{i,r}\Big)
\Big(\sum_{k=0}^\infty Q^{w_3}_{k,\rho,r}\Big)\bigg]\\
&\quad\cdot\bigg[1-\Big(\sum_{i=11}^{12} Q_{i,r}\Big)\bigg]\bigg[1-\Big(\sum_{i=14}^{15} Q_{i,r}\Big)\bigg]\overline{Q}_{17,r}\biggr\}\\
&-\biggl\{\bigg[1-\Big(\sum_{i=2}^3 Q_{i,r}\Big)-\Big(\prod_{i=2}^3 \overline{Q}_{i,r}\Big)
\Big(\sum_{k=0}^\infty Q^{w_1}_{k,\rho,r}\Big)\bigg]\bigg[\sum_{i=5}^6 Q_{i,r}\ln(p_{e_i})\bigg]\\
&\quad\cdot\bigg[1-\Big(\sum_{i=8}^9 Q_{i,r}\Big)-\Big(\prod_{i=8}^9 \overline{Q}_{i,r}\Big)
\Big(\sum_{k=0}^\infty Q^{w_3}_{k,\rho,r}\Big)\bigg]\bigg[1-\Big(\sum_{i=11}^{12} Q_{i,r}\Big)\bigg]
\bigg[1-\Big(\sum_{i=14}^{15} Q_{i,r}\Big)\bigg]\overline{Q}_{17,r}\biggr\}\\
&-\biggl\{\bigg[1-\Big(\sum_{i=2}^3 Q_{i,r}\Big)-\Big(\prod_{i=2}^3 \overline{Q}_{i,r}\Big)
\Big(\sum_{k=0}^\infty Q^{w_1}_{k,\rho,r}\Big)\bigg]\bigg[1-\Big(\sum_{i=5}^6 Q_{i,r}\Big)\bigg]\\
&\quad\cdot\bigg[\Big(\sum_{i=8}^9 Q_{i,r}\ln(p_{e_j})\Big)-\Big(\prod_{i=8}^9 \overline{Q}_{i,r}\Big)\Big(\sum_{k=0}^\infty Q^{w_3}_{k,\rho,r}\ln(w_3(k))\Big)\\
&~\quad-\Big(Q_{8,r}\overline{Q}_{9,r}\ln(p_{e_8})+Q_{9,r}\overline{Q}_{8,r}\ln(p_{e_9})\Big)
\Big(\sum_{k=0}^\infty Q^{w_3}_{k,\rho,r}\Big)\bigg]\\
&\quad\cdot\bigg[1-\Big(\sum_{i=11}^{12} Q_{i,r}\Big)\bigg]\bigg[1-\Big(\sum_{i=14}^{15} Q_{i,r}\Big)\bigg]\overline{Q}_{17,r}\biggr\}\\
&-\biggl\{\bigg[1-\Big(\sum_{i=2}^3 Q_{i,r}\Big)
-\Big(\prod_{i=2}^3 \overline{Q}_{i,r}\Big)\Big(\sum_{k=0}^\infty Q^{w_1}_{k,\rho,r}\Big)\bigg]
\bigg[1-\Big(\sum_{i=5}^6 Q_{i,r}\Big)\bigg]\\
&\quad\cdot\bigg[1-\Big(\sum_{i=8}^9 Q_{i,r}\Big)-\Big(\prod_{i=8}^9 \overline{Q}_{i,r}\Big)
\Big(\sum_{k=0}^\infty Q^{w_3}_{k,\rho,r}\Big)\bigg]\\
&\quad\cdot\bigg[\Big(\sum_{i=11}^{12} Q_{i,r}\ln(p_{e_i})\Big)
\Big(1-\Big(\sum_{i=14}^{15} Q_{i,r}\Big)\Big)\overline{Q}_{17,r}\\
&\quad+\Big(1-\Big(\sum_{i=11}^{12} Q_{i,r}\Big)\Big)
\cdot\Big(\Big(\sum_{i=14}^{15} Q_{i,r}\ln(p_{e_i})\Big)\overline{Q}_{17,r}
+\Big(1-\Big(\sum_{i=14}^{15} Q_{i,r}\Big)\Big) Q_{17,r}\Big)\bigg]\biggr\},
\end{aligned}
\end{eqnarray}
and
\begin{eqnarray}\label{e:result_tau_q_4_alpha}
\begin{aligned}
&H_\alpha(q,\alpha)\\
=&\biggl\{\bigg[\Big(\prod_{i=2}^3 \overline{Q}_{i,r}\Big)\Big(\sum_{k=0}^\infty Q^{w_1}_{k,\rho,r}\ln(\rho r^k)\Big)-\Big(Q_{2,r}\ln(r)\overline{Q}_{3,r}+Q_{3,r}\ln(r)\overline{Q}_{2,r}\Big)
\Big(\sum_{k=0}^\infty Q^{w_1}_{k,\rho,r}\Big)\\
&\quad+\Big(\sum_{i=2}^3 Q_{i,r}\ln(r)\Big)\bigg]\bigg[1-\Big(\sum_{i=5}^6 Q_{i,r}\Big)\bigg]
\bigg[1-\Big(\sum_{i=8}^9 Q_{i,r}\Big)-\Big(\prod_{i=8}^9 \overline{Q}_{i,r}\Big)
\Big(\sum_{k=0}^\infty Q^{w_3}_{k,\rho,r}\Big)\bigg]\\
&\quad\cdot\bigg[1-\Big(\sum_{i=11}^{12} Q_{i,r}\Big)\bigg]
\bigg[1-\Big(\sum_{i=14}^{15} Q_{i,r}\Big)\bigg]\overline{Q}_{17,r}\biggr\}\\
&+\bigg\{\bigg[1-\Big(\sum_{i=2}^3 Q_{i,r}\Big)-\Big(\prod_{i=2}^3\overline{Q}_{i,r}\Big)
\Big(\sum_{k=0}^\infty Q^{w_1}_{k,\rho,r}\Big)\bigg]
\bigg[\sum_{i=5}^6 Q_{i,r}\ln(r)\bigg]\\
&\quad\cdot\bigg[1-\Big(\sum_{i=8}^9 Q_{i,r}\Big)-\Big(\prod_{i=8}^9 \overline{Q}_{i,r}\Big)
\Big(\sum_{k=0}^\infty Q^{w_3}_{k,\rho,r}\Big)\bigg]
\bigg[1-\Big(\sum_{i=11}^{12} Q_{i,r}\Big)\bigg]\bigg[1-\Big(\sum_{i=14}^{15}Q_{i,r}\Big)\bigg]\overline{Q}_{17,r}\biggr\}\\
&+\biggl\{\bigg[1-\Big(\sum_{i=2}^3  Q_{i,r}\Big)-\Big(\prod_{i=2}^3 \overline{Q}_{i,r}\Big)
\Big(\sum_{k=0}^\infty Q^{w_1}_{k,\rho,r}\Big)\bigg]\bigg[1-\Big(\sum_{i=5}^6 Q_{i,r}\Big)\bigg]\\
&\quad\cdot\bigg[\Big(\sum_{i=8}^9  Q_{i,r}\ln(r)\Big)
-\ln(r)\Big( Q_{8,r}\overline{Q}_{9,r}+ Q_{9,r}\overline{Q}_{8,r}\Big)
\Big(\sum_{k=0}^\infty Q^{w_3}_{k,\rho,r}\Big)\\
&~~\quad+\Big(\prod_{i=8}^9 \overline{Q}_{i,r}\Big)\Big(\sum_{k=0}^\infty Q^{w_3}_{k,\rho,r}\ln(\rho r^k)\Big)\bigg]
\bigg[1-\Big(\sum_{i=11}^{12}Q_{i,r}\Big)\bigg]
\bigg[1-\Big(\sum_{i=14}^{15}Q_{i,r}\Big)\bigg]\overline{Q}_{17,r}\biggr\}\\
&+\biggl\{\bigg[1-\Big(\sum_{i=2}^3 Q_{i,r}\Big)-\Big(\prod_{i=2}^3 \overline{Q}_{i,r}\Big)
\Big(\sum_{k=0}^\infty Q^{w_1}_{k,\rho,r}\Big)\bigg]
\bigg[1-\Big(\sum_{i=5}^6 Q_{i,r}\Big)\bigg]\\
&\quad\cdot\bigg[1-\Big(\sum_{i=8}^9 Q_{i,r}\Big)-\Big(\prod_{i=8}^9 \overline{Q}_{i,r}\Big)
\Big(\sum_{k=0}^\infty Q^{w_3}_{k,\rho,r}\Big)\bigg]
\bigg[\Big(\sum_{i=11}^{12} Q_{i,r}\ln(r)\Big)\Big(1-\Big(\sum_{i=14}^{15} Q_{i,r}\Big)\Big) \overline{Q}_{17,r}\\
&~~\quad+\Big(1-\Big(\sum_{i=11}^{12} Q_{i,r}\Big)\Big)
\bigg(\Big(\sum_{i=14}^{15} Q_{i,r}\ln(r)\Big) Q_{17,r}
+\Big(1-\Big(\sum_{i=14}^{15} Q_{i,r}\Big)\Big)\Big(Q_{17,r}\ln(r)\Big)\bigg)\bigg]\biggr\}.
\end{aligned}
\end{eqnarray}
If $\ell=1,2$, then
\begin{eqnarray*}
  \lim_{x\to\infty} x^{-1} e^{(1+\alpha)x}\varphi_\ell(e^{-x})=c_\ell\quad\mbox{ for some constant }c_\ell\ge0;
  \end{eqnarray*}
if $\ell=5,6$, there are constants $c_\ell>0$ such that
\begin{eqnarray*}
  \lim_{x\to\infty} x^{-2} e^{(1+\alpha)x}\varphi_\ell(e^{-x})=c_\ell;
\end{eqnarray*}
if $\ell=3,4,7,8,9,10,11,12$, then
\begin{eqnarray*}
  \lim_{x\to\infty} e^{(1+\alpha)x}\varphi_\ell(e^{-x})=0.
  \end{eqnarray*}
\end{cor}

\begin{exam}\label{E:exam_GIFS_not_str_con_R2}
Let $G=(V,E)$ be a GIFS that is not strongly connected with $V=\{1,2\}$ and $E=\{e_i: 1\le i\le 7\}$, where
\begin{eqnarray*}
e_1\in E^{1,2},\quad e_i\in E^{1,1},\quad e_j\in E^{2,2}\quad\mbox{ for }i=2,3\mbox{ and }j=4,5,6,7.
\emph{}\end{eqnarray*}
The $7$ similitudes are defined by
\begin{eqnarray}\label{E:exam_str_con_R2_similitudes}
\begin{aligned}
&S_{e_1}{\bf x}=s{\bf x}+(-2s,0),\quad&& S_{e_2}{\bf x}=t{\bf x}+(1-t,0),\\
&S_{e_3}{\bf x}=(1-t){\bf x}+(0,t),\quad&& S_{e_4}{\bf x}=\rho{\bf x}+(2(1-\rho),0),\\
&S_{e_5}{\bf x}=r{\bf x}+((2+\rho)(r-1),0),\quad&& S_{e_6}{\bf x}=r{\bf x}+(3(1-r),0),\\
&S_{e_7}{\bf x}=r{\bf x}+(2(1-r),1-r),
\end{aligned}
\end{eqnarray}
where $t\in(0,1)$, $s\in \big(0, \min\{t, 1-t\}\big)$ and $\rho+2r-\rho r\le1$ (see Figure \ref{F:fig1_not_str_con_R2}).
\end{exam}

\begin{center}
\begin{figure}[h]

\begin{picture}(280,200)
\unitlength=0.30cm

\thicklines

\put(0,13.5){\line(1,0){10}}
\put(0,13.5){\line(0,1){10}}
\put(10,13.5){\line(-1,1){10}}
\put(2,16.5){$\Omega_1$}

\put(20,13.5){\line(1,0){10}}
\put(20,13.5){\line(0,1){10}}
\put(30,13.5){\line(0,1){10}}
\put(20,23.5){\line(1,0){10}}
\put(24.5,18.5){$\Omega_2$}


\put(0,6.5){\line(0,1){4}}
\put(4,0.5){\line(1,0){6}}
\put(4,0.5){\line(0,1){6}}
\put(10,0.5){\line(-1,1){10}}

\put(0,6.5){\line(1,0){4}}
\put(0,0.5){\line(0,1){3.3}}
\put(0,0.5){\line(1,0){3.3}}
\put(3.3,0.5){\line(0,1){3.3}}
\put(0,3.8){\line(1,0){3.3}}
\multiput(3.3,0.5)(0.25,0){3}{\line(1,0){0.125}}
\multiput(0,3.8)(0,0.25){12}{\line(0,1){0.125}}

\put(0,2.1){\vector(-1,0){1}}
\put(-5.5,1.8){$S_{\e_1}(\Omega_2)$}
\put(7.5,3){\vector(1,0){1}}
\put(8.6,2.6){$S_{\e_2}(\Omega_1)$}
\put(0,8){\vector(-1,0){1}}
\put(-5.5,7.7){$S_{\e_3}(\Omega_1)$}

\put(20,0.5){\line(1,0){2.5}}
\put(20,0.5){\line(0,1){2.5}}
\put(22.5,0.5){\line(0,1){2.5}}
\put(20,3.0){\line(1,0){2.5}}

\put(21.9,0.5){\line(0,1){3.3}}
\put(21.9,0.5){\line(1,0){3.3}}
\put(25.2,3.8){\line(-1,0){3.3}}
\put(25.2,3.8){\line(0,-1){3.3}}

\put(26.7,0.5){\line(0,1){3.3}}
\put(26.7,0.5){\line(1,0){3.3}}
\put(26.7,3.8){\line(1,0){3.3}}
\put(30,0.5){\line(0,1){3.3}}

\put(20,7.2){\line(0,1){3.3}}
\put(20,7.2){\line(1,0){3.3}}
\put(23.3,10.5){\line(0,-1){3.3}}
\put(23.3,10.5){\line(-1,0){3.3}}

\multiput(25.2,0.5)(0.25,0){6}{\line(1,0){0.125}}
\multiput(23.3,10.5)(0.25,0){27}{\line(1,0){0.125}}
\multiput(30,3.85)(0,0.25){27}{\line(0,1){0.125}}
\multiput(20,3.)(0,0.25){19}{\line(0,1){0.125}}

\put(20,2.1){\vector(-1,0){1}}
\put(14.5,1.7){$S_{\e_4}(\Omega_2)$}
\put(24,3.8){\vector(0,1){1}}
\put(21.75,5.25){$S_{\e_5}(\Omega_2)$}
\put(30,2.1){\vector(1,0){1}}
\put(31,1.7){$S_{\e_6}(\Omega_2)$}
\put(20,8.8){\vector(-1,0){1}}
\put(14.5,8.5){$S_{\e_7}(\Omega_2)$}

\end{picture}

\caption{The first iteration of the GIFS defined in Example \ref{E:exam_GIFS_not_str_con_R2},
with $\Omega_1=\cup_{x\in(0,1)}(0,1)\times(x,1-x)$, and $\Omega_2=(2,3)\times(0,1)$.}
\label{F:fig1_not_str_con_R2}
\end{figure}
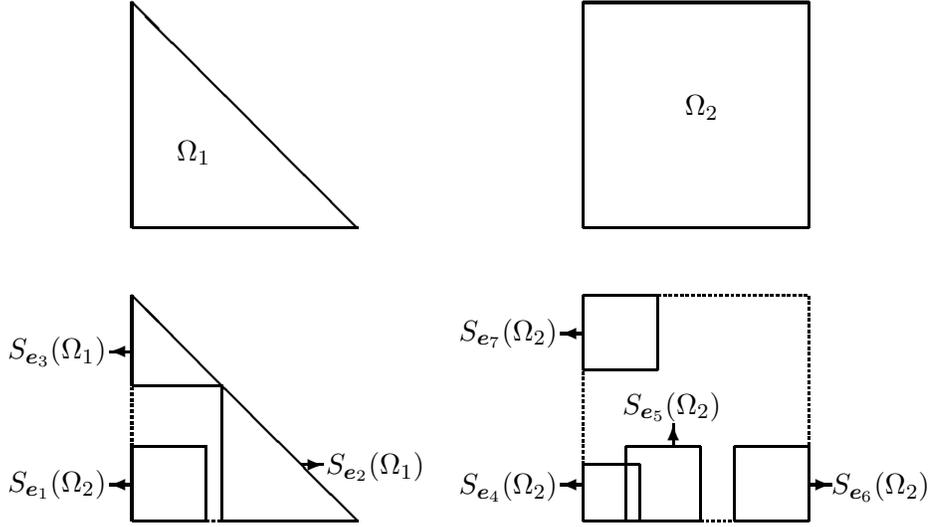
\end{center}

\begin{cor}\label{C:exam_GIFS_not_str_con_R2}
Let $\mu=\mu_1+\mu_2$ be the graph-directed self-similar measure defined by the GIFS $G=(V, E)$ in Example \ref{E:exam_GIFS_not_str_con_R2} together with a probability matrix $(p_e)_{e\in E}$. Then there exists two unique real numbers
$\alpha_i$, $i=1,2$, satisfying $H(q,\alpha)=0$ with
\begin{eqnarray*}
H(q,\alpha):=\Big[1-Q_{2,t}-Q_{3,1-t}\Big]
\bigg[1-\Big(\sum_{i=5}^7 Q_{i,r}\Big)-\Big(\prod_{i=5}^6 \overline{Q}_{i,r}\Big)\Big(\sum_{k=0}^\infty Q^w_{k,\rho,r}\Big)\bigg],
\end{eqnarray*}
where
\begin{eqnarray*}
\begin{aligned}
&Q_{2,t}=Q_{2,t,q,\alpha}:=p^q_{e_2}t^{-\alpha},\qquad \overline{Q}_{2,t}=1-Q_{2,t},\\
&Q_{3,1-t}=Q_{3,1-t}:=p^q_{e_3}(1-t)^{-\alpha},\qquad \overline{Q}_{3,1-t}=1-Q_{3,1-t},\\
&Q_{i,r}= Q_{i,r,q,\alpha}:=p^q_{e_i} r^{-\alpha}, \qquad \overline{Q}_{i,r}=1-Q_{i,r}\quad\mbox{ for }i=5,6,7,\\
&Q^w_{k,\rho,r}=Q^w_{k,\rho, r, \alpha}:=w_k^q (\rho r^k)^{-\alpha}\mbox{ for } k\ge0.
\end{aligned}
\end{eqnarray*}
Hence $\tau(q)=\min\{\alpha_1, \alpha_2\}$. In particular, if $q>0$ and $H_\alpha(q,\alpha)\neq 0$, then $\tau$ is differentiable at $q$. Moreover,
\begin{eqnarray}\label{e:result_tau_q_5}
\begin{aligned}
\tau'
=&\biggl\{\Big[Q_{2,t}\ln (p_{e_2})+Q_{3,1-t}\ln(p_{e_3})\Big]
\bigg[1-\Big(\sum_{i=5}^7 Q_{i,r}\Big)
-\Big(\prod_{i=5}^6 \overline{Q}_{i,r}\Big)\Big(\sum_{k=0}^\infty Q^w_{k,\rho,r}\Big)\bigg]\\
&+\bigg[\Big(\sum_{i=5}^7 Q_{i,r}\ln(p_{e_i})\Big)
-\Big(Q_{5,r}\ln (p_{e_5})\overline{Q}_{6,r}+Q_{6,r}\ln(p_{e_6})\overline{Q}_{5,r}\Big)
\Big(\sum_{k=0}^\infty Q^w_{k,\rho,r}\Big)\\
&\quad+\Big(\prod_{i=5}^6 \overline{Q}_{i,r}\Big)\Big(\sum_{k=0}^\infty Q^w_{k,\rho,r}\ln (w_k)\Big)\bigg]
\Big[1-Q_{2,t}-Q_{3,1-t}\Big]\biggr\}\\
&\cdot\biggl\{\Big[Q_{2,t}\ln t+ Q_{3,1-t}\ln(1-t)\Big]
\bigg[1-\Big(\sum_{i=5}^7 Q_{i,r}\Big)-\Big(\prod_{i=5}^6 \overline{Q}_{i,r}\Big)\Big(\sum_{k=0}^\infty Q^w_{k,\rho,r}\Big)\bigg]\\
&\quad+\bigg[\Big(\sum_{i=5}^7 Q_{i,r}\ln r\Big)
-\Big(Q_{5,r}\ln(r)\overline{Q}_{6,r}+Q_{6,r}\ln(r)\overline{Q}_{5,r}\Big)\Big(\sum_{k=0}^\infty Q^w_{k,\rho,r}\Big)\\
&\quad\quad+\Big(\prod_{i=5}^6 \overline{Q}_{i,r}\Big)\Big(\sum_{k=0}^\infty Q^w_{k,\rho,r}\ln (\rho r^k)\Big)\bigg]
\Big[1-Q_{2,t}-Q_{3,1-t}\Big]\biggr\}^{-1}.
\end{aligned}
\end{eqnarray}
If $\ell=1,4$, then
  \begin{eqnarray*}
  \lim_{x\to\infty}\big(e^{(2+\alpha)x}\varphi_\ell(e^{-x})-\widetilde{q}_\ell(x)\big)=0\quad\mbox{ for some }\widetilde{q}_\ell,
  \end{eqnarray*}
where $\widetilde{q}_\ell$ is either non-negative constant or periodic;
if $\ell=2,3,5,6$, then
  \begin{eqnarray*}
  \lim_{x\to\infty}\big(e^{(2+\alpha)x}\varphi_\ell(e^{-x})-q_\ell(x)\big)=0,
  \end{eqnarray*}
  where $q_\ell$ is either periodic or non-negative constant depending on whether $\nu_\ell$ is lattice or not.
\end{cor}

\begin{rem}
Recently, Kesseb\"ohmer and Niemann \cite{Kessebohmer-Niemann_2022_1, Kessebohmer-Niemann_2022} studied the spectral dimension of
the Kre\u{\i}n-Feller operator defined by an arbitrary finite Borel measure $\mu$ on $\R^d$ for $d\ge1$.
In \cite[Theorem 1.16]{Kessebohmer-Niemann_2022}, they proved that for self-conformal measures, the spectral dimension can be obtained from the unique intersection of the $L^q$-spectrum with the line with slope $2-d$ and passing through the origin.
It follows that the spectral dimension of $\mu$ defined as in
Corollaries \ref{C:cor_GIFS_str_con_R2} and \ref{C:exam_GIFS_not_str_con_R2} can be obtained.
Furthermore, the spectral dimension of $\mu$ defined as in Corollaries \ref{C:cor_str_con_R}, \ref{C:cor_GIFS_not_str_con_R_1} and \ref{C:cor_GIFS_not_str_con_R_2} have been studied by Ngai and the author \cite{Ngai-Xie_2020}.
\end{rem}

We use the multidimensional renewal theorem proved by Hambly and Nyberg \cite{Hambly-Nyberg_2003},
which is an extension of the vector-valued renewal theorem of Lau et al. \cite{Lau-Wang-Chu_1995}, to derive the stated formulas for $\tau(q)$. The classical renewal theorem used in \cite{Lau-Ngai_1998, Lau-Ngai_2000} and the vector-valued renewal theorem
are not sufficient, as there are a finite number of renewal equations in our derivations, and the measures we considered are graph-directed self-similar measures.

This paper is organized as follows. In Section \ref{S:GEFT}, we recall a modified version of the definition of (EFT).
We derive renewal equations and prove Theorems \ref{T:main_thm_str_con} and \ref{T:main_thm_not_str_con} in Section \ref{S:Renewal_equation}.
Section \ref{S:str_con_GIFSs} illustrates Theorem \ref{T:main_thm_str_con} by the strongly connected GIFSs defined in Examples \ref{E:exam_str_con_R} and \ref{E:exam_str_con_R2}; we also prove Corollaries \ref{C:cor_str_con_R} and \ref{C:cor_GIFS_str_con_R2}. In Section \ref{S:not_str_conn_R}, we study GIFSs in Examples \ref{E:exam_not_str_con_R_1}, \ref{E:exam_GIFS_not_str_con_R_2}
and \ref{E:exam_GIFS_not_str_con_R2}, which are not strongly connected; we also prove Corollaries \ref{C:cor_GIFS_not_str_con_R_1}, \ref{C:cor_GIFS_not_str_con_R_2} and \ref{C:exam_GIFS_not_str_con_R2}.

\section{Graph-directed self-similar systems and measures that are essentially of finite type}\label{S:GEFT}
\setcounter{equation}{0}

\subsection{Graph-directed iterated function systems}\label{SS:GIFS}

A \textit{graph-directed iterated function system} (GIFS) of contractive similitudes is an ordered pair $G=(V,E)$ (see \cite{Mauldin-Williams_1988}) with $V=\{1,\ldots,M\}$ being a set of \textit{vertices} and $E$ being a set of \textit{directed edges}, each beginning and ending at a vertex. It is possible for an edge to begin and end at the same vertex,
and more than one edge between two vertices are permitted. For any $e\in E$, there is a corresponding contractive similitude $S_e(x):\R^d\to\R^d$ defined by $$S_e(x)=\rho_eR_ex+b_e,$$ with $\rho_e\in(0,1)$ being the contraction ratio, $R_e$ being an orthogonal transformation, and $b_e\in\R^d$. Let $E^{i,j}$ be the set of all edges that begin and end at vertices $i$ and $j$, respectively.
We say that $\e=e_1\dots e_k$ is a \textit{path} (or an \textit{$\e$-path}) with length $k$, if the terminal vertex of each edge $e_i$ $(1\le i\le k-1)$ and the initial vertex of the edge $e_{i+1}$ are the same one. It is widely recognized that there exists a unique family of nonempty compact sets $K_1,\ldots,K_M$ satisfying
\begin{equation}\label{e:Fi}
K_i=\bigcup_{j=1}^M \bigcup_{e\in E^{i,j}}S_e(K_j),\qquad i=1,\ldots,M.
\end{equation}
We refer to
\begin{equation}\label{e:F}
K:=\bigcup_{i=1}^M K_i.
\end{equation}
as the \textit{graph self-similar set} associated with $G=(V,E)$. Assume that for each edge $e\in E$, there is a corresponding transition probability $p_e>0$, and the weights of all edges, which begin at a given vertex $i$, sum to $1$, namely,
\begin{equation}\label{e:transition_probability}
\sum_{j\in V}\sum_{e\in E^{i,j}}p_e=1.
\end{equation}
Then for each $i\in V$, there exists a unique Borel probability measures $\mu_i$ such that
\begin{equation}\label{e:probability_measures}
\mu_i=\sum_{j=1}^M \sum_{e\in E^{i,j}}p_e\mu_j\circ S_e^{-1}.
\end{equation}
We note that ${\rm supp}(\mu_i)=K_i$ for all $i\in V$. Lastly, let $\mu:=\sum_{i=1}^M \mu_i$ be a \textit{graph-directed self-similar measure}. We say \textit {the graph open set condition} (GOSC) (see \cite{Wang_1997}) holds for $G=(V,E)$ if there exists a family $\{O_i\}_{i=1}^M \subseteq\R^d$ of nonempty bounded open sets such that for all $i, j, j'\in V$,
\begin{equation*}
\bigcup_{e\in E^{i,j}} S_e(O_j)\subseteq O_i\quad\mbox{ and }\quad S_e(O_j)\cap S_{e'}(O_j)=\emptyset\quad \mbox{for all distinct }e\in E^{i,j}\mbox{ and }e'\in E^{i, j'}.
\end{equation*}
Obviously, $K_i\subseteq \overline{O}_i$. A GIFS, as well as any associated graph-directed self-similar measure, are said to have \textit{overlaps} if (GOSC) fails. Let $\{\Omega_i\}_{i=1}^M$ denote a family of nonempty bounded open subsets of $\R^d$. $\{\Omega_i\}_{i=1}^M$ is said to be \textit{invariant} under the GIFS $G=(V,E)$ if $\bigcup_{e\in E^{i,j}}S_e(\Omega_j)\subseteq\Omega_i$ for $i=1,\ldots,M$.
We say $G$ is \textit{connected} if for each pair of vertices $i,j\in V$, there exists a (non-directed) path connecting them.
$G$ is said to be \textit{strongly connected} if for each pair of vertices $i,j\in V$, there exists a directed path from $i$ to $j$. A strongly connected component of $G$ is a maximal subgraph $G'$ of $G$ such that $G'$ is strongly connected. Strongly connected components are pairwise disjoint and do not necessarily encompass $G$. A single vertex may be a strongly connected component if it has a loop to itself. In this paper, it is assumed that each graph contains at least one strongly connected component.

\subsection{The essentially finite type condition for graph-directed self-similar measures}\label{S:sub_EFT}

Let $\Omega\subseteq \R^d$ be a bounded open subset and $\mu$ be a positive finite Borel measure with ${\rm supp}(\mu)\subseteq \overline \Omega$ and $\mu(\Omega)>0$. A $\mu$-measurable subset $U$ of $\Omega$ is called a \textit{cell (in $\Omega$)} if $\mu(U)>0$. Obviously, $\Omega$ itself is a cell.

Two cells $U$ and $W$ are considered \textit{$\mu$-equivalent}, denoted as $U\simeq_{\mu,\tau,w}W$ (or simply $U\simeq_{\mu} W$), if there exist some similitude $\tau:U\to W$ and some constant $w>0$ such that $\tau(U)=W$ and
\begin{equation}
\mu|_{W}=w\cdot\mu|_{U}\circ\tau^{-1}.
\end{equation}
It is straightforward to verify that $\simeq_\mu$ is an equivalence relation.

We say that two cells $U, W$ in $\Omega$ are \textit{measure disjoint} with respect to $\mu$ if $\mu(U\cap W)=0$.
A finite family ${\mathbf P}$ of measure disjoint cells is referred as a \textit{$\mu$-partition} of $\Omega$
if $U\subseteq \Omega$ for all $U\in {\mathbf P}$, and $\mu(\Omega)=\sum\limits_{U\in {\mathbf P}}\mu(U)$.
A sequence of $\mu$-partitions $\{\mathbf{P}_k\}_{k\ge1}$ is \textit{refining} if for any $W\in\mathbf{P}_k$ and any $U\in \mathbf{P}_{k+1}$, either $U\subseteq W$ or they are measure disjoint, i.e., each member of $\mathbf{P}_{k+1}$ is a subset of some member of $\mathbf{P}_k$.

Let ${\mathbf B}:=\{B_{1,\ell}\}_{\ell\in \Gamma}$ represent a finite family of measure disjoint cells in $\Omega$, and for each $\ell\in \Gamma$, let $\{{\mathbf P}_{k,\ell}\}_{k\ge 1}$ be a family of refining $\mu$-partitions of $B_{1,\ell}$ with ${\mathbf P}_{1,\ell}:=\{B_{1,\ell}\}$, where $\Gamma$ is a finite index set. We divide each ${\mathbf P}_{k,\ell}$, $ k\ge 2$, into two (possibly empty) subcollections, ${\mathbf P}'_{k,\ell}$ and ${\mathbf P}_{k,\ell}^\ast$, with respect to ${\mathbf B}$, defined as follows:
\begin{equation}\label{eq:P_1_2}
\begin{aligned}\
{\mathbf P}'_{k,\ell}&:=\big\{B\in {\mathbf P}_{k,\ell}: B\simeq_\mu B_{1,i}\mbox{ for some }i\in \Gamma\big\},\\
{\mathbf P}^\ast_{k,\ell}&:={\mathbf P}_{k,\ell}\setminus{\mathbf P}'_{k,\ell}
=\big\{B\in {\mathbf P}_{k,\ell}: B\notin {\mathbf P}'_{k,\ell}\big\}.
\end{aligned}
\end{equation}
\begin{defi}\label{defi:EFT}
A graph-directed self-similar measure $\mu=\sum_{i=1}^M \mu_i$ on $\R^d$ is said to be {\rm essentially of finite type (EFT)} if there exist a family of bounded open subsets $\{\Omega_i\}_{i=1}^q$ with $\Omega_i\subseteq \R^d$, ${\rm supp}(\mu_i)\subseteq \overline \Omega_i$ and $\mu(\Omega_i)>0$, and a finite family ${\mathbf B}:=\{B_{1,\ell}\}_{\ell\in \Gamma}$ of measure disjoint cells, $B_{1,\ell}\subseteq\Omega_{i_\ell}$ for some $i_\ell=1,\ldots,M$, such that for any $\ell\in \Gamma$, there is a family of $\mu$-partitions $\{{\mathbf P}_{k,\ell}\}_{k\ge 1}$ of $B_{1,\ell}$ satisfying the following conditions:
\begin{enumerate}
  \item[(1)]${\mathbf P}_{1,\ell}=\{B_{1,\ell}\}$, and there exists some $B\in {\mathbf P}'_{2,\ell}$ such that $B\neq B_{1,\ell}$;
   \item[(2)]if for some $k\ge 2$, there exists some $B\in {\mathbf P}'_{k,\ell}$, then $B\in {\mathbf P}'_{k+1,\ell}$ and hence $B\in{\mathbf P}'_{m,\ell}$ for all $m\ge k$;
  \item[(3)] $\lim\limits_{k\to\infty}\sum\limits_{B\in\mathbf{P}^\ast_{k,\ell}}\mu(B)=0$.
\end{enumerate}
Here ${\mathbf P}'_{k,\ell}$ and ${\mathbf P}^\ast_{k,\ell}$ ($k\ge 2$) are defined as in \eqref{eq:P_1_2}.
In this case, $\{\Omega_i\}_{i=1}^{M}$ is referred to as an {\rm EFT-family}, $\mathbf{B}$ as a {\rm basic family of cells}, and $({\mathbf B},{\mathbf P}):=(\{B_{1,\ell}\},\{{\mathbf P}_{k,\ell}\}_{k\ge 1})_{\ell\in \Gamma}$ as a {\rm basic pair}.
\end{defi}

Let
$$\mathbf{P}_{k,\ell}=\{B_{k,\ell,i}, i=1,2,\ldots\}\qquad\mbox{ for }k\ge2\mbox{ and }\ell\in\Gamma,$$
where $B_{k,\ell,i}$ denotes the $i$-th measure disjoint cell of the $\mu$-partition $\mathbf{P}_{k.\ell}$.

\begin{defi}\label{defi:regular_pair}
Assume that a graph-directed self-similar measure $\mu=\sum_{i=1}^M \mu_i$ satisfies (EFT) with $\{\Omega_i\}_{i=1}^M$ being an EFT-family and $({\mathbf B},{\mathbf P}):=(\{B_{1,\ell}\},\{{\mathbf P}_{k,\ell}\}_{k\ge 1})_{\ell\in \Gamma}$ being a basic pair. The basic pair $({\mathbf B},{\mathbf P})$ is called {\rm weakly regular} if for any $\ell\in \Gamma$, there exist some similitude $\sigma_\ell$ and some $\Omega_{j_\ell}$ such that $\sigma_\ell(\Omega_{j_\ell})\subseteq B_{1,\ell}$.
In this case, we call ${\mathbf B}$ a {\rm weakly regular basic family of cells}.
We say that $({\mathbf B},{\mathbf P})$ is {\rm regular} if it is weakly regular and $\mu\ge \omega(\ell)\mu\circ\sigma_\ell^{-1}$ on $\sigma_\ell(\Omega_{j_\ell})$ for some constant $\omega(\ell)>0$.
\end{defi}

\section{Renewal equation and proofs of Theorem \ref{T:main_thm_str_con} and Theorem \ref{T:main_thm_not_str_con}.}\label{S:Renewal_equation}
\setcounter{equation}{0}

Let $G=(V, E)$ be a finite type GIFS, and $\{S_e\}_{e\in E}$ be the corresponding contractive similitudes defined on a compact subset $X\subseteq \R^d$. Let $\Omega=\{\Omega_i\}_{i=1}^M$ be a finite type condtion family of $G$ with $\cup_{i=1}^M \Omega_i\subseteq X$. Let $\mu=\sum_{i=1}^M \mu_i$ be a graph-directed self-similar measures defined by a GIFS $G=(V,E)$ together with a probability matrix $(p_e)_{e\in E}$. Assume that $G$ has $\eta$ strongly connected compute the $L^q$-spectrum $\tau(q)$ for $q\ge 0$, we will use the equivalent definition in \eqref{e:equi_defi_of_Lq_spectrum}. Assume \eqref{e:equi_defi_of_Lq_spectrum}. Assume that $\mu$ satisfies (EFT) with $\{\Omega_i\}_{i=1}^M$ being an EFT-family and assume that there exists a weakly regular basic pair $({\bf B}, {\bf P}):=(\{B_{1,\ell}\}, \{{\bf P}_{k,\ell}\}_{k\ge 1})_{\ell\in \Gamma}$. Let $\varphi_\ell(h)$ and $\Phi_\ell^{(\alpha)}(h)$ be defined as in \eqref{e:defi_varphi_and_Phi}. In the following, we will give some results about $\varphi_\ell(h)$ and $\tau(q)$. The proof is similar to that of \cite[Proposition 3.1]{Ngai-Xie_2019}.

\begin{prop}\label{P:equ_defi_tau_q_on_rd}
Assume the above hypotheses hold. Let $q\ge 0$. Then there exist two positive constants $c_1, c_2$ such that
\begin{equation}\label{e:ineq_about_int_mu}
c_1\int_X \mu(B_{c_2h}(x))^q \,dx\le \sum_{\ell\in\Gamma}\int_{B_{1,\ell}}\mu(B_h(x))^q\,dx\le \int_X \mu(B_h(x))^q\,dx.
\end{equation}
Consequently,
\begin{eqnarray}\label{eq:tau_q_on_rd_equ}
\tau(q)&=&\inf\left\{\alpha\ge0: \varlimsup_{h\to 0^+}\sum_{\ell\in\Gamma}\Phi_\ell^{(\alpha)}(h)>0\right\}\notag\\
&=&\sup\left\{\alpha\ge0: \varlimsup_{h\to 0^+}\sum_{\ell\in\Gamma}\Phi_\ell^{(\alpha)}(h)<\infty\right\}.
\end{eqnarray}
\end{prop}

For $i\in V$, let $\Gamma_i$ be defined as in \eqref{e:defi_Gamma_i} and
\begin{eqnarray}
{\bf B}_i:=\{B_{1,\ell}:\ell\in\Gamma_i\}.
\end{eqnarray}
Obviously, $\Gamma=\bigcup_{i=1}^M \Gamma_i$ and ${\bf B}=\bigcup_{i=1}^M {\bf B}_i$.
Also, $\Gamma_i$ and ${\bf B}_i$ maybe empty.
In view of Proposition \ref{P:equ_defi_tau_q_on_rd}, it is suffices to study $\Phi_\ell^{(\alpha_m)}(h)$ for $\ell\in\Gamma_i$.

\noindent\textit{Step 1. Derivation of a functional equation.}
For $m=1,\ldots,\eta$, $i\in \mathcal{SC}_m$, $\ell\in\Gamma_i$ and $k\ge 2$, let ${\bf P}'_{k,\ell}$ and ${\bf P}^\ast_{k,\ell}$ be defined as in \eqref{eq:P_1_2} with respect to ${\bf B}$. Without loss of generality, we may assume that $\Gamma_i$ can be partitioned into two (possibly empty) sub-collections, $\Gamma_i'$ and $\Gamma_i^\ast$, defined as follows. For $\ell\in\Gamma_i$, we say $\ell\in\Gamma_i'$ if there exists some integer $\kappa$ satisfying ${\bf P}^\ast_{k,\ell}=\emptyset$. Denote $\kappa_\ell\ge 2$ (depending on $\ell$) be the smallest of such $\kappa$. Define $\Gamma_i^\ast:=\Gamma_i\backslash\Gamma_i'$ and let $\kappa_\ell:=\infty$ for $\ell\in\Gamma_i^\ast$. Denote
\begin{eqnarray*}
\Gamma'=\bigcup_{i=1}^M \Gamma'_i\quad\mbox{and}\quad\Gamma^\ast=\bigcup_{i=1}^M \Gamma_i^\ast.
\end{eqnarray*}
It is easy to see that $\Gamma=\Gamma'\cup \Gamma^\ast$ and $\Gamma_i=\Gamma'_i\cup \Gamma^\ast_i$. For $m=1,\ldots,\eta$, $i\in \mathcal{SC}_m$, fix any $\ell\in\Gamma_i$. The definition of (EFT) implies that for any $2\le k\le\kappa_\ell$, there exist two finite disjoint subsets $G'_{k,\ell}, G^\ast_{k,\ell}\subseteq\mathbb{N}$ such that
\begin{eqnarray*}
{\bf P}'_{k,\ell}=\bigcup_{k_0=2}^k\{B_{k_0,\ell,i}: i\in G'_{k,\ell}\}\quad\mbox{and}
\quad{\bf P}^\ast_{k,\ell}=\{B_{k,\ell,i}: i\in G^\ast_{k,\ell}\}.
\end{eqnarray*}
It follows from condition (1) of EFT that $G'_{2,\ell}\neq\emptyset$ for $\ell\in\Gamma_i$. Condition (3) of EFT implies that
$\lim_{k\to\infty}\sum_{j\in G^\ast_{k,\ell}}\mu(B_{k,\ell,j})=0$ for $\ell\in \Gamma^\ast$. Thus, for all $\ell\in\Gamma'_i$,
\begin{eqnarray}\label{eq:var'_1}
\varphi_\ell(h)=\sum_{k=2}^{\kappa_\ell}\sum_{j\in G'_{k,\ell}}\int_{B_{k,\ell,j}}\mu(B_h(x))^q\,dx,
\end{eqnarray}
and for all $\ell\in\Gamma^\ast_i$ and $n\ge 2$,
\begin{eqnarray}\label{eq:var_ast_1}
\varphi_\ell(h)=\sum_{k=2}^n\sum_{j\in G'_{k,\ell}}\int_{B_{k,\ell,j}}\mu(B_h(x))^q\,dx
+\sum_{j\in G^\ast_{n,\ell}}\int_{B_{n,\ell,j}}\mu(B_h(x))^q\,dx.
\end{eqnarray}

For $m=1,\ldots,\eta$, $i\in\mathcal{SC}_m$, $\ell\in\Gamma_i$, $2\le k\le \kappa_\ell$, $j\in G'_{k,\ell}$ and $h>0$,
let $\widetilde{B}_{k,\ell,j}(h)$ be the largest subset of $B_{k,\ell,j}$ satisfying
$B_h(x)\subseteq B_{k,\ell,j}$ for any $x\in\widetilde{B}_{k,\ell,j}(h)$.
We denote
\begin{eqnarray*}
\widehat{B}_{k,\ell,j}(h):=B_{k,\ell,j} \backslash \widetilde{B}_{k,\ell,j}(h).
\end{eqnarray*}
So, \eqref{eq:var'_1} and \eqref{eq:var_ast_1} can be written as
\begin{eqnarray}\label{eq:var'_2}
\varphi_\ell(h)=\sum_{k=2}^{\kappa_\ell}\sum_{j\in G'_{k,\ell}}\Big(\int_{\widetilde{B}_{k,\ell,j}}+\int_{\widehat{B}_{k,\ell,j}}\Big)\mu(B_h(x))^q\,dx\quad\mbox{ for all }\ell\in\Gamma'_i,
\end{eqnarray}
and for all $\ell\in\Gamma^\ast_i$ and $n\ge 2$,
\begin{eqnarray}\label{eq:var_ast_2}
\begin{aligned}
\varphi_\ell(h)
=\sum_{k=2}^n\sum_{j\in G'_{k,\ell}}\Big(\int_{\widetilde{B}_{k,\ell,j}}+\int_{\widehat{B}_{k,\ell,j}}\Big)\mu(B_h(x))^q\,dx
+\sum_{j\in G^\ast_{n,\ell}}\int_{B_{n,\ell,j}}\mu(B_h(x))^q\,dx.
\end{aligned}
\end{eqnarray}
For $m=1,\ldots,\eta$, $i\in\mathcal{SC}_m$, $\ell\in\Gamma_i$, $2\le k\le\kappa_\ell$ and $j\in G'_{k,\ell}$, by the definition of
${\bf P}'_{k,\ell}$, there exist some similitude $S_{e(k,\ell,j)}$ with Lipschitz constant $\rho_{e(k,\ell,j)}$, as well as constants
$w(k,\ell,j)>0$ and $c(k,\ell,j)\in \Gamma_j$ such that
\begin{eqnarray}\label{eq:relationship_between_mui_and_B1c}
\mu_i|_{B_{k,\ell,j}}=w(k,\ell,j)\mu_j|_{B_{1,c(k,\ell,j)}}\circ S^{-1}_{e(k,\ell,j)},
\end{eqnarray}
where $j\in J_\ell:=\{j\in V:S_e(\Omega_j)\subseteq B_{1,\ell}\mbox{ for }\ell\in E^{i,j}\}.$
For $\widetilde{B}_{k,\ell,j}(h)\subseteq B_{k,\ell,j}$, let $\widetilde{B}_{1,c(k,\ell,j)}(h/{\rho_{e(k,\ell,j)}})$ be the largest subset of $B_{1,c(k,\ell,j)}$ satisfying
$B_{h/{\rho_{e(k,\ell,j)}}}(x)\subseteq B_{1,c(k,\ell,j)}$ for any $x\in \widetilde{B}_{1,c(k,\ell,j)}(h/{\rho_{e(k,\ell,j)}})$.
Thus
\begin{eqnarray*}
\mu_i|_{\widetilde{B}_{k,\ell,j}(h)}=w(k,\ell,j)\mu_j|_{\widetilde{B}_{1,c(k,\ell,j)}(h/{\rho_{e(k,\ell,j)}})}\circ S^{-1}_{e(k,\ell,j)}.
\end{eqnarray*}
Denote
\begin{eqnarray*}
\widehat{B}_{1,c(k,\ell,j)}(h/{\rho_{e(k,\ell,j)}})=B_{1,c(k,\ell,j)}\backslash\widetilde{B}_{1,c(k,\ell,j)}(h/{\rho_{e(k,\ell,j)}}).
\end{eqnarray*}
Hence for all $\ell\in\Gamma'_i$,
\begin{eqnarray}\label{eq:var'_3}
\begin{aligned}
\varphi_\ell(h)
=\sum_{k=2}^{\kappa_\ell}\sum_{j\in G'_{k,\ell}} w(k,\ell,j)^q\rho_{e(k,\ell,j)}^d \int_{B_{1,c(k,\ell,j)}}\mu(B_{h/{\rho_{e(k,\ell,j)}}}(x))^q\,dx
+\sum_{k=2}^{\kappa_\ell}e_{k,\ell},
\end{aligned}
\end{eqnarray}
and for all $\ell\in\Gamma^\ast_i$ and $n\ge2$,
\begin{eqnarray}\label{eq:var_ast_3}
\begin{aligned}
\varphi_\ell(h)
=&\sum_{k=2}^n\sum_{j\in G'_{k,\ell}} w(k,\ell,j)^q\rho_{e(k,\ell,j)}^d \int_{B_{1,c(k,\ell,j)}}\mu(B_{h/{\rho_{e(k,\ell,j)}}}(x))^q\,dx+\sum_{k=2}^{n}e_{k,\ell}\\
&+\sum_{j\in G^\ast_{n,\ell}}\int_{B_{n,\ell,j}}\mu(B_h(x))^q\,dx,
\end{aligned}
\end{eqnarray}
with
\begin{eqnarray*}
\begin{aligned}
e_{k,\ell}=&-\sum_{j\in G'_{k,\ell}} w(k,\ell,j)^q\cdot \rho_{e(k,\ell,j)}^d
\int_{\widehat{B}_{1,c(k,\ell,j)}(h/{\rho_{e(k,\ell,j)}})}\mu(B_{h/{\rho_{e(k,\ell,j)}}}(x))^q\,dx\\
&+\sum_{j\in G'_{k,\ell}}\int_{\widehat{B}_{k,\ell,j}(h)}\mu(B_h(x))^q\,dx
\qquad\qquad\mbox{for }\ell\in \Gamma'_i\cup\Gamma^\ast_i.
\end{aligned}
\end{eqnarray*}
Multiplying both sides of \eqref{eq:var'_3} and \eqref{eq:var_ast_3} by $h^{-(d+\alpha_m)}$, we get for $\ell\in\Gamma'_i$,
\begin{eqnarray}\label{eq:var'_4}
\Phi_\ell^{(\alpha_m)}(h)=\sum_{k=2}^{\kappa_\ell}\sum_{j\in G'_{k,\ell}}w(k,\ell,j)^q \rho_{e(k,\ell,j)}^{-\alpha_m} \Phi_{c(k,\ell,j)}^{(\alpha_m)}(h/{\rho_{e(k,\ell,j)}})
+E_\ell^{(\alpha_m)}(h),
\end{eqnarray}
with $E_\ell^{(\alpha_m)}(h)=\sum_{k=2}^{\kappa_\ell}h^{-(d+\alpha_m)}e_{k,\ell}$,
and for $\ell\in\Gamma^\ast_i$ and $n\ge2$,
\begin{eqnarray}\label{eq:var_ast_4}
\begin{aligned}
\Phi_\ell^{(\alpha_m)}(h)
=&\sum_{k=2}^n\sum_{j\in G'_{k,\ell}} w(k,\ell,j)^q \rho_{e(k,\ell,j)}^{-\alpha_m}\Phi_{c(k,\ell,j)}^{\alpha_m}(h/{\rho_{e(k,\ell,j)}})\\
&+h^{-(d+\alpha_m)}\Big(\sum_{k=2}^n e_{k,\ell}+ \sum_{j\in G^\ast_{n,\ell}}\int_{B_{n,\ell,j}}\mu(B_h(x))^q\,dx\Big).
\end{aligned}
\end{eqnarray}
For $h>0$ and $\ell\in\Gamma_i^\ast$, let
\begin{eqnarray}\label{defi:N_on_kuangjia}
N=N(\ell):=\max\{n\in\mathbb{N}: h\le \max\{\rho_{e(k,\ell,j)}: j\in G'_{k,\ell}\mbox{ for any }k\le n\}\}.
\end{eqnarray}
Letting $n=N$ in \eqref{eq:var_ast_4}, we have for $\ell\in\Gamma^\ast_i$ and $n\ge2$
\begin{eqnarray}\label{eq:var_ast_5}
\Phi_\ell^{(\alpha_m)}(h)
&=&\sum_{k=2}^N\sum_{j\in G'_{k,\ell}} w(k,\ell,j)^q \rho_{e(k,\ell,j)}^{-\alpha_m}\Phi_{c(k,\ell,j)}^{\alpha_m}(h/{\rho_{e(k,\ell,j)}})\notag\\
&~&+h^{-(d+\alpha_m)}\Big(\sum_{k=2}^N e_{k,\ell}+ \sum_{j\in G^\ast_{N,\ell}}\int_{B_{N,\ell,j}}\mu(B_h(x))^q\,dx\Big)\notag\\
&=&\sum_{k=2}^\infty\sum_{j\in G'_{k,\ell}} w(k,\ell,j)^q \rho_{e(k,\ell,j)}^{-\alpha_m}\Phi_{c(k,\ell,j)}^{\alpha_m}(h/{\rho_{e(k,\ell,j)}})
+E_\ell^{(\alpha_m)}(h)-E_{\ell,\infty}^{(\alpha_m)}(h).
\end{eqnarray}
with
\begin{eqnarray*}
E_\ell^{(\alpha_m)}(h)=h^{-(d+\alpha_m)}\Big(\sum_{k=2}^N e_{k,\ell}+ \sum_{j\in G^\ast_{N,\ell}}\int_{B_{N,\ell,j}}\mu(B_h(x))^q\,dx\Big),
\end{eqnarray*}
and
\begin{eqnarray*}
E_{\ell,\infty}^{(\alpha_m)}(h)=\sum_{k=N+1}^\infty\sum_{j\in G'_{k,\ell}} w(k,\ell,j)^q\rho_{e(k,\ell,j)}^{-\alpha_m}\Phi_{c(k,\ell,j)}^{(\alpha_m)}(h/{\rho_{e(k,\ell,j)}}).
\end{eqnarray*}

\noindent\textit{Step 2. Derivation of the vector-valued equation.}

For $m=1,\ldots, \eta$, $i\in\mathcal{SC}_m$ and $\ell\in\Gamma_i$, define
\begin{equation}\label{eq:defi_of_f_ell}
f_\ell(x)=f_\ell^{(\alpha_m)}(x):=\Phi_\ell^{(\alpha_m)}(e^{-x}).
\end{equation}
Let $h=e^{-x}$. Then
\begin{equation*}
\Phi_\ell^{(\alpha_m)}(\beta h)=f_\ell(x-\ln\beta)\quad\mbox{ for any }\beta>0.
\end{equation*}
Using \eqref{eq:var'_4} and \eqref{eq:var_ast_5}, we get for $\ell\in\Gamma'_i$,
\begin{eqnarray}\label{eq:f_ell_'_1}
f_\ell(x)=\sum_{k=2}^{\kappa_\ell}\sum_{j\in G'_{k,\ell}}w(k,\ell,j)^q \rho_{e(k,\ell,j)}^{-\alpha_m} f_{c(k,\ell,j)}(x+\ln({\rho_{e(k,\ell,j)}}))+z_\ell^{(\alpha_m)}(x)
\end{eqnarray}
and for $\ell\in\Gamma^\ast_i$ and $n\ge2$,
\begin{eqnarray}\label{eq:f_ell_ast_1}
\begin{aligned}
f_\ell(x)
=\sum_{k=2}^\infty\sum_{j\in G'_{k,\ell}} w(k,\ell,j)^q \rho_{e(k,\ell,j)}^{-\alpha_m}f_{c(k,\ell,j)}(x+\ln(\rho_{e(k,\ell,j)}))
+z_\ell^{(\alpha_m)}(x)-z_{\ell,\infty}^{(\alpha_m)}(x).
\end{aligned}
\end{eqnarray}
with $z_\ell^{(\alpha_m)}(x)=E_\ell^{(\alpha_m)}(e^{-x})$ for $\ell\in\Gamma'_i\cup\Gamma^\ast_i$,
and $z_{\ell,\infty}^{(\alpha_m)}(x)=E_{\ell,\infty}^{\alpha_m}(e^{-x})$ for $\ell\in\Gamma^\ast_i$.

For $m=1,\ldots,\eta$, $i\in\mathcal{SC}_m$ and $\ell\in\Gamma_i$, let $\mu^{(\alpha_m)}_{\ell\ell'}$ be the discrete measure such that
\begin{eqnarray}\label{defi:muij_on_kuangjia}
\mu^{(\alpha_m)}_{\ell\ell'}(-\ln(\rho_{e(k,\ell,j)}))=w(k,\ell,j)^q\rho_{e(k,\ell,j)}^{-\alpha_m}
\mbox{ for }2\le k\le\kappa_\ell,~j\in G'_{k,\ell}\mbox{ and }\ell'=c(k,\ell,j).
\end{eqnarray}
Then
\begin{equation*}
\mu_{\ell\ell'}^{(\alpha_m)}(\mathbb{R})=\sum_{k=2}^{\kappa_\ell}\sum_{j\in G'_{k,\ell}}w(k,\ell,j)^q\rho_{e(k,\ell,j)}^{-\alpha_m}
\quad\mbox{ with }\ell'=c(k,\ell,j),
\end{equation*}
and
\begin{equation}\label{e:section_F_ell_alpha_m}
F_\ell(\alpha_m)=\sum_{\ell'\in\Gamma}\mu_{\ell\ell'}^{(\alpha_m)}(\mathbb{R})
=\sum_{\ell'\in \Gamma}\sum_{k=2}^{\kappa_\ell}\sum_{j\in G'_{k,\ell}} w(k,\ell,j)^q\rho_{e(k,\ell,j)}^{-\alpha_m}.
\end{equation}

\begin{thm}\label{T:vector-valued_renewal_equation}
Let $\mu=\sum_{i=1}^M\mu_i$ be a graph-directed self-similar measure defined by a GIFS $G=(V,E)$ on $\R^d$. Assume that $\mu$ satisfies (EFT). Let ${\bf f}, {\bf M}_{\boldsymbol\alpha}$, and ${\bf z}$ be defined as in \eqref{e:defi_f_M_z}. Then ${\bf f}$ satisfies the vector-valued renewal equation ${\bf f}={\bf f}\ast{\bf M}_{\boldsymbol\alpha}+{\bf z}$.
\end{thm}

Before proving Theorem \ref{T:main_thm_str_con} and \ref{T:main_thm_not_str_con}, we state a result relating strong connectedness of a graph and irreducibility of the corresponding matrix. The proof is similar to that of \cite[Proposition 4.4]{Ngai-Xie_2020}, and the details are omitted.

\begin{prop}\label{P:relation_strong_connectedness_and_irreducibility}
Let $\mu=\sum_{i=1}^M \mu_i$ be a graph-directed self-similar measure defined by a GIFS $G=(V,E)$ on $\R^d$. Assume that $\mu$ satisfies (EFT) with $\{\Omega_i\}_{i=1}^q$ being an EFT-family and assume that there exists a weakly regular basic pair
$({\bf B}, {\bf P}):=(\{B_{1,\ell}\}, \{ {\bf P} \}_{k\ge1})_{\ell\in\Gamma}$. Let $\Omega=\bigcup_{i=1}^M \Omega_i$, $\Delta_\mu$ be the Dirichlet Laplacian defined by $\mu$, and let ${\bf M}(\boldsymbol\alpha;\infty)$ be defined as in \eqref{e:defi_M_alpha_infty}. Then ${\bf M}(\boldsymbol\alpha;\infty)$ is irreducible if and only if $G$ is strongly connected.
\end{prop}

\begin{proof}
[Proof of Theorem \ref{T:main_thm_str_con}]
(a) Obviously, $\eta=1$. It follows from \eqref{e:section_F_ell_alpha_m} that $F_\ell(\alpha)$ is a strictly increasing continuous positive function of $\alpha$ and
\begin{eqnarray}\label{eq:limits_F_ell_main_thm_str_con}
\lim_{\alpha\to-\infty} F_\ell(\alpha)=0\quad\mbox{ and }\quad\lim_{\alpha\to\infty} F_\ell(\alpha)=\infty.
\end{eqnarray}
Hence there exists a unique $\alpha$ such that the spectral radius of ${\bf M}(\boldsymbol\alpha; \infty)$ is equal to $1$ with $\boldsymbol\alpha:=(\alpha)$.

(b) Let $\alpha$ be the unique number in (a), and ${\bf m}:=[m^{(\alpha)}_{\ell \ell'}]=[\int_0^\infty x\,d\mu^{(\alpha)}]$ be the moment matrix. According to the proof of \cite[Theorem 1.1]{Ngai_2011}, we need to show that some moment condition holds, and it suffices to show that $0<\sum_{\ell'\in\Gamma} m^{(\alpha)}_{\ell \ell'}<\infty$.
By \eqref{eq:limits_F_ell_main_thm_str_con}, we get there exists $\epsilon>0$ such that $0<F_\ell(\alpha+\epsilon)<\infty$, and hence
\begin{eqnarray*}
\begin{aligned}
0<\sum_{\ell'\in\Gamma} m^{(\alpha)}_{\ell\ell'}
=\sum_{\ell'\in\Gamma}\sum_{k=2}^{\kappa_\ell}\sum_{j\in G'_{k,\ell}} w(k,\ell,j)^q \rho^{-(\alpha+\epsilon)}_{e(k,\ell,j)}\rho^\epsilon_{e(k,\ell,j)}|\ln(\rho_{e(k,\ell,j)})|
<\infty,
\end{aligned}
\end{eqnarray*}
where we use the fact that $\lim_{x\to 0^+} x^\epsilon \ln x=0$ in the last inequality.
Using \eqref{defi:muij_on_kuangjia}, we have
\begin{eqnarray*}
\sum_{\ell'\in\Gamma}\mu^{(\alpha)}_{\ell\ell'}(0)=0<\sum_{\ell'\in\Gamma}\mu^{(\alpha)}_{\ell\ell'}(\infty).
\end{eqnarray*}
It means that each column of ${\bf M}_{\boldsymbol\alpha}$ is nondegenerate at $0$. It follows from Theorem \ref{T:vector-valued_renewal_equation} that ${\bf f}={\bf f}\ast{\bf M}_{\boldsymbol\alpha}+{\bf z}$, where, by assumption, ${\bf z}$ is directly Riemann integrable on $\R$.

Since $G$ is strongly connected, using Proposition \ref{P:relation_strong_connectedness_and_irreducibility}, we get ${\bf M}(\boldsymbol\alpha,\infty)$ is irreducible. It follows from the above observations and \cite[Theorem 4.1]{Ngai_2011} that there exist constants $C_1, C_2>0$ such that
\begin{eqnarray*}
0<C_1\le \varlimsup_{x\to\infty} f_\ell(x)\le C_2<\infty\qquad\mbox{ for }\ell\in\Gamma.
\end{eqnarray*}
Combining this with \eqref{eq:defi_of_f_ell}, we get
\begin{eqnarray*}
0<C_1\le \varlimsup_{h\to 0^+} \Phi^{(\alpha)}_\ell(h)\le C_2<\infty.
\end{eqnarray*}
Thus
\begin{eqnarray*}
\begin{aligned}
0<C_1
\le \varlimsup_{h\to 0^+} \Phi^{(\alpha)}_\ell(h)
\le \varlimsup_{h\to 0^+}\sum_{\ell\in\Gamma}\Phi^{(\alpha)}_\ell(h)
\le \sum_{\ell\in\Gamma}\varlimsup_{h\to 0^+}\Phi^{(\alpha)}_\ell(h)
\le C_2\#\Gamma
<\infty,
\end{aligned}
\end{eqnarray*}
combining this with \eqref{eq:tau_q_on_rd_equ}, we get $\tau(q)=\alpha$.

(c) The proof follows \cite[Corollary A.2 and Theorem A.1]{Ngai-Xie_2020}.
\end{proof}

\begin{proof}
[Proof of Theorem \ref{T:main_thm_not_str_con}]
(a) The proof is similar to that of Theorem \ref{T:main_thm_str_con}(a).

(b) Since $G$ is not strongly connected, using Proposition \ref{P:relation_strong_connectedness_and_irreducibility}, we get ${\bf M}(\boldsymbol\alpha,\infty)$ is reducible. Let $\alpha:=\min\{\alpha_i: i=1,\ldots, \eta\}$. By the proof of \cite[Theorem 1.1(b), Case 2]{Ngai_2011}, we have
\begin{eqnarray*}
\lim_{x\to\infty} f_\ell^{(\beta)}(x)=0\qquad\mbox{ for all }\ell\in \Gamma_i,\ i\in \mathcal{SC}_m,\ m=1,\ldots, \eta\mbox{ and all }\beta<\alpha.
\end{eqnarray*}
Moreover, $\varliminf_{x\to\infty} f_{\ell_0}^{(\alpha)}(x)>0$ holds for some $\ell_0\in\Gamma$.
Combining these with \eqref{eq:defi_of_f_ell}, we get
\begin{eqnarray*}
\lim_{h\to 0^+} \Phi_\ell^{(\beta)}(x)=0\quad\mbox{ for all }\ell\in \Gamma_i, i\in \mathcal{SC}_m, m=1,\ldots, \eta\mbox{ and }\beta<\alpha,
\end{eqnarray*}
and
\begin{eqnarray*}
0<\varliminf_{h\to0^+} \Phi_{\ell_0}^{(\alpha)}(h)<\varlimsup_{h\to0^+}\sum_{\ell\in\Gamma} \Phi_{\ell_0}^{(\alpha)}(h),
\end{eqnarray*}
which, together with \eqref{eq:tau_q_on_rd_equ}, yields $\tau(q)\ge\alpha$ and $\tau(q)\le \alpha$. This completes the proof of (b).

(c) The proof is similar to that of Theorem \ref{T:main_thm_str_con}(c).
\end{proof}

\section{Strongly connected GIFSs on $\R$ and $\R^2$}\label{S:str_con_GIFSs}
\setcounter{equation}{0}

The goal of this section is to compute the $L^q$-spectrum of some graph-directed self-similar measures $\mu$ defined by
the strongly connected GIFSs on $\R$ and $\R^2$.

\subsection{A strongly connected GIFS on $\R$}\label{SS:str_con_GIFSs_R}
For the graph-directed self-similar measure $\mu$ defined by the strongly connected GIFS in Example \ref{E:exam_str_con_R},
Ngai and Xie \cite[Section 5.1]{Ngai-Xie_2020} have proved that $\mu$ satisfies (EFT). We state the result as follows and omit the details of proof.

\begin{prop}\label{P:str_con_R}(See \cite[Proposition 5.1]{Ngai-Xie_2020})
Let $\mu=\sum_{i=1}^2\mu_i$ be a graph-directed self-similar measure defined by a GIFS $G=(V,E)$ in Example \ref{E:exam_str_con_R} together with a probability matrix $(p_e)_{e\in E}$. Then $\mu$ satisfies (EFT) with $\{\Omega_i\}_{i=1}^2=\{(0,1), (0,1)\}$ being an EFT-family and there exists a weakly regular basic pair.
\end{prop}

Let $\{S_{e_i}\}_{i=1}^5$ be a GIFS in \eqref{e:exam_str_con_R_similitudes}, $(p_e)_{e\in E}$ be a probability matrix, and $\mu$ be the graph-directed self-similar measure defined as in Proposition \ref{P:str_con_R}. By \eqref{e:transition_probability} and \eqref{e:probability_measures}, we have $\sum_{i=1}^3 p_{e_i}=1$, $\sum_{i=4}^5 p_{e_i}=1$,
$\mu_1=\sum_{i=1}^2 p_{e_i}\mu_i\circ S^{-1}_{e_i}+p_{e_3}\mu_1\circ S^{-1}_{e_3}$, and
$\mu_2=p_{e_4}\mu_2\circ S^{-1}_{e_4}+ p_{e_5}\mu_1\circ S^{-1}_{e_5}$.

Define (see Figure \ref{F:fig2_str_con_R})
\begin{eqnarray}\label{eq:defi_B_1_ell_str_con_R}
B_{1,1}:=\bigcup_{i=1}^2 S_{e_i}(\Omega_i),\quad B_{1,2}:=S_{e_3}(\Omega_1),
\quad B_{1,3}:=S_{e_5}(\Omega_1),\quad B_{1,4}:=S_{e_4}(\Omega_2).
\end{eqnarray}

In the rest of this subsection, we use the notation defined as in Section \ref{S:Renewal_equation}.
Let $\Gamma=\{1,3,4\}$ and ${\bf B}:=\{B_{1,\ell}: \ell\in\Gamma\}$.
Then $\Gamma_1=\{1\}$ and $\Gamma_2=\{3,4\}$.
Denote
\begin{eqnarray}\label{eq:defi_B_k,ell_ell_str_con_R}
B_{2,1,j}=S_{e_1}(B_{1,j}),\quad B_{k,3,j}=S_{e_5e_3^{k-2}}(B_{1,j}), \quad B_{2,4,j}=S_{e_4}(B_{1,j+2})
\quad\mbox{ for }j=1,2\mbox{ and }k\ge 2,
\end{eqnarray}
and $B_{2,1,3}=S_{e_2}(B_{1,4})$. For $\ell\in\Gamma$,
let ${\bf P}_{1,\ell}:=\{B_{1,\ell}:\ell\in\Gamma\}$. Also, let
\begin{eqnarray*}
\begin{aligned}
{\bf P}_{2,1}=\{B_{2,1,j}:j=1,2,3\},\quad\mbox{ and }\quad
{\bf P}_{2,\ell}=\{B_{2,\ell,j}:j=1,2\}\qquad
\mbox{for }\ell=3,4
\end{aligned}
\end{eqnarray*}
(see Figure \ref{F:fig3_str_con_R}).

In the following, we first state some elementary properties that will be used to compute the $L^q$-spectrum of $\mu$.

\begin{lem}(See \cite[Lemma 5.3]{Ngai-Xie_2020})\label{L:equilant_GIFS_str_con_R_1}
Assume the hypotheses of Proposition \ref{P:str_con_R}. Then
\begin{enumerate}
\item[(a)] $\mu_1|_{B_{2,1,1}}=p_{e_1}\mu_1|_{B_{1,1}}\circ S^{-1}_{e_1}$;
\item[(b)] $\mu_1|_{B_{2,1,2}}=(p_{e_1e_3}+p_{e_2e_5})/{p_{e_5}}\cdot\mu_2|_{B_{1,3}}\circ (S_{e_1e_3}S^{-1}_{e_5})^{-1}$;
\item[(c)] $\mu_1|_{B_{2,1,3}}=p_{e_2}\mu_2|_{B_{1,4}}\circ S^{-1}_{e_2}$;
\item[(d)] $\mu_2|_{B_{k,3,j}}= P_{e_5e_3^{k-2}}\mu_1|_{B_{1,j}}\circ S^{-1}_{e_5e_3^{k-2}}$ for $j=1,2$ and $k\ge2$;
\item[(e)] $\mu_2|_{B_{2,4,j}}= p_{e_4}\mu_2|_{B_{1,j+2}}\circ S^{-1}_{e_4}$ for $j=1,2$.
\end{enumerate}
\end{lem}

It follows from Lemma \ref{L:equilant_GIFS_str_con_R_1} that
${\bf P}'_{2,\ell}={\bf P}_{2,\ell}, {\bf P}^{\ast}_{2,\ell}=\emptyset$ for $\ell=1,4$, and
${\bf P}'_{2,3}=\{B_{2,3,1}\}, {\bf P}^{\ast}_{2,3}=\{B_{2,3,2}\}$.
For $k\ge3$, define
\begin{eqnarray*}
{\bf P}_{k,3}={\bf P}'_{k-1,3}\cup \{B_{k,3,j}: j=1,2\}.
\end{eqnarray*}
Using Lemma \ref{L:equilant_GIFS_str_con_R_1}(d), we obtain
${\bf P}'_{k,3}={\bf P}'_{k-1,3}\cup\{B_{k,3,1}\}$ and ${\bf P}^\ast_{k,3}=\{B_{k,3,2}\}$. Hence $\Gamma'_1=\{1\}, \Gamma^{\ast}_1=\emptyset, \Gamma'_2=\{4\}, \Gamma^{\ast}_2=\{3\}, \kappa_\ell=2$ for $\ell=1,4$, and $\kappa_3=\infty$. Also,
$G'_{2,1}=\{1,2,3\}, G'_{2,4}=\{1,2\}$, $G^\ast_{2,\ell}=\emptyset$ for $\ell=1,4$, and $G'_{k,3}=\{1\}, G^\ast_{k,3}=\{2\}$ for $k\ge2$.

\begin{lem}
Let $w(k,\ell,j), c(k,\ell,j), e(k,\ell,j)$ and $\rho_{e(k,\ell,j)}$ be defined as in \eqref{eq:relationship_between_mui_and_B1c}. Then
\begin{enumerate}
\item[(a)] $w(2,1,1)=p_{e_1},\ c(2,1,1)=1,\ \rho_{e(2,1,1)}=\rho$;
\item[(b)] $w(2,1,2)=(p_{e_1e_3}+p_{e_2e_5})/{p_{e_5}},\ c(2,1,2)=3,\ \rho_{e(2,1,2)}=r$;
\item[(c)] $w(2,1,3)=p_{e_2},\ c(2,1,3)=4,\ \rho_{e(2,1,3)}=r$;
\item[(d)] $w(k,3,j)=p_{e_5e_3^{k-2}},\ c(k,3,j)=i,\ \rho_{e(2,5,j)}=\rho r^{k-2}$ for $k\ge2$ and $j=1,2$;
\item[(e)] $w(2,4,j)=p_{e_4},\ c(2,1,j)=j+2,\ \rho_{e(2,4,j)}=r$ for $j=1,2$.
\end{enumerate}
\end{lem}

\begin{center}
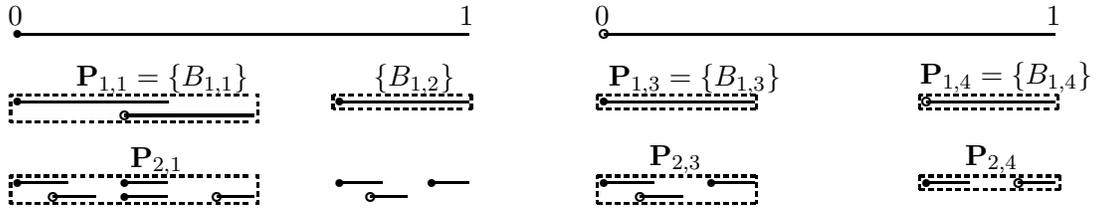
\begin{figure}[h]
\begin{picture}(400,90)
\unitlength=0.06cm
\thicklines
\put(0,42){\line(1,0){100}}
\put(-2,44){$0$}
\put(98,44){$1$}
\put(0,42){\circle*{1.5}}

\put(0,27){\line(1,0){33.33}}
\put(0,27){\circle*{1.5}}
\put(23.81,24){\line(1,0){28.57}}
\put(23.81,24){\circle{1.5}}
\put(13,30.5){$\mathbf{P}_{1,1}=\{B_{1,1}\}$}
\put(-1.5,22.5){\dashbox(54.88,6){}}

\put(71.43,27){\line(1,0){28.57}}
\put(71.43,27){\circle*{1.5}}
\put(79,30.5){$\{B_{1,2}\}$}
\put(70,25.5){\dashbox(30.57,3){}}

\put(0,9){\line(1,0){11.11}}
\put(0,9){\circle*{1.5}}
\put(7.94,6){\line(1,0){9.52}}
\put(7.94,6){\circle{1.5}}

\put(23.81,9){\line(1,0){9.52}}
\put(23.81,9){\circle*{1.5}}
\put(23.81,6){\line(1,0){9.52}}
\put(23.81,6){\circle*{1.5}}

\put(44.21,6){\line(1,0){8.163}}
\put(44.21,6){\circle{1.5}}
\put(-1.5,4.5){\dashbox(54.873,6){}}
\put(25,12.5){$\mathbf{P}_{2,1}$}

\put(71.43,9){\line(1,0){9.52}}
\put(71.43,9){\circle*{1.5}}
\put(78.23,6){\line(1,0){8.163}}
\put(78.23,6){\circle{1.5}}

\put(91.84,9){\line(1,0){8.163}}
\put(91.84,9){\circle*{1.5}}

\put(130,42){\line(1,0){100}}
\put(128,44){$0$}
\put(228,44){$1$}
\put(130,42){\circle{1.5}}

\put(130,27){\line(1,0){33.33}}
\put(130,27){\circle*{1.5}}
\put(131,30.5){$\mathbf{P}_{1,3}=\{B_{1,3}\}$}
\put(128.5,25.5){\dashbox(35.33,3){}}

\put(201.43,27){\line(1,0){28.57}}
\put(201.43,27){\circle{1.5}}
\put(200,31){$\mathbf{P}_{1,4}=\{B_{1,4}\}$}
\put(200,25.5){\dashbox(30.57,3){}}

\put(130,9){\line(1,0){11.11}}
\put(130,9){\circle*{1.5}}
\put(137.94,6){\line(1,0){9.52}}
\put(137.94,6){\circle{1.5}}

\put(153.81,9){\line(1,0){9.52}}
\put(153.81,9){\circle*{1.5}}
%
%
\put(128.5,4.5){\dashbox(35.33,6){}}
\put(140,13){$\mathbf{P}_{2,3}$}

\put(201.43,9){\line(1,0){9.52}}
\put(201.43,9){\circle*{1.5}}

\put(221.84,9){\line(1,0){8.163}}
\put(221.84,9){\circle{1.5}}
\put(199.93,7.5){\dashbox(31.07,3){}}
\put(210,13){$\mathbf{P}_{2,4}$}
\end{picture}

\caption{$\mu$-partitions $\mathbf{P}_{k,\ell}$, $k=1,2$, of $B_{1,\ell}$ for the GIFS defined in \eqref{e:exam_str_con_R_similitudes}.
The figure is drawn with $\rho=1/3$ and $r=2/7$.}
\label{F:fig2_str_con_R}
\end{figure}
\end{center}

\begin{center}
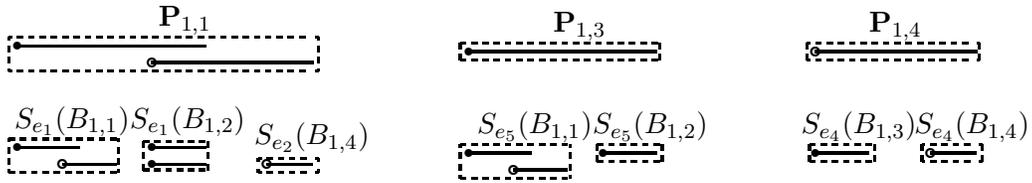
\begin{figure}[h]
\begin{picture}(330,60)
\unitlength=0.075cm
\thicklines
\put(-10,23){\line(1,0){33.33}}
\put(-10,23){\circle*{1.5}}
\put(13.81,20){\line(1,0){28.57}}
\put(13.81,20){\circle{1.5}}
\put(15,26.5){$\mathbf{P}_{1,1}$}
\put(-11.5,18.5){\dashbox(54.88,6){}}

\put(-10,5){\line(1,0){11.11}}
\put(-10,5){\circle*{1.5}}
\put(-2.06,2){\line(1,0){9.52}}
\put(-2.06,2){\circle{1.5}}
\put(-11.5,0.5){\dashbox(19.46,6){}}
\put(-10.5,8.5){$S_{e_1}(B_{1,1})$}

\put(13.81,5){\line(1,0){9.52}}
\put(13.81,5){\circle*{1.5}}
\put(13.81,2){\line(1,0){9.52}}
\put(13.81,2){\circle*{1.5}}
\put(12.31,1){\dashbox(11.52,5){}}
\put(10,8.5){$S_{e_1}(B_{1,2})$}

\put(34.21,2){\line(1,0){8.163}}
\put(34.21,2){\circle{1.5}}
\put(32.71,0.5){\dashbox(10.663,2.5){}}
\put(32,5.5){$S_{e_2}(B_{1,4})$}


\put(70,22){\line(1,0){33.33}}
\put(70,22){\circle*{1.5}}
\put(85,25.5){$\mathbf{P}_{1,3}$}
\put(68.5,20.5){\dashbox(35.33,3){}}

\put(70,4){\line(1,0){11.11}}
\put(70,4){\circle*{1.5}}
\put(77.94,1){\line(1,0){9.52}}
\put(77.94,1){\circle{1.5}}
\put(68.5,-0.5){\dashbox(19.46,6){}}
\put(71.5,7.5){$S_{e_5}(B_{1,1})$}

\put(93.81,4){\line(1,0){9.52}}
\put(93.81,4){\circle*{1.5}}
\put(92.81,2.5){\dashbox(11.52,3){}}
\put(92,7.5){$S_{e_5}(B_{1,2})$}

\put(131.43,22){\line(1,0){28.57}}
\put(131.43,22){\circle{1.5}}
\put(141,25.5){$\mathbf{P}_{1,4}$}
\put(130,20.5){\dashbox(30.57,3){}}

\put(131.43,4){\line(1,0){9.52}}
\put(131.43,4){\circle*{1.5}}
\put(130.43,2.5){\dashbox(11.52,3){}}
\put(129,7.5){$S_{e_4}(B_{1,3})$}

\put(151.84,4){\line(1,0){8.163}}
\put(151.84,4){\circle{1.5}}
\put(150.34,2.5){\dashbox(10.663,3){}}
\put(149,7.5){$S_{e_4}(B_{1,4})$}
\end{picture}

\caption{$\mu$-partitions $\mathbf{P}_{2,\ell}$ for $\ell=1,3,4$.}
\label{F:fig3_str_con_R}
\end{figure}
\end{center}

Using the above results, we can express the vector-valued renewal equations \eqref{eq:f_ell_'_1} and \eqref{eq:f_ell_ast_1} as follows
\begin{eqnarray*}
\begin{aligned}
f_1(x)
=&p_{e_1}^q \rho^{-\alpha} f_1(x+\ln \rho)
+((p_{e_1e_3}+p_{e_2e_5})/p_{e_5})^qr^{-\alpha} f_3(x+\ln r)+p_{e_2}^q r^{-\alpha} f_4(x+\ln r)+z_1^{(\alpha)}(x),\\
f_3(x)
=&\sum_{k=2}^\infty p_{e_5e_3^{k-2}}^q (\rho r^{k-2})^{-\alpha} f_1(x+\ln(\rho r^{k-2}))
+z_3^{(\alpha)}(x)-z_{3,\infty}^{(\alpha)}(x),\\
f_4(h)
=&p_{e_4}^q r^{-\alpha}\sum_{i=3}^4 f_i(x+\ln r)+z_4^{(\alpha)}(x),
\end{aligned}
\end{eqnarray*}
with $z_\ell^{(\alpha)}(x)=E_\ell^{(\alpha)}(e^{-x})$ for $\ell\in\Gamma$ and $z_{3,\infty}^{(\alpha)}(x)=E_{3,\infty}^{(\alpha)}(e^{-x})$.

Let $\mu^{(\alpha)}_{\ell\ell'}$, $\ell,\ell'\in\Gamma$, be the  discrete measures defined as in \eqref{defi:muij_on_kuangjia}. Then
\begin{eqnarray*}
\begin{aligned}
&\mu_{11}^{(\alpha)}(-\ln\rho)=p^q_{e_1}\rho^{-\alpha},\quad \mu_{13}^{(\alpha)}(-\ln r)=((p_{e_1e_3}+p_{e_2e_5})/p_{e_5})^q r^{-\alpha},\quad
\mu_{14}^{(\alpha)}(-\ln r)=p^q_{e_2}r^{-\alpha},\\
&\mu_{31}^{(\alpha)}(-\ln (\rho r^{k-2}))=p_{e_5e_3^k}^q(\rho r^{k-2})^{-\alpha}\quad\mbox{ for }k\ge2,\\
&\mu_{4\ell'}^{(\alpha)}(-\ln r)=p^q_{e_4}r^{-\alpha}\quad\mbox{ for }\ell'=3,4.
\end{aligned}
\end{eqnarray*}
Moreover,
\begin{eqnarray*}
{\bf M}(\boldsymbol\alpha;\infty)
=\left(
 \begin{array}{ccc}
 p^q_{e_1}\rho^{-\alpha} & ((p_{e_1e_3}+p_{e_2e_5})/p_{e_5})^q r^{-\alpha} & p^q_{e_2}r^{-\alpha}\\
 p^q_{e_5}\rho^{-\alpha}/(1-p^q_{e_3}r^{-\alpha}) & 0 & 0\\
 0 & p^q_{e_4}r^{-\alpha} & p^q_{e_4}r^{-\alpha}\\
\end{array}
\right).
\end{eqnarray*}

Finally, we give the results about the error terms, precisely as, $z_\ell^{(\alpha)}(x)=o(e^{-\epsilon x})$ and $z_{3,\infty}^{(\alpha)}(x)=o(e^{-\epsilon x})$ as $x\to\infty$, that is, $E_\ell^{(\alpha)}(h)=o(h^\epsilon)$ and $E_{3,\infty}^{(\alpha)}(h)=o(h^\epsilon)$ as $h\to 0$ for $\ell\in\Gamma$.
The proofs of the following results are similar to that of \cite[Propositions 4.1 and 4.4]{Ngai-Xie_2019}. We omit the details.
\begin{prop}\label{P:str_con_R_error}
For any $k\ge N+1$, $\Phi_1^{(\alpha)}(h/{\rho r^{k-2}})\le 1$.
\end{prop}

\begin{prop}\label{P:str_con_R_prop_3}
Assume that $\alpha\in D_\ell$ for $\ell\in \Gamma$. Then for $q\ge 0$, there exists $\epsilon>0$ such that
\begin{enumerate}
\item[(a)]
$\sum_{k=N+1}^\infty p_{e_5e_3^{k-2}}^q
    (\rho r^{k-2})^{-\alpha}\Phi_1^{(\alpha)}(h/{\rho r^{k-2}})=o(h^{\epsilon/2})$;
\item[(b)] $h^{-(1+\alpha)}\sum_{k=2}^N e_{k,3}=o(h^{\epsilon/2})$;
\item[(c)] $h^{-(1+\alpha)}\int_{B_{N,3,2}}\mu(B_h(x))^q\,dx=o(h^{\epsilon/2})$;
\item[(d)] $h^{-(1+\alpha)} e_{2,\ell}=o(h^{\epsilon/2})$ for $\ell=1,4$.
\end{enumerate}
\end{prop}

\begin{proof}[Proof of Corollary~\ref{C:cor_str_con_R}]
Theorem \ref{T:main_thm_str_con} and Proposition \ref{P:str_con_R_prop_3} imply that $\tau(q)=\alpha$.

Next, we study the differentiability of $\tau$. The proof is similar to that of \cite[Theorem 1.2]{Ngai-Xie_2019}.
Let $H_q(q,\alpha)$ and $H_\alpha(q,\alpha)$ be the partial derivative of $H(q,\alpha)$ with respect to $q$ and $\alpha$, respectively. Since
\begin{eqnarray*}
\begin{aligned}
H_q(q,\alpha)
=&-p_{e_4}^q r^{-\alpha}\ln p_{e_4}\cdot\Big((1-p_{e_1}^q\rho^{-\alpha})(1-p_{e_3}^qr^{-\alpha})-(p_{e_1e_3}+p_{e_2e_5})^q(\rho r)^{-\alpha}\Big)\\
&-(1-p_{e_4}^qr^{-\alpha})\Big(p_{e_1}^q\rho^{-\alpha}\ln p_{e_1}\cdot (1-p_{e_3}^q r^{-\alpha})
+p_{e_3}^q r^{-\alpha}\ln p_{e_3}\cdot (1-p_{e_1}^q\rho^{-\alpha})\Big)\\
&-(1-p_{e_4}^qr^{-\alpha})(p_{e_1e_3+p_{e_2e_5}})^q(\rho r)^{-\alpha}\ln(p_{e_1e_3}+p_{e_2e_5})
-p^q_{e_2e_4e_5}(\rho r^2)^{-\alpha}\ln p_{e_2e_4e_5},
\end{aligned}
\end{eqnarray*}
and
\begin{eqnarray*}
\begin{aligned}
H_\alpha(q,\alpha)
=&p_{e_4}^qr^{-\alpha} \ln r\cdot
\Big((1-p^q_{e_1}\rho^{-\alpha})(1-p^q_{e_3} r^{-\alpha})-(p_{e_1e_3}+p_{e_2e_5})^q (\rho r)^{-\alpha}\Big)\\
&+(1-p^q_{e_4} r^{-\alpha})\Big(p^q_{e_1}\rho^{-\alpha}\ln\rho\cdot(1-p^q_{e_3}r^{-\alpha})
+p^q_{e_3}r^{-\alpha}\ln r\cdot(1-p^q_{e_1}\rho^{-\alpha})\Big)\\
&+(1-p^q_{e_4} r^{-\alpha})(p_{e_1e_3}+p_{e_2e_5})^q(\rho r)^{-\alpha}\ln(\rho r)
+p^q_{e_2e_4e_5}(\rho r^2)^{-\alpha}\ln(\rho r^2),
\end{aligned}
\end{eqnarray*}
we get $H(q,\alpha)$ is $C^1$.

For any $(q, \alpha)\in (0,\infty)\times D_3$ satisfying $H(q,\alpha)=0$ and $H_\alpha(q,\alpha)\neq 0$, let $\{q_n\}_{n=1}^\infty$ be an increasing sequence of positive numbers such that $\lim\limits_{n\to\infty} q_n=q$ and that $\tau$ is differentiable at each $q_n$. Then
\begin{eqnarray}\label{eq:relation_about_H_q_and_H_alpha_str_con_R}
H_q(q_n,\alpha_n)+\alpha'(q_n)\cdot H_\alpha(q_n,\alpha_n)=0\qquad\mbox{for } n=1,2,\ldots.
\end{eqnarray}
This implies that
\begin{eqnarray*}
H_q(q,\alpha)+\alpha'_{-}(q)\cdot H_\alpha(q,\alpha)=0,
\end{eqnarray*}
where $\alpha'_{-}(q)$ is the left-hand derivative of $\alpha(q)(=\tau(q))$ at $q$.
It follows from the implicit function theorem that $\tau$ is differentiable at $q$. Moreover, \eqref{e:result_tau_q_1} holds.
This completes the proof.
\end{proof}

\subsection{A strongly connected GIFS on $\R^2$}\label{SS:str_con_GIFSs_R2}
In this subsection, we compute the $L^q$-spectrum of the graph-directed self-similar measure $\mu$,
which is defined by the strongly connected GIFS in Example \ref{E:exam_str_con_R2}.
In the following, we first show that $\mu$ satisfy (EFT) and there exists a weakly regular basic pair. The proof is similar to that of \cite[Proposition 5.1]{Ngai-Xie_2020}, and we omit the details.

\begin{prop}\label{P:strongly_conn_R2}
Let $\mu=\sum_{i=1}^2\mu_i$ be a graph-directed self-similar measure defined by a GIFS $G=(V,E)$ in Example \ref{E:exam_str_con_R2} together with a probability matrix $(p_e)_{e\in E}$. Then $\mu$ satisfies (EFT) with $\{\Omega_i\}_{i=1}^2=\{\bigcup_{x\in (0,1)}(0,1)\times (x,1-x), (0,1)\times(0,1)\}$ being an EFT-family and there exists a weakly regular basic pair.
\end{prop}

Let $\{S_{e_i}\}_{i=1}^8$ be defined as in \eqref{e:exam_str_con_R2_similitudes}, $(p_e)_{e\in E}$ be a probability matrix, and $\mu$ be a graph-directed self-similar measure defined as in Proposition \ref{P:strongly_conn_R2}.
It follows from \eqref{e:transition_probability} and \eqref{e:probability_measures} that
$\sum_{i=1}^4 p_{e_i}=1$, $\sum_{i=5}^8 p_{e_i}=1$, and
\begin{equation*}
\mu_1=\sum_{i=1}^3 p_{e_i}\mu_1\circ S^{-1}_{e_i}+ p_{e_4}\mu_2\circ S^{-1}_{e_4},\qquad
\mu_2=\sum_{i=5}^6 p_{e_i}\mu_1\circ S^{-1}_{e_i}+ \sum_{i=7}^8 p_{e_i}\mu_2\circ S^{-1}_{e_i}.
\end{equation*}
Denote (see Figure \ref{F:fig2_GIFS_strong_con_R2})
\begin{eqnarray}\label{e:defi_B_1_ell_str_conn_R2}
\begin{aligned}
&B_{1,1}:=S_{e_1}(\Omega_1)\cup S_{e_4}(\Omega_2),\qquad&& B_{1,\ell}:=S_{e_\ell}(\Omega_1)\quad\mbox{ for }\ell=2,3,\\
&B_{1,\ell}:=S_{e_{\ell+1}}(\Omega_1)\mbox{ for }\ell=4,5,\qquad&& B_{1,\ell}:=S_{e_{\ell+1}}(\Omega_2)\quad\mbox{ for }\ell=6,7.
\end{aligned}
\end{eqnarray}

Let $\Gamma=\{1,2,\ldots,7\}$ and ${\bf B}:=\{B_{1,\ell}: \ell\in\Gamma\}$. Then $\Gamma_1=\{1,2,3\}$ and $\Gamma_2=\{4,5,6,7\}$.
Define
\begin{eqnarray}\label{e:defi_B_k_ell_j_str_conn_R2}
B_{k,1,j}=S_{e^{k-1}_1}(B_{1,j}),\quad B_{k,1,j'}=S_{e^{k-2}_1e_4}(B_{1,j'})\quad\mbox{ for }k\ge2,~j=1,2,3\mbox{ and }j'=4,5,6,
\end{eqnarray}
and let
\begin{eqnarray}
\begin{aligned}\label{e:defi_B_2_ell_j_str_conn_R2}
&B_{2,\ell,j}=S_{e_\ell}(B_{1,j}),\qquad\mbox{ for }\ell=2,3\mbox{ and }j=1,2,3,\\
&B_{2,\ell,j}=S_{e_{\ell+1}}(B_{1,j}),\qquad\mbox{ for }\ell=4,5\mbox{ and }j=1,2,3,\\
&B_{2,\ell,j}=S_{e_{\ell+1}}(B_{1,j+3}),\qquad\mbox{ for }\ell=6,7\mbox{ and }j=1,2,3,4.
\end{aligned}
\end{eqnarray}
Also, define ${\bf P}_{1,\ell}:=\{B_{1,\ell}\}$ for $\ell\in\Gamma$, and ${\bf P}_{2,1}=\{B_{2,1,j}: j=1,2,3,4,5,6\}$,
\begin{eqnarray*}
{\bf P}_{2,\ell}=\{B_{2,\ell,j}: j=1,2,3\},
\quad{\bf P}_{2,\ell'}=\{B_{2,\ell',j}: j=1,2,3,4\}\quad\mbox{ for }\ell=2,3,4,5\mbox{ and }\ell'=6,7.
\end{eqnarray*}

In order to compute the $L^q$-spectrum of $\mu$, we firstly state some elementary properties.
\begin{lem}
Let $\{S_{e_i}\}_{i=1}^8$ be defined as in \eqref{e:exam_str_con_R2_similitudes}. Then
$S_{e_1 e_4}=S_{e_4e_8}$.
\end{lem}

\begin{lem}\label{L:mu_equi_not_str_conn_R2}
Let $\{B_{1,\ell}\}_{\ell=1}^7$ be defined as in \eqref{e:defi_B_1_ell_str_conn_R2}, and
\begin{equation}\label{e:defi_S(k)}
w(k):=p_{e_4}\sum_{i=0}^k p_{e_1}^ip_{e_8}^{k-i}\qquad\mbox{ for }k\ge0.
\end{equation}
Then
\begin{enumerate}
\item[(a)] $\mu_1|_{B_{2,\ell,j}}=p_{e_\ell}\mu_1|_{B_{1,j}}\circ S^{-1}_{e_\ell}$ for $\ell=2,3,\ j=1,2,3$;
\item[(b)] $\mu_2|_{B_{2,\ell,j}}=p_{e_{\ell+1}}\mu_1|_{B_{1,j}}\circ S^{-1}_{e_{\ell+1}}$ for $\ell=4,5,\ j=1,2,3$;
\item[(c)] $\mu_2|_{B_{2,\ell,j}}=p_{e_{\ell+1}}\mu_2|_{B_{1,j+3}}\circ S^{-1}_{e_{\ell+1}}$ for $\ell=6,7,\ j=1,2,3,4$;
\item[(d)] for $k\ge2$,
\begin{eqnarray*}
\begin{aligned}
\mu_1|_{B_{k,1,1}}=&p_{e_1^{k-1}}\mu_1|_{B_{1,1}}\circ S^{-1}_{e_1^{k-1}}+w(k-2)\mu_2|_{B_{1,7}}\circ S^{-1}_{e_1^{k-2}e_4},\\
\mu_1|_{B_{k,1,j}}=&p_{e_1^{k-1}}\mu_1|_{B_{1,j}}\circ S^{-1}_{e_1^{k-1}}\quad\mbox{ for }j=2,3,\\
\mu_1|_{B_{k,1,j}}=&w(k-2)\mu_2|_{B_{1,j}}\circ S^{-1}_{e_1^{k-2}e_4}\quad\mbox{ for }j=4,5,6.
\end{aligned}
\end{eqnarray*}
\end{enumerate}
\end{lem}

\begin{center}
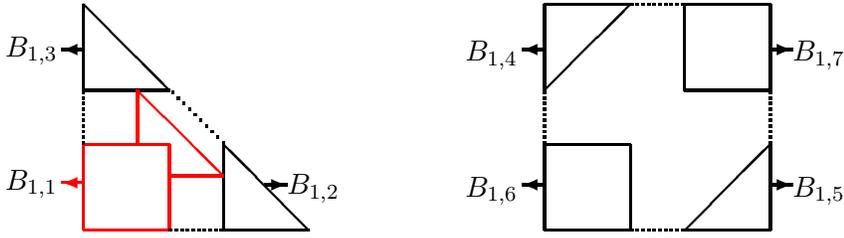
\begin{figure}[h]

\begin{picture}(500,90)
\unitlength=0.3cm

\thicklines
{\color{red}
\put(10,1){\line(0,1){3.8}}
\put(10,1){\line(1,0){3.8}}
\put(10,4.8){\line(1,0){3.8}}
\put(13.8,1){\line(0,1){3.8}}}

\put(16.2,1){\line(1,0){3.8}}
\put(16.2,1){\line(0,1){3.8}}
\put(20,1){\line(-1,1){3.8}}

\put(10,7.2){\line(0,1){3.8}}
\put(10,11){\line(1,-1){3.8}}
\put(10,7.2){\line(1,0){3.8}}
\multiput(13.8,1)(0.25,0){10}{\line(1,0){0.125}}
\multiput(10,4.8)(0,0.25){10}{\line(0,1){0.125}}
\multiput(16.2,4.8)(-0.25,0.25){10}{\line(0,1){0.125}}

{\color{red}
\put(13.8,3.4){\line(1,0){2.4}}
\put(12.4,4.8){\line(0,1){2.4}}
\put(12.4,7.2){\line(1,-1){3.8}}}

{\color{red}
\put(9.6,3.1){\vector(-1,0){1}}}
\put(6.1,2.8){$B_{1,1}$}
\put(17.6,3){\vector(1,0){1}}
\put(18.6,2.6){$B_{1,2}$}
\put(9.6,9){\vector(-1,0){1}}
\put(6.1,8.7){$B_{1,3}$}

\put(30,1){\line(1,0){3.8}}
\put(30,1){\line(0,1){3.8}}
\put(33.8,1){\line(0,1){3.8}}
\put(30,4.8){\line(1,0){3.8}}

\put(40,11){\line(0,-1){3.8}}
\put(40,11){\line(-1,0){3.8}}
\put(36.2,7.2){\line(0,1){3.8}}
\put(36.2,7.2){\line(1,0){3.8}}

\put(36.2,1){\line(1,0){3.8}}
\put(36.2,1){\line(1,1){3.8}}
\put(40,1){\line(0,1){3.8}}

\put(30,7.2){\line(0,1){3.8}}
\put(30,11){\line(1,0){3.8}}
\put(30,7.2){\line(1,1){3.8}}

\multiput(33.8,1)(0.25,0){16}{\line(1,0){0.125}}
\multiput(33.8,11)(0.25,0){16}{\line(1,0){0.125}}
\multiput(30,4.8)(0,0.25){16}{\line(0,1){0.125}}
\multiput(40,4.8)(0,0.25){16}{\line(0,1){0.125}}

\put(30,3){\vector(-1,0){1}}
\put(26.5,2.5){$B_{1,6}$}
\put(40,9){\vector(1,0){1}}
\put(41,8.5){$B_{1,7}$}
\put(30,9){\vector(-1,0){1}}
\put(26.5,8.5){$B_{1,4}$}
\put(40,3){\vector(1,0){1}}
\put(41,2.5){$B_{1,5}$}
\end{picture}

\caption{The first iteration of the GIFS defined in \eqref{e:exam_str_con_R2_similitudes}, where $\Omega_1=\bigcup_{x\in(0,1)}(0,1)\times(x,1-x)$ and $\Omega_2=(0,1)\times(0,1)$.}\label{F:fig2_GIFS_strong_con_R2}
\end{figure}
\end{center}

By Lemma \ref{L:mu_equi_not_str_conn_R2}, we get
\begin{eqnarray*}
\begin{aligned}
&{\bf P}'_{2,1}=\{B_{2,1,j}: j=2,3,4,5,6\}, \quad&&{\bf P}^\ast_{2,1}=\{B_{2,1,1}\},\\
&{\bf P}'_{2,\ell}={\bf P}_{2,\ell}, \quad&&{\bf P}^\ast_{2,\ell}=\emptyset\quad\mbox{ for }\ell\in\Gamma\backslash\{1\}.
\end{aligned}
\end{eqnarray*}
For $k\ge3$, define
\begin{eqnarray*}
{\bf P}_{k,1}={\bf P}'_{k,1}\cup \{B_{k,1,j}: j=1,2,3,4,5,6\}.
\end{eqnarray*}
It follows from Lemma \ref{L:mu_equi_not_str_conn_R2}(d) that for $k\ge 2$,
\begin{eqnarray*}
{\bf P}'_{k,1}={\bf P}'_{k-1,1}\cup\{B_{k,1,j}: j=2,3,4,5,6\}\quad\mbox{and}\quad{\bf P}^\ast_{k,1}=\{B_{k,1,1}\}.
\end{eqnarray*}
Hence $\Gamma'_1=\{2,3\}$, $\Gamma^\ast_1=\{1\}$, $\Gamma'_2=\Gamma_2$, $\Gamma^\ast_2=\emptyset$, $\kappa_1=\infty$, and $\kappa_\ell=2$ for $\ell\in\Gamma\backslash\{1\}$. Also,
\begin{eqnarray*}
\begin{aligned}
&G'_{2,\ell}=\{1,2,3\},\quad&& G^\ast_{2,\ell}=\emptyset\quad\mbox{ for }\ell=2,3,4,5,\\
&G'_{2,\ell'}=\{1,2,3,4\},\quad&& G^\ast_{2,\ell'}=\emptyset\quad\mbox{ for }\ell'=6,7,\\
&G'_{k,1}=\{2,3,4,5,6\}, \quad&& G^\ast_{k,1}=\{1\}\quad\mbox{ for }k\ge2.
\end{aligned}
\end{eqnarray*}

\begin{lem}
Let $w(k,\ell,j), c(k,\ell,j), e(k,\ell,j)$ and $\rho_{e(k,\ell,j)}$ be defied as in \eqref{eq:relationship_between_mui_and_B1c}. Then
\begin{enumerate}
\item[(a)] $w(2,\ell,j)=p_{e_\ell},\ c(2,\ell,j)=j,\ \rho_{e(2,\ell,j)}=\rho^2$ for $\ell=2,3$ and $j=1,2,3$;
\item[(b)] $w(2,\ell,j)=p_{e_{\ell+1}},\ c(2,\ell,j)=j,\ \rho_{e(2,\ell,j)}=\rho^2$ for $\ell=4,5$ and $j=1,2,3$;
\item[(c)] $w(2,\ell,j)=p_{e_{\ell+1}},\ c(2,\ell,j)=j+3,\ \rho_{e(2,\ell,j)}=\rho^2$ for $\ell=6,7$ and $j=1,2,3$;
\item[(d)] $w(k,1,j)=p_{e^{k-1}_1},\ c(k,1,j)=j,\ \rho_{e(k,1,j)}=\rho^{2(k-1)}$ for $k\ge2$ and $j=2,3$;
\item[(e)] $w(k,1,j)=w(k-2),\ c(k,1,j)=j,\ \rho_{e(k,1,j)}=\rho^{2(k-1)}$ for $k\ge2$ and $j=4,5,6$.
\end{enumerate}
\end{lem}

Based on the above results, we can express the vector-valued renewal equations \eqref{eq:f_ell_'_1} and \eqref{eq:f_ell_ast_1} as follows
\begin{eqnarray*}
\begin{aligned}
f_1(x)=&\sum_{k=2}^\infty\Big(\big(p_{e_1}^{q}\rho^{-2\alpha}\big)^{k-1}\sum_{j=2}^3 f_j\big(x+\ln(\rho^{2(k-1)})\big)
+\sum_{k=2}^\infty w^q(k-2)\rho^{-2\alpha(k-1)}\sum_{j=4}^6 f_j\big(x+\ln (\rho^{2(k-1)})\big)\Big)\\
&+z_1^{(\alpha)}(x)-z_{1,\infty}^{(\alpha)}(x),\\
f_\ell(x)=&p^q_{e_\ell}\rho^{-2\alpha}\sum_{j=1}^3 f_j\big(x+\ln(\rho^2)\big)+z_\ell^{(\alpha)}(x)\qquad\mbox{ for }\ell=2,3,\\
f_\ell(x)=&p^q_{e_{\ell+1}}\rho^{-2\alpha}\sum_{j=1}^3 f_j\big(x+\ln(\rho^2)\big),+z_\ell^{(\alpha)}(x)\qquad\mbox{ for }\ell=4,5,\\
f_\ell(x)=&p^q_{e_{\ell+1}}\rho^{-2\alpha}\sum_{j=4}^7 f_j\big(x+\ln(\rho^2)\big)+z_\ell^{(\alpha)}(x)\qquad\mbox{ for }\ell=6,7,
\end{aligned}
\end{eqnarray*}
where $z_{1,\infty}^{(\alpha)}(x)=E_{1,\infty}^{(\alpha)}(e^{-x})$, and $z^{(\alpha)}_\ell(x)=E_\ell^{(\alpha)}(e^{-x})$
for $\ell\in\Gamma$. For $\ell, \ell'\in\Gamma$, let $\mu^{(\alpha)}_{\ell \ell'}$ be the discrete measures defined as in \eqref{defi:muij_on_kuangjia}. Then
\begin{eqnarray*}
\begin{aligned}
&\mu_{1\ell'}^{(\alpha)}(-\ln(\rho^{2(k-1)}))=\big(p_{e_1}^{q}\rho^{-2\alpha}\big)^{k-1}
\qquad\mbox{for }k\ge 2 \mbox{ and }\ell'=2,3,\\
&\mu_{1\ell'}^{(\alpha)}(-\ln(\rho^{2(k-1)}))=w^q(k-2)\rho^{-2\alpha(k-1)}\qquad\mbox{for }k\ge2\mbox{ and }\ell'=4,5,6,\\
&\mu_{\ell\ell'}^{(\alpha)}(-\ln(\rho^2))=p_{e_\ell}^{q}\rho^{-2\alpha}\qquad\mbox{for }\ell=2,3\mbox{ and }\ell'=1,2,3,\\
&\mu_{\ell\ell'}^{(\alpha)}(-\ln(\rho^2))=p_{e_{\ell+1}}^{q}\rho^{-2\alpha}\qquad\mbox{for }\ell=4,5\mbox{ and }\ell'=1,2,3,\\
&\mu_{\ell\ell'}^{(\alpha)}(-\ln(\rho^2))=p_{e_{\ell+1}}^{q}\rho^{-2\alpha}\qquad\mbox{for }\ell=6,7\mbox{ and }\ell'=4,5,6,7.
\end{aligned}
\end{eqnarray*}
Also,
\begin{eqnarray*}
{\bf M}(\boldsymbol\alpha; \infty)=
\left(
  \begin{array}{ccccccc}
    0 & \zeta_1 & \zeta_1 & \zeta_2 & \zeta_2 & \zeta_2 & 0 \\
    \zeta_3 & \zeta_3 & \zeta_3 & 0 & 0 & 0 & 0 \\
    \zeta_4 & \zeta_4 & \zeta_4 & 0 & 0 & 0 & 0 \\
    \zeta_5 & \zeta_5 & \zeta_5 & 0 & 0 & 0 & 0 \\
    \zeta_6 & \zeta_6 & \zeta_6 & 0 & 0 & 0 & 0 \\
    0 & 0 & 0 & \zeta_7 & \zeta_7 & \zeta_7 & \zeta_7 \\
    0 & 0 & 0 & \zeta_8 & \zeta_8 & \zeta_8 & \zeta_8 \\
  \end{array}
\right),
\end{eqnarray*}
where
\begin{eqnarray*}
\begin{aligned}
&\zeta_1=p^q_{e_1}\rho^{-2\alpha}/(1-p^q_{e_1}\rho^{-2\alpha}),\qquad\zeta_2=\sum_{k=0}^\infty w^q(k)\rho^{-2\alpha(k+1)},\\
&\zeta_j=p^q_{e_{j-1}}\rho^{-2\alpha},\qquad\zeta_{j'}=p^q_{e_{j'}}\rho^{-2\alpha}\qquad\mbox{ for }j=3,4\mbox{ and }j'=5,6,7,8.
\end{aligned}
\end{eqnarray*}

Now, we show that the error terms $z_{1,\infty}^{(\alpha)}(x)=o(e^{-\epsilon x})$ and $z_\ell^{(\alpha)}(x)=o(e^{-\epsilon x})$, $\ell\in\Gamma$, as $x\to\infty$. It is equivalent to prove that $E_{1,\infty}^{(\alpha)}(h)=o(h^\epsilon)$ and $E_\ell^{(\alpha)}=o(h^\epsilon)$ as $h\to 0$ for $\ell\in\Gamma$. The proofs of the following results are similar to that of \cite[Propositions 4.1, 4.3 and 4.4]{Ngai-Xie_2019}. The details are omitted.
\begin{prop}\label{P:str_con_R_D1_open}
For $q\ge 0$, let $D_1$ be defined as in \eqref{e:defi_F_ell_and_D_ell}. Then $D_1$ is open.
\end{prop}
\begin{prop}\label{P:str_con_R2_error_Phi}
For any $k\ge N+1$, $\Phi_j^{(\alpha)}(h/{\rho^{2(k-1)}})\le 1$ for $j=2,3,4,5,6$.
\end{prop}
\begin{prop}\label{P:str_con_R2_error_esti}
For $q\ge0$, assume that $\alpha\in D_\ell$ for $\ell\in\Gamma$. Then there exists some $\epsilon>0$ such that
\begin{enumerate}
\item[(a)]
$\sum\limits_{k=N+1}^\infty \Big(\big(p_{e_1}^{q}\rho^{-2\alpha}\big)^{k-1}
\sum\limits_{j=2}^3\Phi_j^{(\alpha)}(h/{\rho^{2(k-1)}})+w^q(k-2)\rho^{-2\alpha(k-1)}
\sum\limits_{j=4}^6\Phi^{(\alpha)}_j(h/{\rho^{2(k-1)}})\Big)=o(h^{\epsilon/2})$;
\item[(b)] $h^{-(2+\alpha)}\sum\limits_{k=2}^N e_{k,1}=o(h^{\epsilon/2})$;
\item[(c)] $h^{-(2+\alpha)}\int_{B_{N,1,1}}\mu(B_h(x))^q\,dx=o(h^{\epsilon/2})$;
\item[(d)] $h^{-(2+\alpha)} e_{2,\ell}=o(h^{\epsilon/2})$ for $\ell\in\Gamma\backslash\{1\}$.
\end{enumerate}
\end{prop}

\begin{proof}
(a)
The proof is similar to that of \cite[Proposition 4.4(1)]{Ngai-Xie_2019}.
Proposition \ref{P:str_con_R_D1_open} implies that there exists $\epsilon>0$ sufficiently small such that $\alpha+\epsilon\in D_1$, and so, there exists a constant $C>0$ such that
\begin{eqnarray*}
\sum_{k=1}^\infty\big(p_{e_1}^q{\rho^{-2(\alpha+\epsilon)}}\big)^k+\sum_{k=0}^\infty w^q(k)\rho^{-2(\alpha+\epsilon)(k+1)}\le C.
\end{eqnarray*}
Hence
\begin{equation*}
\sum_{k=N+1}^\infty \big(p_{e_1}^q{\rho^{-2\alpha}}\big)^k\le C\rho^{2\epsilon k}\le Ch^\epsilon
\quad\mbox{and}\quad \sum_{k=N+1}^\infty w^q(k)\rho^{-2\alpha(k+1)}\le C\rho^{2\epsilon(k+1)}\le Ch^\epsilon.
\end{equation*}
Combining these with Proposition \ref{P:str_con_R2_error_Phi}, we have
\begin{equation*}
\sum_{k=N+1}^\infty\Big(p_{e_1}^{qk}\rho^{-2\alpha k}\sum_{j=2}^3\Phi_j^{(\alpha)}(h/{\rho^{2(k-1)}})
+w^q(k-2)\rho^{-2\alpha(k-1)}\sum\limits_{j=4}^6\Phi^{(\alpha)}_j(h/{\rho^{2(k-1)}})\Big)
\le C h^\epsilon.
\end{equation*}

(b) From the proof of (a), we get
\begin{eqnarray*}
\begin{aligned}
\big(p_{e_1}^q{\rho^{-2\alpha}}\big)^{N+1}\le Ch^\epsilon\quad\mbox{and}\quad
w^q(N+1)\rho^{-2\alpha(N+2)}\le Ch^\epsilon.
\end{aligned}
\end{eqnarray*}
Hence
\begin{eqnarray}\label{e:GIFS_str_con_R2_fact_est}
p_{e_1}^{q(N+1)}\le C h^{\alpha+\epsilon} \quad
\quad\mbox{ and }\quad w^q(N+1)\le C h^{\alpha+\epsilon}\quad\mbox{ for } \alpha\in\R.
\end{eqnarray}

It is easy to see that (b) is equivalent to $e_{k,1}=o(h^{2+\alpha+\epsilon/2})$ for any $2\le k\le N$.
Hence, we only need to prove that for any $2\le k\le N$,
\begin{eqnarray}\label{e:GIFS_str_con_R2_B_k1j_est}
\int_{\widehat{B}_{k,1,j}(h)}\mu(B_h({\bf x}))^q\,d{\bf x}=o(h^{2+\alpha+\epsilon/2})\quad\mbox{ for }j=2,3,4,5,6,
\end{eqnarray}
\begin{eqnarray}\label{e:GIFS_str_con_R2_B_1j_23_est}
(p_{e_1}^{q}\rho^{4})^{k-1}\int_{\widehat{B}_{1,j}(h/{\rho^{2(k-1)}})}\mu(B_{h/{\rho^{2(k-1)}}}({\bf x}))^q\,d{\bf x}
=o(h^{2+\alpha+\epsilon/2})\quad\mbox{ for }j=2,3,
\end{eqnarray}
and
\begin{eqnarray}\label{e:GIFS_str_con_R2_B_1j_456_est}
w^q(k-2)\rho^{4(k-1)}\int_{\widehat{B}_{1,j}(h/{\rho^{2(k-1)}})}\mu(B_{h/{\rho^{2(k-1)}}}({\bf x}))^q\,d{\bf x}
=o(h^{2+\alpha+\epsilon/2})\quad\mbox{ for }j=4,5,6.
\end{eqnarray}
The proof of (b) consists of the following three parts:

\noindent\textit{Part 1. The proof of \eqref{e:GIFS_str_con_R2_B_k1j_est}.}

Using \eqref{e:defi_B_1_ell_str_conn_R2} and \eqref{e:defi_B_k_ell_j_str_conn_R2}, we get for $j=2,3$ and $j'=4,5$,
\begin{eqnarray*}
\begin{aligned}
&B_{k,1,j}=S_{e_1^{k-1}e_j}(\Omega_1),\quad B_{k,1,j'}=S_{e_1^{k-2}e_4e_{j'+1}}(\Omega_1),
\quad B_{k,1,6}=S_{e_1^{k-2}e_4e_7}(\Omega_2),
\end{aligned}
\end{eqnarray*}
and so (see Figure \ref{fig:fig1_of_str_GIFS_on R2} and \ref{fig:fig2_of_str_GIFS_on R2})
\begin{eqnarray*}
\int_{\widehat{B}_{k,1,j}(h)}\mu(B_h({\bf x}))^q\,d{\bf x}
=\sum_{i=1}^5\int_{\widehat{U}_{k,1,j,i}(h)}\mu(B_h({\bf x}))^q\,dx\quad\mbox{ for }j=2,3,4,5,
\end{eqnarray*}
and
\begin{eqnarray*}
\int_{\widehat{B}_{k,1,6}(h)}\mu(B_h({\bf x}))^q\,d{\bf x}
=\sum_{i=1}^4\int_{\widehat{U}_{k,1,6,i}(h)}\mu(B_h({\bf x}))^q\,d{\bf x}.
\end{eqnarray*}

\begin{center}
\begin{figure}
\begin{picture}(180,110)
\unitlength=0.18cm
\thicklines

\put(6,0){\line(0,1){21}}
\put(6,0){\line(1,0){21}}

\put(9,3){\line(0,1){10.75}}
\put(9,3){\line(1,0){10.75}}

\put(6,21){\line(1,-1){21}}
\put(9,13.75){\line(1,-1){10.75}}

\put(9,3.75){$\widetilde{B}_{k,1,j}(h)$}

\multiput(6.25,1)(1,0){20}{\line(1,0){0.5}}
\multiput(6.25,2)(1,0){19}{\line(1,0){0.5}}
\multiput(6.25,3)(1,0){3}{\line(1,0){0.5}}
\multiput(6.25,4)(1,0){3}{\line(1,0){0.5}}
\multiput(6.25,5)(1,0){3}{\line(1,0){0.5}}
\multiput(6.25,6)(1,0){3}{\line(1,0){0.5}}
\multiput(6.25,7)(1,0){3}{\line(1,0){0.5}}
\multiput(6.25,8)(1,0){3}{\line(1,0){0.5}}
\multiput(6.25,9)(1,0){3}{\line(1,0){0.5}}
\multiput(6.25,10)(1,0){3}{\line(1,0){0.5}}
\multiput(6.25,11)(1,0){3}{\line(1,0){0.5}}
\multiput(6.25,12)(1,0){3}{\line(1,0){0.5}}
\multiput(6.25,13)(1,0){3}{\line(1,0){0.5}}
\multiput(6.25,14)(1,0){7}{\line(1,0){0.5}}
\multiput(6.25,15)(1,0){6}{\line(1,0){0.5}}
\multiput(6.25,16)(1,0){5}{\line(1,0){0.5}}
\multiput(6.25,17)(1,0){4}{\line(1,0){0.5}}
\multiput(6.25,18)(1,0){3}{\line(1,0){0.5}}
\multiput(6.25,19)(1,0){2}{\line(1,0){0.5}}
\multiput(6.25,20)(1,0){1}{\line(1,0){0.5}}

\multiput(10,13)(1,0){4}{\line(1,0){0.5}}
\multiput(11,12)(1,0){4}{\line(1,0){0.5}}
\multiput(12,11)(1,0){4}{\line(1,0){0.5}}
\multiput(13,10)(1,0){4}{\line(1,0){0.5}}
\multiput(14,9)(1,0){4}{\line(1,0){0.5}}
\multiput(15,8)(1,0){4}{\line(1,0){0.5}}
\multiput(16,7)(1,0){4}{\line(1,0){0.5}}
\multiput(17,6)(1,0){4}{\line(1,0){0.5}}
\multiput(18,5)(1,0){4}{\line(1,0){0.5}}
\multiput(19,4)(1,0){4}{\line(1,0){0.5}}
\multiput(20,3)(1,0){4}{\line(1,0){0.5}}

{\color{red}
\multiput(19.75,0)(1,0){1}%
{\line(0,1){3}}

\multiput(19.75,3)(1,1){1}%
{\line(1,1){2.1}}

\multiput(6,3)(1,0){1}%
{\line(1,0){3}}

\multiput(9,13.75)(1,0){1}%
{\line(-1,0){3}}

\multiput(9,13.75)(1,0){1}%
{\line(1,1){2.1}}
}

\put(5.9,1.5){\vector(-1,0){3}}
\put(-6.5,1){$\widehat{U}_{k,1,j,1}(h)$}
\put(24.9,2){\vector(1,0){3}}
\put(27.8,1.5){$\widehat{U}_{k,1,j,2}(h)$}
\put(15.3,11.2){\vector(1,0){3}}
\put(18.5,10.25){$\widehat{U}_{k,1,j,3}(h)$}
\put(5.9,16){\vector(-1,0){3}}
\put(-6.5,15.5){$\widehat{U}_{k,1,j,4}(h)$}
\put(5.9,9.5){\vector(-1,0){3}}
\put(-6.5,9){$\widehat{U}_{k,1,j,5}(h)$}

\put(19.9,4.1){\vector(1,0){6}}
\put(26.45,3.55){$h$}

\end{picture}
\caption{The middle part and shaded region are $\widetilde{B}_{k,1,j}(h)$ and $\widehat{B}_{k,1,j}(h):=\cup_{i=1}^5\widehat{U}_{k,1,j,i}(h)$, respectively. The union is $B_{k,1,j}$ for $j=2,3,4,5$.}
\label{fig:fig1_of_str_GIFS_on R2}
 \end{figure}
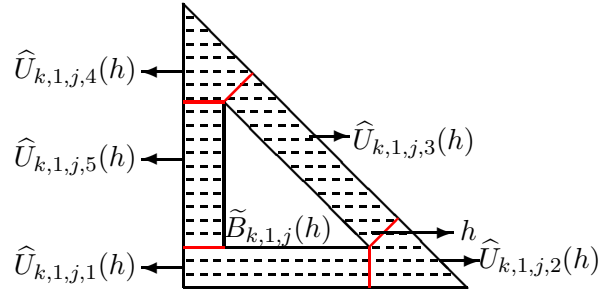
\end{center}

\begin{center}
\begin{figure}
\begin{picture}(180,110)
\unitlength=0.18cm
\thicklines
\put(6,1.5){\vector(-1,0){3}}
\put(27,10){\vector(1,0){3}}
\put(27,19.5){\vector(1,0){3}}
\put(6,11){\vector(-1,0){3}}

\put(-6.5,0.75){$\widehat{U}_{k,1,6,1}(h)$}
\put(30.5,9.25){$\widehat{U}_{k,1,6,2}(h)$}
\put(30.5,18.75){$\widehat{U}_{k,1,6,3}(h)$}
\put(-6.5,10.25){$\widehat{U}_{k,1,6,4}(h)$}

\put(27,21){\line(-1,0){21}}
\put(6,0){\line(0,1){21}}
\put(6,0){\line(1,0){21}}
\put(27,0){\line(0,1){21}}

\put(9,3){\line(0,1){15}}
\put(9,3){\line(1,0){15}}
\put(24,18){\line(0,-1){15}}
\put(24,18){\line(-1,0){15}}

\put(13,10){$\widetilde{B}_{k,1,6}(h)$}

\multiput(6.25,1)(1,0){21}{\line(1,0){0.5}}
\multiput(6.25,2)(1,0){21}{\line(1,0){0.5}}
\multiput(6.25,3)(1,0){3}{\line(1,0){0.5}}
\multiput(6.25,4)(1,0){3}{\line(1,0){0.5}}
\multiput(6.25,5)(1,0){3}{\line(1,0){0.5}}
\multiput(6.25,6)(1,0){3}{\line(1,0){0.5}}
\multiput(6.25,7)(1,0){3}{\line(1,0){0.5}}
\multiput(6.25,8)(1,0){3}{\line(1,0){0.5}}
\multiput(6.25,9)(1,0){3}{\line(1,0){0.5}}
\multiput(6.25,10)(1,0){3}{\line(1,0){0.5}}
\multiput(6.25,11)(1,0){3}{\line(1,0){0.5}}
\multiput(6.25,12)(1,0){3}{\line(1,0){0.5}}
\multiput(6.25,13)(1,0){3}{\line(1,0){0.5}}
\multiput(6.25,14)(1,0){3}{\line(1,0){0.5}}
\multiput(6.25,15)(1,0){3}{\line(1,0){0.5}}
\multiput(6.25,16)(1,0){3}{\line(1,0){0.5}}
\multiput(6.25,17)(1,0){3}{\line(1,0){0.5}}
\multiput(6.25,18)(1,0){3}{\line(1,0){0.5}}
\multiput(24.25,3)(1,0){3}{\line(1,0){0.5}}
\multiput(24.25,4)(1,0){3}{\line(1,0){0.5}}
\multiput(24.25,5)(1,0){3}{\line(1,0){0.5}}
\multiput(24.25,6)(1,0){3}{\line(1,0){0.5}}
\multiput(24.25,7)(1,0){3}{\line(1,0){0.5}}
\multiput(24.25,8)(1,0){3}{\line(1,0){0.5}}
\multiput(24.25,9)(1,0){3}{\line(1,0){0.5}}
\multiput(24.25,10)(1,0){3}{\line(1,0){0.5}}
\multiput(24.25,11)(1,0){3}{\line(1,0){0.5}}
\multiput(24.25,12)(1,0){3}{\line(1,0){0.5}}
\multiput(24.25,13)(1,0){3}{\line(1,0){0.5}}
\multiput(24.25,14)(1,0){3}{\line(1,0){0.5}}
\multiput(24.25,15)(1,0){3}{\line(1,0){0.5}}
\multiput(24.25,16)(1,0){3}{\line(1,0){0.5}}
\multiput(24.25,17)(1,0){3}{\line(1,0){0.5}}
\multiput(24.25,18)(1,0){3}{\line(1,0){0.5}}
\multiput(6.25,19)(1,0){21}{\line(1,0){0.5}}
\multiput(6.25,20)(1,0){21}{\line(1,0){0.5}}

\put(24,1.5){\vector(1,0){6}}
\put(31,1){$h$}
{\color{red}
\multiput(24,0)(1,0){1}%
{\line(0,1){3}}
\multiput(6,3)(1,0){1}%
{\line(1,0){3}}
\multiput(9,21)(1,0){1}%
{\line(0,-1){3}}
\multiput(27,18)(1,0){1}%
{\line(-1,0){3}}}

\end{picture}
\caption{The middle part and shaded region are $\widetilde{B}_{k,1,6}(h)$ and $\widehat{B}_{k,1,6}(h):=\cup_{i=1}^4 \widehat{U}_{k,1,6,i}(h)$, respectively, the union is $B_{k,1,6}$.}
\label{fig:fig2_of_str_GIFS_on R2}
 \end{figure}
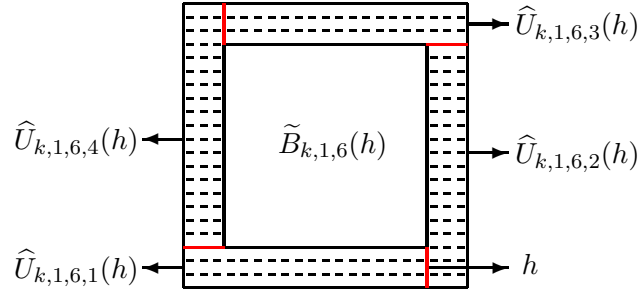
\end{center}
(i) We show \eqref{e:GIFS_str_con_R2_B_k1j_est} for $j=2$, that is,
\begin{eqnarray*}
\int_{\widehat{B}_{k,1,2}(h)}\mu(B_h({\bf x}))^q\,d{\bf x}=o(h^{2+\alpha+\epsilon/2})\qquad \mbox{ for any }2\le k\le N.
\end{eqnarray*}

Since $B_h({\bf x})\subseteq B_{\rho^{2k}}\big(S_{e^{k-1}_1e_2}(0,0)\big)$ for any ${\bf x}\in \widehat{U}_{k,1,2,1}(h)$,
we get
\begin{eqnarray}\label{e:GIFS_str_con_R2_U_k121}
\int_{\widehat{U}_{k,1,2,1}(h)} \mu(B_h({\bf x}))^q\,d{\bf x}
\le \int_{\widehat{U}_{k,1,2,1}(h)} \mu(B_{\rho^{2k}}(S_{e^{k-1}_1e_2}(0,0)))^q\,d{\bf x}
\le p_{e^{k-1}_1 e_2}^q\rho^{2k}h.
\end{eqnarray}
For any ${\bf x}\in \widehat{U}_{k,1,2,2}(h)$, $B_h({\bf x})\subseteq B_{h(1/\sin 22.5^\circ+1)}\big(S_{e_1^{k-1}e_2}(1,0)\big)$ implies that
\begin{eqnarray}\label{e:GIFS_str_con_R2_U_k122}
\int_{\widehat{U}_{k,1,2,2}(h)}\mu(B_h({\bf x}))^q\,d{\bf x}
\le
(\sqrt{2}+1)p_{e_1^{k-1}e_2}^qh^2.
\end{eqnarray}
The fact $B_h({\bf x})\subseteq B_{\sqrt{2}\rho^{2k}}(S_{e_1^{k-1}e_2}(0,1)+h(1,-(\sqrt{2}+1)))$
for any ${\bf x}\in \widehat{U}_{k,1,2,3}(h)$ implies that
\begin{eqnarray}\label{e:GIFS_str_con_R2_U_k123}
\begin{aligned}
\int_{\widehat{U}_{k,1,2,3}(h)}\mu(B_h({\bf x}))^q\,d{\bf x}
\le&p_{e_1^{k-1}e_2}^q \mu\Big(B_{\sqrt{2}}\big((0,1)+\rho^{-2k}h(1,-(\sqrt{2}+1))\big)\Big)^q \int_{\widehat{U}_{k,1,2,3}(h)}1\,d{\bf x}\\
\le&\sqrt{2}p_{e_1^{k-1}e_2}^q\rho^{2k} h.
\end{aligned}
\end{eqnarray}
It follows from $B_h({\bf x})\subseteq B_{h(1/\sin 22.5^\circ+1)}(S_{e_1^{k-1}e_2}(0,1))$
for any $x\in \widehat{U}_{k,1,2,4}(h)$ that
\begin{eqnarray}\label{e:GIFS_str_con_R2_U_k124}
\int_{\widehat{U}_{k,1,2,4}(h)}\mu(B_h({\bf x}))^q\,d{\bf x}
\le (\sqrt{2}+1)p_{e_1^{k-1}e_2}^q h^2.
\end{eqnarray}
Note that $B_h({\bf x})\subseteq B_{\rho^{2k}}(S_{e_1^{k-1}e_2}(0,0)+(0,h))$ for any ${\bf x}\in \widehat{U}_{k,1,2,5}(h)$,
and so
\begin{eqnarray}\label{e:GIFS_str_con_R2_U_k125}
\int_{\widehat{U}_{k,1,2,5}(h)}\mu(B_h({\bf x}))^q\,d{\bf x}\le p_{e_1^{k-1}e_2}^q\big(\rho^{2k}-h-h/\tan 22.5^\circ\big) h
\le p_{e_1^{k-1}e_2}^q \rho^{2k}h.
\end{eqnarray}
By \eqref{e:GIFS_str_con_R2_U_k121}--\eqref{e:GIFS_str_con_R2_U_k125}, we obtain
\begin{eqnarray*}
\begin{aligned}
\int_{\widehat{B}_{k,1,2}(h)}\mu(B_h({\bf x}))^q\,d{\bf x}
\le&p_{e_1^{k-1}e_2}^q \big(2(\sqrt{2}+1)h^2+(2+\sqrt{2})\rho^{2k}h\big)\\
\le&C \big(2(\sqrt{2}+1)+(2+\sqrt{2})\rho^{2(2-N)}\big)p_{e_1}^{-Nq}p^q_{e_2}h^{2+\alpha+\epsilon}.
\end{aligned}
\end{eqnarray*}
This completes the proof of \eqref{e:GIFS_str_con_R2_B_k1j_est} for $j=2$.

(ii) The proofs of \eqref{e:GIFS_str_con_R2_B_k1j_est} for $j=3,4,5$ are similar to that of (i).

(iii) We show \eqref{e:GIFS_str_con_R2_B_k1j_est} holds for $j=6$, that is,
\begin{eqnarray*}
\int_{\widehat{B}_{k,1,6}(h)}\mu(B_h({\bf x}))^q\,d{\bf x}=\sum_{i=1}^4 \int_{\widehat{U}_{k,1,6,i}}\mu(B_h({\bf x}))^q\,d{\bf x}
=o(h^{2+\alpha+\epsilon/2})\qquad\mbox{ for any }2\le k\le N.
\end{eqnarray*}

Since
\begin{eqnarray*}
\int_{\widehat{U}_{k,1,6,i}(h)}\mu(B_h({\bf x}))^q\,dx
\le(p_{e_1}^{k-2}p_{e_4}p_{e_7})^q\rho^{2k} h\qquad\mbox{ for any }{\bf x}\in \widehat{U}_{k,1,6,i}(h)\mbox{ and }i=1,2,3,4,
\end{eqnarray*}
using \eqref{e:GIFS_str_con_R2_fact_est}, we get
\begin{eqnarray*}
\begin{aligned}
\int_{\widehat{B}_{k,1,6}(h)}\mu(B_h({\bf x}))^q\,d{\bf x}
\le &4 p_{e_1}^{q(N+1)}p_{e_1}^{q(k-N-3)} p_{e_4e_7}^q \rho^{2N}\rho^{2(k-N)}h\\
\le &Cp_{e_1}^{-(N+1)q} p_{e_4e_7}^q \rho^{2(2-N)} h^{2+\alpha+\epsilon}.
\end{aligned}
\end{eqnarray*}

By (i-iii), we get \eqref{e:GIFS_str_con_R2_B_k1j_est} holds for $j=2,3,4,5,6$.

\noindent\textit{Part 2. Prove \eqref{e:GIFS_str_con_R2_B_1j_23_est}.}
The proof of \eqref{e:GIFS_str_con_R2_B_1j_23_est} is similar to that of \eqref{e:GIFS_str_con_R2_B_k1j_est} holds for $j=2$.

\noindent\textit{Part 3. The proof of \eqref{e:GIFS_str_con_R2_B_1j_456_est}.}
Let
\begin{eqnarray*}
\int_{\widehat{B}_{1,j}(h/{\rho^{2(k-1)}})}\mu(B_{h/{\rho^{2(k-1)}}}({\bf x}))^q\,d{\bf x}
=\sum_{i=1}^5 \int_{\widehat{U}_{1,j,i}(h/{\rho^{2(k-1)}})}\mu(B_{h/{\rho^{2(k-1)}}}({\bf x}))^q\,d{\bf x}\quad\mbox{ for }j=4,5,
\end{eqnarray*}
and
\begin{eqnarray*}
\int_{\widehat{B}_{1,6}(h/{\rho^{2(k-1)}})}\mu(B_{h/{\rho^{2(k-1)}}}({\bf x}))^q\,d{\bf x}
=\sum_{i=1}^4 \int_{\widehat{U}_{1,6,i}(h/{\rho^{2(k-1)}})}\mu(B_{h/{\rho^{2(k-1)}}}({\bf x}))^q\,d{\bf x}.
\end{eqnarray*}

(i) We show that \eqref{e:GIFS_str_con_R2_B_1j_456_est} holds for $j=4$, that is,
\begin{eqnarray*}
w^q(k-2)\rho^{4(k-1)}\sum_{i=1}^5 \int_{\widehat{U}_{1,4,i}(h/{\rho^{2(k-1)}})}\mu(B_{h/{\rho^{2(k-1)}}}({\bf x}))^q\,d{\bf x}
=o(h^{2+\alpha+\epsilon/2})\quad\mbox{for any }2\le k\le N.
\end{eqnarray*}

It is easy to see that
\begin{eqnarray*}
\begin{aligned}
\int_{\widehat{U}_{1,4,1}(h/{\rho^{2(k-1)}})}\mu(B_{h/{\rho^{2(k-1)}}}({\bf x}))^q\,d{\bf x}
\le&\int_{\widehat{U}_{1,4,1}(h/{\rho^{2(k-1)}})}\mu\big(B_{\rho^2}(S_{e_5}(0,0))\big)^q\,d{\bf x}\\
\le&p^q_{e_5}\rho^{2(2-k)} h.
\end{aligned}
\end{eqnarray*}

The fact
\begin{eqnarray*}
B_{h/{\rho^{2(k-1)}}}({\bf x})\subseteq B_{(1/\sin 22.5^\circ+1)\cdot h/{\rho^{2(k-1)}}}(S_{e_5}(1,0))
\quad\mbox{ for any }{\bf x}\in \widehat{U}_{1,4,2}(h/{\rho^{2(k-1)}})
\end{eqnarray*}
implies that
\begin{eqnarray*}
\int_{\widehat{U}_{1,4,2}(h/{\rho^{2(k-1)}})}\mu(B_{h/{\rho^{2(k-1)}}}({\bf x}))^q\,d{\bf x}
\le p^q_{e_5}\int_{\widehat{U}_{1,4,2}(h/{\rho^{2(k-1)}})}1\,d{\bf x}
\le (\sqrt{2}+1)p^q_{e_5}\rho^{4(1-k)}h^2.
\end{eqnarray*}

Since
\begin{eqnarray*}
B_{h/{\rho^{2(k-1)}}}({\bf x})\subseteq B_{\sqrt{2}\rho^2}\big(S_{e_5}(1,0)+\big(-(\sqrt{2}+1),1\big)\cdot h/{\rho^{2(k-1)}}\big)\quad\mbox{for any }{\bf x}\in \widehat{U}_{1,4,3}(h/{\rho^{2(k-1)}}),
\end{eqnarray*}
we get
\begin{eqnarray*}
\begin{aligned}
\int_{\widehat{U}_{1,4,3}(h/{\rho^{2(k-1)}})}\mu(B_{h/{\rho^{2(k-1)}}}({\bf x}))^q\,d{\bf x}
\le p^q_{e_5}\int_{\widehat{U}_{1,4,3}(h/{\rho^{2(k-1)}})}1\,d{\bf x}
\le\sqrt{2}p^q_{e_5}\rho^{2(2-k)}h.
\end{aligned}
\end{eqnarray*}

The fact
\begin{eqnarray*}
B_{h/{\rho^{2(k-1)}}}({\bf x})\subseteq B_{(1+1/{\sin 22.5^\circ})\cdot h/{\rho^{2(k-1)}}}(S_{e_5}(0,1))\quad\mbox{ for any }{\bf x}\in \widehat{U}_{1,4,4}(h/{\rho^{2(k-1)}}),
\end{eqnarray*}
implies that
\begin{eqnarray*}
\int_{\widehat{U}_{1,4,4}(h/{\rho^{2(k-1)}})}\mu(B_{h/{\rho^{2(k-1)}}}({\bf x}))^q\,d{\bf x}
\le p^q_{e_5}\int_{\widehat{U}_{1,4,4}(h/{\rho^{2(k-1)}})}1 \,d{\bf x}
\le(\sqrt{2}+1)p^q_{e_5}\rho^{4(1-k)}h^2.
\end{eqnarray*}

Since
\begin{eqnarray*}
B_{h/{\rho^{2(k-1)}}}({\bf x})\subseteq B_{\rho^2}(S_{e_5}(0,0)+(0,h/{\rho^{2(k-1)}}))\qquad
\mbox{ for any }{\bf x}\in\widehat{U}_{1,4,5}(h/{\rho^{2(k-1)}}),
\end{eqnarray*}
we have
\begin{eqnarray*}
\int_{\widehat{U}_{1,4,5}(h/{\rho^{2(k-1)}})}\mu(B_{h/{\rho^{2(k-1)}}}({\bf x}))^q\,d{\bf x}
\le p^q_{e_5}\int_{\widehat{U}_{1,4,5}(h/{\rho^{2(k-1)}})}1\,d{\bf x}
\le p^q_{e_5}\rho^{2(2-k)} h.
\end{eqnarray*}
Consequently,
\begin{eqnarray*}
\begin{aligned}
&w^q(k-2)\rho^{4(k-1)}\sum_{i=1}^5 \int_{\widehat{U}_{1,4,i}}\mu(B_{h/{\rho^{2(k-1)}}}({\bf x}))^q\,d{\bf x}\\
\le&w^q(k-2)\rho^{4(k-1)}p^q_{e_5}\big(2(\sqrt{2}+1)\rho^{4(1-k)} h^2
+(2+\sqrt{2})\rho^{2(2-k)}h\big)\\
=& w^q(k-2) p^q_{e_5}\big(2(\sqrt{2}+1)h^2+(2+\sqrt{2})\rho^{2k}h\big)\\
=&w^q(N+1)\cdot \big(2/{p_{e_8}}\big)^{q(N+1)}\cdot p^q_{e_5}\big(2(\sqrt{2}+1)h^2+(2+\sqrt{2})\cdot\rho^{2N}\rho^{2(k-N)}\cdot h\big)\\
\le&C\big(2(\sqrt{2}+1)+(2+\sqrt{2})\rho^{2(2-N)}\big)p^q_{e_5}\big(2/{p_{e_8}}\big)^{q(N+1)}h^{2+\alpha+\epsilon},
\end{aligned}
\end{eqnarray*}
where we use the fact that
\begin{eqnarray*}
\begin{aligned}
w(k)
=&\frac{p_{e_4}\big(p^k_{e_1}+p^{k-1}_{e_1}p_{e_8}+\cdots+p_{e_1}p^{k-1}_{e_8}+p^k_{e_8}\big)
\big(p^{N+1}_{e_1}+p^N_{e_1}p_{e_8}+\cdots+p_{e_1}p^{N}_{e_8}+p^{N+1}_{e_8}\big)}
{\big(p^{N+1}_{e_1}+p^N_{e_1}p_{e_8}+\cdots+p_{e_1}p^{N}_{e_8}+p^{N+1}_{e_8}\big)}\\
\le&\frac{(p_{e_1}+p_{e_8})^k}{p^{N+1}_{e_1}+p^N_{e_1}p_{e_8}+\cdots+p_{e_1}p^{N}_{e_8}+p^{N+1}_{e_8}} w(N+1)\\
\le&2^{N-2}p^{-(N+1)}_{e_8} w(N+1)\\
\le&\big(2/{p_{e_8}}\big)^{N+1} w(N+1)\qquad\mbox{ for any }0\le k\le N-2.
\end{aligned}
\end{eqnarray*}

(ii-iii) The proof of \eqref{e:GIFS_str_con_R2_B_1j_456_est} for $j=5,6$ are similar to that of \eqref{e:GIFS_str_con_R2_B_1j_456_est} for $j=4$ and \eqref{e:GIFS_str_con_R2_B_k1j_est} for $j=6$, respectively.

(c) It is equivalent to show that $\int_{B_{N,1,1}}\mu(B_h(x))^q\,dx=o(h^{2+\alpha+\epsilon/2})$.
From \eqref{e:defi_B_1_ell_str_conn_R2} and \eqref{e:defi_B_k_ell_j_str_conn_R2}, we get
\begin{eqnarray*}
B_{N,1,1}=S_{e_1^{N-1}}(B_{1,1})=S_{e_1^N}(\Omega_1)\cup S_{e_1^{N-1}e_4}(\Omega_2).
\end{eqnarray*}
Hence (see Figure \ref{fig:fig5_of_str_GIFS_on R2})
\begin{eqnarray*}
\int_{B_{N,1,1}}\mu(B_h({\bf x}))^q\,d{\bf x}
=\Big(\int_{\widetilde{B}_{N,1,1}(h)}+\sum_{i=1}^9\int_{\widehat{U}_{N,1,1,i}(h)}\Big)\mu(B_h({\bf x}))^q\,d{\bf x}.
\end{eqnarray*}

From Lemma \ref{L:mu_equi_not_str_conn_R2}(d), we get
\begin{eqnarray*}
\mu_1(B_{N,1,1})=p^{N-1}_{e_1}\mu_1|_{B_{1,1}}\circ S^{-1}_{e^{N-1}_1}+w(N-2)\mu_2|_{B_{1,7}}\circ S^{-1}_{e_1^{N-2}e_4}
\le p^{N-1}_{e_1}+w(N-2).
\end{eqnarray*}
Combining this with $B_h({\bf x})\subseteq B_{N,1,1}$ for any ${\bf x}\in \widetilde{B}_{N,1,1}(h)$,
we get
\begin{eqnarray*}
\begin{aligned}
\int_{\widetilde{B}_{N,1,1,(h)}}\mu(B_h({\bf x}))^q\,d{\bf x}
\le& \frac{3}{2}\big(p^{N-1}_{e_1}+w(N-2)\big)^q(\rho^{2N}-2h)^2\\
\le& \frac{3}{2}\big(p^{-2}_{e_1}p^{N+1}_{e_1}+(2/{p_{e_8}})^{N+1}w(N+1)\big)^q \rho^{4N}\\
\le& C(p^{-2}_{e_1}+(2/{p_{e_8}})^{N+1})^q h^{2+\alpha+\epsilon}.
\end{aligned}
\end{eqnarray*}

For any ${\bf x}\in \widehat{U}_{N,1,1,1}(h)$,
$B_h({\bf x})\subseteq B_{\rho^{2N}+h}(S_{e_1^N}(1,0))$ imply that
\begin{eqnarray*}
\begin{aligned}
\int_{\widehat{U}_{N,1,1,1}(h)}\mu(B_h({\bf x}))^q\,d{\bf x}
\le \int_{\widehat{U}_{N,1,1,1}(h)}\mu(B_{\rho^{2N}+h}(S_{e_1^N}(1,0)))^q\,d{\bf x}
\le p_{e_1}^{Nq}\rho^{2N}h.
\end{aligned}
\end{eqnarray*}
Since $B_h({\bf x})\subseteq B_{\rho^{2N+1}+2h}(S_{e_1^N}(1,0)+(0,h))$ for any ${\bf x}\in \widehat{U}_{N,1,1,2}(h)$,
we get
\begin{eqnarray*}
\begin{aligned}
\int_{\widehat{U}_{N,1,1,2}(h)}\mu(B_h({\bf x}))^q\,d{\bf x}
\le\int_{\widehat{U}_{N,1,1,2}(h)}\mu(B_{\rho^{2N+1}+2h}(S_{e_1^N}(1,0)+(0,h)))^q\,d{\bf x}
\le p_{e_1}^{Nq}\rho^{2N+1} h.
\end{aligned}
\end{eqnarray*}
The fact
$B_h({\bf x})\subseteq B_{\rho^{2N+1}}(S_{e_1^{N-1}e_4}(0,0)+(\rho^{2(N+1)},0))$ for any ${\bf x}\in \widehat{U}_{N,1,1,3}(h)$
implies that
\begin{eqnarray*}
\begin{aligned}
\int_{\widehat{U}_{N,1,1,3}(h)}\mu(B_h({\bf x}))^q\,d{\bf x}
\le &\int_{\widehat{U}_{N,1,1,3}(h)}\mu\big(B_{\rho^{2N+1}}(S_{e_1^{N-1}e_4}(0,0)+(\rho^{2(N+1)},0))\big)^q\,d{\bf x}\\
\le&(p_{e_1}^{N-1}p_{e_4})^q\rho^{2N+1} h.
\end{aligned}
\end{eqnarray*}
Since $B_h({\bf x})\subseteq B_{h(1+1/\sin 22.5^\circ)}(S_{e_1^{N-1}e_4}(1,0))$ for any ${\bf x}\in \widehat{U}_{N,1,1,4}(h)$, we get
\begin{eqnarray*}
\begin{aligned}
\int_{\widehat{U}_{N,1,1,4}(h)}\mu(B_h({\bf x}))^q\,d{\bf x}
\le&\int_{\widehat{U}_{N,1,1,4}(h)}\mu(B_{h(1/\sin 22.5^\circ+1)}(S_{e_1^{N-1}e_4}(1,0)))^q\,d{\bf x}\\
=&(\sqrt{2}+1)(p_{e_1}^{N-1}p_{e_4})^q h^2.
\end{aligned}
\end{eqnarray*}

Note that
\begin{eqnarray*}
\begin{aligned}
h+\sqrt{(\sqrt{2}\rho^{2N}-2h/\tan 22.5^\circ)^2+h^2}
\le \sqrt{2}\rho^{2N}-2(\sqrt{2}+1)h+2h \le \sqrt{2} \rho^{2N},
\end{aligned}
\end{eqnarray*}
we obtain
\begin{eqnarray*}
B_h({\bf x})\subseteq B_{\sqrt{2}\rho^{2N}}\big(S_{e_1^{N-1}e_4}(0,1)+h(1,-(\sqrt{2}+1))\big)\quad\mbox{ for any }{\bf x}\in \widehat{U}_{N,1,1,5}(h),
\end{eqnarray*}
and so
\begin{eqnarray*}
\begin{aligned}
\int_{\widehat{U}_{N,1,1,5}(h)}\mu(B_h({\bf x}))^q\,d{\bf x}
\le&\sqrt{2} p_{e^{N-1}_1 e_4}^q\rho^{2N} h.
\end{aligned}
\end{eqnarray*}
It is easy to see that
\begin{eqnarray*}
\begin{aligned}
\int_{\widehat{U}_{N,1,1,6}(h)}\mu(B_h({\bf x}))^q\,d{\bf x}
\le\int_{\widehat{U}_{N,1,1,6}(h)}\mu(B_{h(1/\sin 22.5^\circ+1)}(S_{e_1^{N-1}e_4}(0,1)))^q\,d{\bf x}
\le(\sqrt{2}+1)p_{e^{N-1}_1 e_4}^q h^2.
\end{aligned}
\end{eqnarray*}
Since $h+\sqrt{(\rho^{2N+1}+h-h/\tan 22.5^\circ)^2+h^2}\le\rho^{2N+1}+h$,
we have
\begin{eqnarray*}
\begin{aligned}
\int_{\widehat{U}_{N,1,1,7}(h)}\mu(B_h({\bf x}))^q\,d{\bf x}
\le&\int_{\widehat{U}_{N,1,1,7}(h)}\mu\big(B_{\rho^{2N+1}+h}(S_{e_1^{N-1}e_4}(0,1))-(0,\sqrt{2}+1)h\big)^q\,d{\bf x}\\
\le&p_{e^{N-1}_1 e_4}^q(\rho^{2N+1}+h-h/\tan22.5^\circ) h\\
\le&p_{e^{N-1}_1 e_4}^q\rho^{2N+1}h.
\end{aligned}
\end{eqnarray*}

The fact $B_h({\bf x})\subseteq B_{\rho^{2N+1}+2h}(S_{e_1^N}(0,1))$ for any ${\bf x}\in \widehat{U}_{N,1,1,8}(h)$
implies that
\begin{eqnarray*}
\begin{aligned}
\int_{\widehat{U}_{N,1,1,8}(h)}\mu(B_h({\bf x}))^q\,d{\bf x}
\le p_{e_1}^{2N}\int_{\widehat{U}_{N,1,1,8}(h)} \mu(B_{\rho+2\rho^{-2N}h}(0,1))^q\,d{\bf x}
\le p_{e_1}^{2Nq}\rho^{2N} h.
\end{aligned}
\end{eqnarray*}
Easily,
\begin{eqnarray*}
\begin{aligned}
\int_{\widehat{U}_{N,1,1,9}(h)}\mu(B_h({\bf x}))^q\,d{\bf x}
\le\int_{\widehat{U}_{N,1,1,9}(h)}\mu(B_{\rho^{2N}+h}(S_{e_1^N}(0,0)))^q\,d{\bf x}
\le p_{e_1}^{2Nq}\rho^{2N}h.
\end{aligned}
\end{eqnarray*}
Therefore
\begin{eqnarray*}
\begin{aligned}
&\int_{B_{N,1,1}}\mu(B_h({\bf x}))^q\,d{\bf x}\\
\le& \Big((1+\rho)p_{e_1}^{Nq}\rho^{2N}+(2\rho+\sqrt{2})p_{e^{N-1}_1 e_4}^q\rho^{2N}+2p_{e_1}^{2Nq}\rho^{2N}\Big)h
+2\big(\sqrt{2}+1\big)p_{e^{N-1}_1 e_4}^q h^2\\
&+C\Big(p^{-2}_{e_1}+(2/{p_{e_8}})^{N+1}\Big)^q h^{2+\alpha+\epsilon}\\
\le&\Big((1+\rho)p^{-q}_{e_1}+(2\rho+\sqrt{2})(p_{e_1}^{-2}p_{e_4})^q+2p_{e_1}^{q(N-1)}\Big)p_{e_1}^{q(N+1)}\rho^{2N}h\\
&+2(\sqrt{2}+1)p_{e_1}^{q(N+1)} \big(p_{e_1}^{-2} p_{e_4}\big)^q h^2
+C\Big(p^{-2}_{e_1}+(2/{p_{e_8}})^{N+1}\Big)^q h^{2+\alpha+\epsilon}\\
\le&C\Big((1+\rho)p^{-q}_{e_1}+(2\rho+\sqrt{2})(p_{e_1}^{-2}p_{e_4})^q+2p_{e_1}^{q(N-1)}
+2(\sqrt{2}+1)(p_{e_1}^{-2} p_{e_4})^q\\
&\quad+(p^{-2}_{e_1}+(2/{p_{e_8}})^{N+1})^q\Big)h^{2+\alpha+\epsilon}.
\end{aligned}
\end{eqnarray*}

\begin{center}
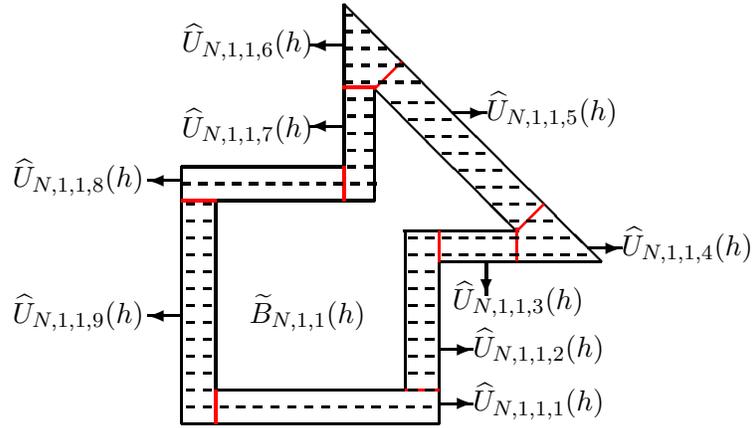
\begin{figure}[h]

\begin{picture}(400,180)
\unitlength=0.9cm

\thicklines

\put(5,1){\line(0,1){3.8}}
\put(5,1){\line(1,0){3.8}}
\put(5,4.8){\line(1,0){2.4}}
\put(8.8,1){\line(0,1){2.4}}

\put(8.8,3.4){\line(1,0){2.4}}
\put(7.4,4.8){\line(0,1){2.4}}
\put(7.4,7.2){\line(1,-1){3.8}}

\put(5.5,1.5){\line(0,1){2.8}}
\put(5.5,1.5){\line(1,0){2.8}}
\put(5.5,4.3){\line(1,0){2.35}}
\put(8.3,1.5){\line(0,1){2.37}}

\put(8.28,3.85){\line(1,0){1.7}}
\put(7.85,4.3){\line(0,1){1.7}}
\put(7.85,5.95){\line(1,-1){2.1}}

{\color{red}
\put(5.5,1){\line(0,1){0.5}}
\put(8.3,1.5){\line(1,0){0.5}}
\put(5,4.3){\line(1,0){0.5}}
\put(7.4,4.3){\line(0,1){0.5}}
\put(8.8,3.4){\line(0,1){0.45}}
\put(7.86,5.975){\line(-1,0){0.485}}
\put(7.85,5.95){\line(1,1){0.4}}
\put(9.95,3.9){\line(0,-1){0.485}}
\put(9.95,3.85){\line(1,1){0.4}}
}

\multiput(5.0625,1.25)(0.25,0){15}{\line(1,0){0.125}}
\multiput(5.0625,1.5)(0.25,0){2}{\line(1,0){0.125}}
\multiput(5.0625,1.75)(0.25,0){2}{\line(1,0){0.125}}
\multiput(5.0625,2)(0.25,0){2}{\line(1,0){0.125}}
\multiput(5.0625,2.25)(0.25,0){2}{\line(1,0){0.125}}
\multiput(5.0625,2.5)(0.25,0){2}{\line(1,0){0.125}}
\multiput(5.0625,2.75)(0.25,0){2}{\line(1,0){0.125}}
\multiput(5.0625,3)(0.25,0){2}{\line(1,0){0.125}}
\multiput(5.0625,3.25)(0.25,0){2}{\line(1,0){0.125}}
\multiput(5.0625,3.5)(0.25,0){2}{\line(1,0){0.125}}
\multiput(5.0625,3.75)(0.25,0){2}{\line(1,0){0.125}}
\multiput(5.0625,4)(0.25,0){2}{\line(1,0){0.125}}
\multiput(5.0625,4.25)(0.25,0){2}{\line(1,0){0.125}}
\multiput(5.0,4.55)(0.25,0){12}{\line(1,0){0.125}}
\multiput(7.45,4.8)(0.25,0){2}{\line(1,0){0.125}}
\multiput(7.45,5.05)(0.25,0){2}{\line(1,0){0.125}}
\multiput(7.45,5.3)(0.25,0){2}{\line(1,0){0.125}}
\multiput(7.45,5.55)(0.25,0){2}{\line(1,0){0.125}}
\multiput(7.45,5.8)(0.25,0){2}{\line(1,0){0.125}}
\multiput(7.4,6.05)(0.25,0){5}{\line(1,0){0.125}}
\multiput(7.4,6.3)(0.25,0){4}{\line(1,0){0.125}}
\multiput(7.4,6.55)(0.25,0){3}{\line(1,0){0.125}}
\multiput(7.4,6.8)(0.25,0){2}{\line(1,0){0.125}}
\multiput(7.4,7.05)(0.25,0){1}{\line(1,0){0.125}}
\multiput(8.35,1.5)(0.25,0){2}{\line(1,0){0.125}}
\multiput(8.35,1.75)(0.25,0){2}{\line(1,0){0.125}}
\multiput(8.35,2)(0.25,0){2}{\line(1,0){0.125}}
\multiput(8.35,2.25)(0.25,0){2}{\line(1,0){0.125}}
\multiput(8.35,2.5)(0.25,0){2}{\line(1,0){0.125}}
\multiput(8.35,2.75)(0.25,0){2}{\line(1,0){0.125}}
\multiput(8.35,3)(0.25,0){2}{\line(1,0){0.125}}
\multiput(8.35,3.25)(0.25,0){2}{\line(1,0){0.125}}
\multiput(8.35,3.5)(0.25,0){11}{\line(1,0){0.125}}
\multiput(8.35,3.75)(0.25,0){10}{\line(1,0){0.125}}
\multiput(9.85,4)(0.25,0){3}{\line(1,0){0.125}}
\multiput(9.60,4.25)(0.25,0){3}{\line(1,0){0.125}}
\multiput(9.35,4.5)(0.25,0){3}{\line(1,0){0.125}}
\multiput(9.10,4.75)(0.25,0){3}{\line(1,0){0.125}}
\multiput(8.85,5)(0.25,0){3}{\line(1,0){0.125}}
\multiput(8.60,5.25)(0.25,0){3}{\line(1,0){0.125}}
\multiput(8.35,5.5)(0.25,0){3}{\line(1,0){0.125}}
\multiput(8.1,5.75)(0.25,0){3}{\line(1,0){0.125}}

\put(8.8,1.3){\vector(1,0){0.5}}
\put(8.8,2.1){\vector(1,0){0.5}}
\put(9.5,3.4){\vector(0,-1){0.5}}
\put(11,3.6){\vector(1,0){0.5}}
\put(9,5.6){\vector(1,0){0.5}}
\put(7.4,6.6){\vector(-1,0){0.5}}
\put(7.4,5.4){\vector(-1,0){0.5}}
\put(5,4.6){\vector(-1,0){0.5}}
\put(5,2.6){\vector(-1,0){0.5}}

\put(9.3,1.2){$\widehat{U}_{N,1,1,1}(h)$}
\put(9.3,2){$\widehat{U}_{N,1,1,2}(h)$}
\put(9,2.7){$\widehat{U}_{N,1,1,3}(h)$}
\put(11.5,3.5){$\widehat{U}_{N,1,1,4}(h)$}
\put(9.5,5.5){$\widehat{U}_{N,1,1,5}(h)$}
\put(5,6.5){$\widehat{U}_{N,1,1,6}(h)$}
\put(5,5.25){$\widehat{U}_{N,1,1,7}(h)$}
\put(2.5,4.5){$\widehat{U}_{N,1,1,8}(h)$}
\put(2.5,2.5){$\widehat{U}_{N,1,1,9}(h)$}
\put(6,2.5){$\widetilde{B}_{N,1,1}(h)$}

\end{picture}
\caption{The middle part and shaded region are $\widetilde{B}_{N,1,1}(h)$ and $\widehat{B}_{N,1,1}(h)=\cup_{i=1}^9 \widehat{U}_{N,1,1,i}(h)$, respectively. The union is $B_{N,1,1}$.}\label{fig:fig5_of_str_GIFS_on R2}
\end{figure}
\end{center}

(d) 
The proofs are similar to that of (b, c).
\end{proof}

\begin{proof}[Proof of Corollary \ref{C:cor_GIFS_str_con_R2}]
It follows from Theorem \ref{T:main_thm_str_con} and Proposition \ref{P:str_con_R2_error_esti} that $\tau(q)=\alpha$.

In the following, we study the differentially of $\tau$. The proof is similar to that of \cite[Theorem 1.2]{Ngai-Xie_2019}.
Now, we firstly prove that $H(q,\alpha)$ is $C^1$. By Proposition \ref{P:str_con_R_D1_open}, we have
\begin{eqnarray*}
\sum_{k=0}^\infty w^q(k)\rho^{-2\alpha(k+1)}<\infty\qquad\mbox{ for any }(q,\alpha)\in (0,\infty)\times D_1.
\end{eqnarray*}
Since $w(k+1)\le w(k)$, we get $w(k+1)\le w(0)=p_{e_4}$, and so $\sum_{k=0}^\infty w^q(k)\rho^{-2\alpha(k+1)}$ is strictly decreasing in $q$ and strictly increasing in $\alpha$. Hence, for any $(\widetilde{q}, \widetilde{\alpha})\in(0,\infty)\times D_1$, the series converges uniformly on $\{(q,\alpha): q\ge \widetilde{q}, \alpha\le \widetilde{\alpha}\}$. Moreover, $\lim_{k\to\infty} w(k)=0$.
Hence, for any $(q,\alpha)\in (0,\infty)\times D_1$,
\begin{eqnarray*}
\begin{aligned}
&H_q(q,\alpha)\\
=&-\Big(\sum_{i=7}^8 p^q_{e_i}\rho^{-2\alpha}\ln(p_{e_i})\Big)\Big[1-\Big(\sum_{i=1}^3 p^q_{e_i}\rho^{-2\alpha}\Big)\Big]
-\Big[1-\Big(\sum_{i=7}^8 p^q_{e_i}\rho^{-2\alpha}\Big)\Big]\Big(\sum_{i=1}^3 p^q_{e_i}\rho^{-2\alpha}\ln(p_{e_i})\Big)\\
&+\Big(p^q_{e_1}\rho^{-2\alpha}\ln(p_{e_1})\big(1-p^q_{e_8}\rho^{-2\alpha}\big)
+p^q_{e_8}\rho^{-2\alpha}\ln(p_{e_8})\big(1-p^q_{e_1}\rho^{-2\alpha}\big)\Big)\\
&\quad\cdot\Big(\sum_{i=5}^6p^q_{e_i}\rho^{-2\alpha}\Big)\Big(\sum_{k=0}^\infty w^q(k)\rho^{-2\alpha(k+1)}\Big)\\
&-\Big[\Big(\sum_{i=5}^6 p^q_{e_i}\rho^{-2\alpha}\ln(p_{e_i})\Big)
\Big(\sum_{k=0}^\infty w^q(k)\rho^{-2\alpha(k+1)}\Big)
+\Big(\sum_{i=5}^6 p^q_{e_i}\rho^{-2\alpha}\Big)\Big(\sum_{k=0}^\infty w^q(k)\rho^{-2\alpha(k+1)}\ln(w(k))\Big)\Big]\\
&\quad\cdot\prod_{i=1,8}\Big(1-p^q_{e_i}\rho^{-2\alpha}\Big)\\
\end{aligned}
\end{eqnarray*}
and
\begin{eqnarray*}
\begin{aligned}
&H_\alpha(q,\alpha)\\
=&\Big\{\Big(\sum_{i=7}^8 p^q_{e_i}\rho^{-2\alpha}\Big)\Big[1-\Big(\sum_{i=1}^3 p^q_{e_i}\rho^{-2\alpha}\Big)\Big]
+\Big[1-\Big(\sum_{i=7}^8 p^q_{e_i}\rho^{-2\alpha}\Big)\Big]\Big(\sum_{i=1}^3 p^q_{e_i}\rho^{-2\alpha}\Big)\\
&-\Big[p^q_{e_1}\rho^{-2\alpha}\Big(1-p^q_{e_8}\rho^{-2\alpha}\Big)+p^q_{e_8}\rho^{-2\alpha}\Big(1-p^q_{e_1}\rho^{-2\alpha}\Big)\Big]
\Big(\sum_{i=5}^6 p^q_{e_i}\rho^{-2\alpha}\Big)\Big(\sum_{k=0}^\infty w^q(k)\rho^{-2\alpha(k+1)}\Big)\\
&+(k+2)\Big(\prod_{i=1,8}\Big(1-p^q_{e_i}\rho^{-2\alpha}\Big)\Big)
\Big(\sum_{i=5}^6 p^q_{e_i}\rho^{-2\alpha}\Big)
\Big(\sum_{k=0}^\infty w^q(k) \rho^{-2\alpha(k+1)}\Big)\Big\}\cdot 2\ln(\rho).
\end{aligned}
\end{eqnarray*}
This proves that $H(q,\alpha)$ is $C^1$.

For any $(q,\alpha)\in (0,\infty)\times D_1$ satisfying $H(q,\alpha)=0$ and $H_\alpha(q,\alpha)\neq 0$, let $\{q_n\}_{n=1}^\infty$ be an increasing sequence of positive numbers satisfying
$\lim_{n\to\infty} q_n=\widetilde{q}$ and $\tau$ is differentiable at each $q_n$. Then
\begin{eqnarray}\label{eq:relation_about_H_q_and_H_alpha_not_str_con_R}
H_q(q_n,\alpha_n)+\alpha'(q_n)\cdot H_\alpha(q_n,\alpha_n)=0\quad\mbox{ for any }n=1,2,\ldots,
\end{eqnarray}
which, yields
\begin{eqnarray*}
\begin{aligned}
H_q(\widetilde{q},\widetilde{\alpha})+\alpha'(\widetilde{q})\cdot H_\alpha(\widetilde{q},\widetilde{\alpha})=0,
\end{aligned}
\end{eqnarray*}
with $\alpha(q)=\tau(q)$ at $q$. The implicit function theorem imply that $\tau$ is differentiable at $q$. Moreover,
\eqref{e:result_tau_q_2} holds.
\end{proof}

\section{GIFSs that are not strongly connected on $\R$ and $\R^2$}\label{S:not_str_conn_R}
\setcounter{equation}{0}
The global of this section is to compute the $L^q$-spectrum of some graph-directed self-similar measures defined by the GIFSs $G=(V,E)$ which have overlaps and are not strongly connected on $\R$ and $\R^2$.

\subsection{A GIFS that is not strongly connected on $\R$ and has a unique basic class}

In this subsection, we study the graph-directed self-similar measure $\mu$ defined by the GIFS in Example \ref{E:exam_not_str_con_R_1}, and compute the $L^q$-spectrum of $\mu$. Ngai and Xie \cite[Section 6.1]{Ngai-Xie_2020} have gave the proof of that $\mu$ satisfies (EFT).
In the following, we state the result and omit the details of proof.

\begin{prop}\label{P:not_str_con_R}
Let $\mu=\sum_{i=1}^2\mu_i$ be a graph-directed self-similar measure defined by $G=(V,E)$ in Example \ref{E:exam_not_str_con_R_1} together with a probability matrix $(p_e)_{e\in E}$. Then $\mu$ satisfies (EFT) with $\Omega=\{\Omega_i\}_{i=1}^2=\{(0,1),(0,1)\}$ being an EFT-family; moreover, there exists a weakly regular basic pair. 	
\end{prop}

Let $\{S_{e_i}\}_{i=1}^5$ be defined as in \eqref{E:exam_str_con_R_1_similitudes},
$(p_e)_{e\in E}$ be a probability matrix, and $\mu$ be a graph-directed self-similar measure defined as in Proposition \ref{P:not_str_con_R}.
It is easy to see that $\eta=2$, $p_{e_1}+p_{e_2}+p_{e_3}=1$, $p_{e_4}+p_{e_5}=1$, $\mu_1=\sum_{j=1}^3p_{e_j}\mu_1\circ S^{-1}_{e_j}$
and $\mu_2=p_{e_4}\mu_2\circ S^{-1}_{e_4}+p_{e_5}\mu_1\circ S^{-1}_{e_5}$.

Denote (see Figure \ref{fig:fig_not_str_conn_R_1})
\begin{eqnarray}\label{eq:not_str_con_R_B_ell}
B_{1,1}:=\bigcup_{i=1}^2 S_{e_i}(\Omega_1),\quad
B_{1,2}:=S_{e_3}(\Omega_1),\quad
B_{1,3}:=S_{e_5}(\Omega_1),\quad
B_{1,4}:=S_{e_4}(\Omega_2).
\end{eqnarray}
\begin{center}
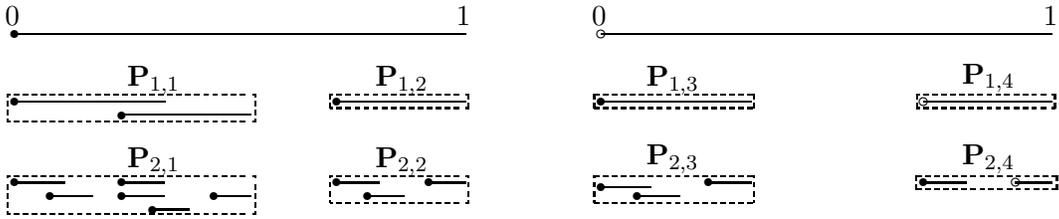
\begin{figure}[h]

\begin{picture}(400,90)
\unitlength=0.06cm

\put(0,47){\line(1,0){100}}
\put(-2,49){$0$}
\put(98,49){$1$}
\put(0,47){\circle*{1.5}}

\put(0,32){\line(1,0){33.33}}
\put(0,32){\circle*{1.5}}
\put(23.81,29){\line(1,0){28.57}}
\put(23.81,29){\circle*{1.5}}
\put(25,35.5){$\mathbf{P}_{1,1}$}
\put(-1.5,27.5){\dashbox(54.88,6){}}

\put(71.43,32){\line(1,0){28.57}}
\put(71.43,32){\circle*{1.5}}
\put(80,35.5){$\mathbf{P}_{1,2}$}
\put(70,30.5){\dashbox(30.57,3){}}

\put(0,14){\line(1,0){11.11}}
\put(0,14){\circle*{1.5}}
\put(7.94,11){\line(1,0){9.52}}
\put(7.94,11){\circle*{1.5}}

\put(23.81,14){\line(1,0){9.52}}
\put(23.81,14){\circle*{1.5}}
\put(23.81,11){\line(1,0){9.52}}
\put(23.81,11){\circle*{1.5}}
\put(30.61,8){\line(1,0){8.163}}
\put(30.61,8){\circle*{1.5}}

\put(44.21,11){\line(1,0){8.163}}
\put(44.21,11){\circle*{1.5}}
\put(-1.5,7){\dashbox(54.873,8.5){}}
\put(25,17.5){$\mathbf{P}_{2,1}$}

\put(71.43,14){\line(1,0){9.52}}
\put(71.43,14){\circle*{1.5}}
\put(78.23,11){\line(1,0){8.163}}
\put(78.23,11){\circle*{1.5}}

\put(91.84,14){\line(1,0){8.163}}
\put(91.84,14){\circle*{1.5}}
\put(70,9.5){\dashbox(30.57,6){}}
\put(80,17.5){$\mathbf{P}_{2,2}$}

\put(130,47){\line(1,0){100}}
\put(128,49){$0$}
\put(228,49){$1$}
\put(130,47){\circle{1.5}}

\put(130,32){\line(1,0){33.33}}
\put(130,32){\circle*{1.5}}
\put(140,35.5){$\mathbf{P}_{1,3}$}
\put(128.5,30.5){\dashbox(35.33,3){}}

\put(201.43,32){\line(1,0){28.57}}
\put(201.43,32){\circle{1.5}}
\put(210,36){$\mathbf{P}_{1,4}$}
\put(200,30.5){\dashbox(30.57,3){}}

\put(130,13){\line(1,0){11.11}}
\put(130,13){\circle*{1.5}}
\put(137.94,11){\line(1,0){9.52}}
\put(137.94,11){\circle*{1.5}}

\put(153.81,14){\line(1,0){9.52}}
\put(153.81,14){\circle*{1.5}}
\put(128.5,9.5){\dashbox(35.33,6){}}
\put(140,18){$\mathbf{P}_{2,3}$}
%
%

\put(201.43,14){\line(1,0){9.52}}
\put(201.43,14){\circle*{1.5}}

\put(221.84,14){\line(1,0){8.163}}
\put(221.84,14){\circle{1.5}}
\put(199.93,12.5){\dashbox(31.07,3){}}
\put(210,18){$\mathbf{P}_{2,4}$}
\end{picture}

\caption{$\mu$-partitions $\mathbf{P}_{k,\ell}$, $k=1,2$, of $B_{1,\ell}$ for the GIFS in \eqref{E:exam_str_con_R_1_similitudes}. The figure is drawn with $\rho=1/3$ and $r=2/7$.}
\label{fig:fig_not_str_conn_R_1}
\end{figure}
\end{center}
Let $\Gamma=\{1,2,3,4\}$ and ${\bf B}:=\{B_{1,\ell}: \ell\in\Gamma\}$. Then $1\in\mathcal{SC}_1$, $2\in\mathcal{SC}_2$,
$\Gamma_1=\{1,2\}$ and $\Gamma_2=\{3,4\}$. Also, let
\begin{eqnarray*}
\begin{aligned}
&B_{k,1,1}=S_{e_2^{k-2}e_1}(B_{1,1}),\quad&&B_{k,1,2}=S_{e_2^{k-1}}(B_{1,1}),\quad&&B_{k,1,3}=S_{e_2^{k-1}}(B_{1,2})\qquad\mbox{ for }k\ge2,\\
&B_{2,2,j}=S_{e_3}(B_{1,j}),\quad&&B_{2,3,j}=S_{e_5}(B_{1,j}),\quad&&B_{2,4,j}=S_{e_4}(B_{1,j+2})\qquad\mbox{ for }j=1,2.
\end{aligned}
\end{eqnarray*}
Define ${\bf P}_{1,\ell}:=\{B_{1,\ell}\}$ for $\ell\in\Gamma$ and
\begin{eqnarray*}
\begin{aligned}
{\bf P}_{2,1}=\{B_{2,1,j}: j=1,2,3\},\qquad{\bf P}_{2,\ell}=\{B_{2,\ell,j}: j=1,2\}\quad\mbox{ for }\ell=2,3,4.
\end{aligned}
\end{eqnarray*}
Denote
\begin{eqnarray}\label{eq:defi_w_k_GIFS_not_str_con_R_1}
w_k:=p_{e_1}\sum_{j=0}^k p_{e_2}^j p_{e_3}^{k-j}\quad\mbox{ for }k\ge0.
\end{eqnarray}
\begin{lem}
Let $\{S_{e_i}\}_{i=1}^5$ defined as in Example \ref{E:exam_not_str_con_R_1}. Then $S_{e_1e_3}=S_{e_2e_1}$ and
$S_{\bf e}= S_{\bf e'}$ for any ${\bf e}, {\bf e'}\in W_k:=\{e_2^j e_1 e_3^{k-j}:j=0,1,\ldots,k\}~(k\ge0)$.
\end{lem}

\begin{lem}\label{L:mu_equi_not_str_conn_R_1}
Let $\{B_{1,\ell}\}_{\ell=1}^4$ be defined as in \eqref{eq:not_str_con_R_B_ell}. Then
\begin{enumerate}
\item[(a)] For $j=1,2$,
\begin{eqnarray*}
\begin{aligned}
\mu_1|_{B_{2,2,j}}=&p_{e_3}\mu_1|_{B_{1,j}}\circ S^{-1}_{e_3},\qquad \mu_2|_{B_{2,3,j}}=p_{e_5}\mu_1|_{B_{1,j}}\circ S^{-1}_{e_5},\\
\mu_2|_{B_{2,4,j}}=&p_{e_4}\mu_2|_{B_{1,j+2}}\circ S^{-1}_{e_4};
\end{aligned}
\end{eqnarray*}
\item[(b)]
For $k\ge2$,
\begin{eqnarray*}
\begin{aligned}
\mu_1|_{B_{k,1,1}}=&w_{k-2}\mu_1|_{B_{1,1}}\circ S^{-1}_{e_2^{k-2}e_1},\qquad\mu_1|_{B_{k,1,3}}=p_{e_2^{k-1}}\mu_1|_{B_{1,2}}\circ S^{-1}_{e_2^{k-1}},\\
\mu_1|_{B_{k,1,2}}=&w_{k-2}\mu_1|_{B_{1,2}}\circ S^{-1}_{e_2^{k-2}e_1}+p_{e_2^{k-1}}\mu_1|_{B_{1,1}}\circ S^{-1}_{e_2^{k-1}}.
\end{aligned}
\end{eqnarray*}
\end{enumerate}
\end{lem}
Using Lemma \ref{L:mu_equi_not_str_conn_R_1}, we get ${\bf P}'_{2,1}=\{B_{2,1,j}: j=1,3\}$, ${\bf P}^\ast_{2,1}=\{B_{2,1,2}\}$ and
${\bf P}'_{2,\ell}={\bf P}_{2,\ell}$, ${\bf P}^\ast_{2,\ell}=\emptyset$ for $\ell=2,3,4$. For $k\ge3$, define
\begin{eqnarray*}
{\bf P}_{k,1}:={\bf P}'_{k-1,1}\cup \{B_{k,1,j}: j=1,2,3\}.
\end{eqnarray*}
Lemma \ref{L:mu_equi_not_str_conn_R_1}(b) implies that ${\bf P}'_{k,1}={\bf P}'_{k-1,1}\cup\{B_{k,1,j}: j=1,3\}$ and
${\bf P}^\ast_{k,1}=\{B_{k,1,2}\}$. Hence $\Gamma'_1=\{2\}$, $\Gamma^\ast_1=\{1\}$, $\Gamma_2'=\Gamma_2$, $\Gamma^\ast_2=\emptyset$, $\kappa_1=\infty$ and $\kappa_\ell=2$ for $\ell=2,3,4$. Also, $G'_{k,1}=\{1,3\}$, $G^\ast_{k,1}=\emptyset$ for $k\ge2$, and $G'_{2,\ell}=\{1,2\}$, $G^\ast_{2,\ell}=\emptyset$ for $\ell=2,3,4$.

\begin{lem}
Let $w(k,\ell,j), c(k,\ell,j), e(k,\ell,j)$ and $\rho_{e(k,\ell,j)}$ be defined as in \eqref{eq:relationship_between_mui_and_B1c}. Then
\begin{enumerate}
\item[(a)] for $k\ge2$,
\begin{eqnarray*}
\begin{aligned}
&w(k,1,1)=w_{k-2},\ c(k,1,1)=1,\ \rho_{e(k,1,1)}=\rho r^{k-2};\\
&w(k,1,3)=p_{e_2^{k-1}},\ c(k,1,3)=2,\ \rho_{e(k,1,3)}=r^{k-1};
\end{aligned}
\end{eqnarray*}
\item[(b)] for $j=1,2$,
\begin{eqnarray*}
\begin{aligned}
&w(2,2,j)=p_{e_3},\ c(2,2,j)=j,\ \rho_{e(2,2,j)}=r;\\
&w(2,3,j)=p_{e_5},\ c(2,3,j)=j,\ \rho_{e(2,3,j)}=r;\\
&w(2,4,j)=p_{e_4},\ c(2,1,j)=j+2,\ \rho_{e(2,4,j)}=r.
\end{aligned}
\end{eqnarray*}
\end{enumerate}
\end{lem}

Bases on the above results, we can express the vector-valued equations \eqref{eq:f_ell_'_1} and \eqref{eq:f_ell_ast_1} as follows
\begin{eqnarray*}
\begin{aligned}
f_1(x)
=&\sum_{k=2}^\infty\Big(w_{k-2}^q (\rho r^{k-2})^{-\alpha_1}\cdot f_1(x+\ln(\rho r^{k-2}))+(p^q_{e_2} r^{-\alpha_1})^{k-1}\cdot f_2(x+\ln(r^{k-1}))\Big)\\
&+z_1^{(\alpha_1)}(x)-z_{1,\infty}^{(\alpha_1)}(x),\\
f_2(x)=&p^q_{e_3} r^{-\alpha_1}\sum_{j=1}^2 f_j(x+\ln r) +z_2^{(\alpha_1)}(x),\\
f_3(x)=&p_{e_5}^q\rho^{-\alpha_2}\sum_{j=1}^2 f_j(x+\ln\rho)+z_3^{(\alpha_2)}(x),\\
f_4(x)=&p_{e_4}^q r^{-\alpha_2} \sum_{j=3}^4 f_j(x+\ln r)+z_4^{(\alpha_2)}(x),
\end{aligned}
\end{eqnarray*}
with $z_{1,\infty}^{(\alpha_1)}(x)=E_{1,\infty}^{(\alpha_1)}(e^{-x})$, and $z_\ell^{(\alpha_i)}(x)=E_\ell^{(\alpha_i)}(e^{-x})$ for $\ell\in\Gamma_i$ and $i=1,2$.

For $m=1,2$, $i\in\mathcal{SC}_m$ and $\ell, \ell'\in\Gamma_i$, let $\mu^{(\alpha_m)}_{\ell\ell'}$ be the discrete measures defined as in \eqref{defi:muij_on_kuangjia}. Then
\begin{eqnarray*}
\begin{aligned}
&\mu_{11}^{(\alpha_1)}(-\ln(\rho r^k))=w_k^q(\rho r^k)^{-\alpha_1}\quad\mbox{ for }k\ge 0,\\
&\mu_{12}^{(\alpha_1)}(-\ln(r^k))=(p_{e_2}^q r^{-\alpha_1})^k\quad\mbox{ for }k\ge1,\\
&\mu_{2\ell'}^{(\alpha_1)}(-\ln r)=p_{e_3}^q r^{-\alpha_1}\quad\mbox{ for }\ell'=1,2,\\
&\mu_{3\ell'}^{(\alpha_2)}(-\ln\rho)=p_{e_5}^q\rho^{-\alpha_2}\quad\mbox{ for }\ell'=1,2,\\
&\mu_{4\ell'}^{(\alpha_2)}(-\ln r)=p_{e_4}^q r^{-\alpha_2}\quad\mbox{ for }\ell'=3,4.
\end{aligned}
\end{eqnarray*}
Moreover,
\begin{eqnarray*}
{\bf M}_{\boldsymbol\alpha}(\infty)=
\left(
  \begin{array}{cccc}
    \sum_{k=0}^\infty w_k^q(\rho r^k)^{-\alpha_1} & p^q_{e_2} r^{-\alpha_1}/(1-p^q_{e_2}r^{-\alpha_1}) & 0 & 0 \\
     p_{e_3}^q r^{-\alpha_1} & p_{e_3}^q r^{-\alpha_1} & 0 & 0 \\
     p_{e_5}^q\rho^{-\alpha_2} & p_{e_5}^q\rho^{-\alpha_2} & 0 & 0 \\
     0 & 0 & p_{e_4}^q r^{-\alpha_2} & p_{e_4}^q r^{-\alpha_2} \\
  \end{array}
\right),
\end{eqnarray*}

In the following, we firstly prove that $D_1$ is open, and then show that the error terms $z^{(\alpha_1)}_{1,\infty}(x)=o(e^{-\epsilon x})$ and $z_\ell^{(\alpha_i)}(x)=o(e^{-\epsilon x})$ as $x\to\infty$
for $\ell\in\Gamma_i$ and $i=1,2$. The proofs are the same as that of Propositions \ref{P:str_con_R_D1_open}-\ref{P:str_con_R2_error_esti}. The details are omitted.

\begin{prop}\label{P:not_str_con_R_1_D1_open}
Let $D_1$ be defined as in \eqref{e:defi_F_ell_and_D_ell}. Then it is open.
\end{prop}
\begin{prop}
For any $k\ge N+1$, $\Phi_1^{(\alpha_1)}(h/{\rho r^{k-2}})\le 1$ and $\Phi_2^{(\alpha_1)}(h/{r^{k-1}})\le 1$.
\end{prop}

\begin{prop}\label{P:not_str_con_R_1_error}
For $q\ge 0$, assume that $\alpha_m\in D_\ell$ for $\ell\in\Gamma_i$, $i\in\mathcal{SC}_m$ and $m=1,2$.
Then there exists some $\epsilon>0$ such that
\begin{enumerate}
\item[(a)] $h^{-(1+\alpha_1)}\Big(\sum\limits_{k=2}^N e_{k,1}+\int_{B_{N,1,2}}\mu(B_h(x))^q\,dx\Big)=o(h^{\epsilon/2})$;
\item[(b)] $\sum\limits_{k=N+1}^\infty w^q_{k-2}(\rho r^{k-2})^{-\alpha_1}\Phi_1^{(\alpha_1)}(h/{\rho r^{k-2}})
+\sum\limits_{k=N+1}^\infty \big(p_{e_2}^{q} r^{-\alpha_1}\big)^{k-1}\Phi_2^{(\alpha_1)}(h/{r^{k-1}})=o(h^{\epsilon/2})$;
\item[(c)] $h^{-(1+\alpha_1)} e_{2,2}=o(h^{\epsilon/2})$ and $h^{-(1+\alpha_2)} e_{2,\ell}=o(h^{\epsilon/2})$ for $\ell=3,4$.
\end{enumerate}
\end{prop}

\begin{proof}[Proof of Corollary~\ref{C:cor_GIFS_not_str_con_R_1}]
By Theorem \ref{T:main_thm_not_str_con} and Proposition \ref{P:not_str_con_R_1_error}, we have $\tau(q)=\alpha$.

The proof of the differentially of $\tau_q$ is similar to that of Corollary~\ref{C:cor_GIFS_str_con_R2}. We omit the details.
\end{proof}

\subsection{A GIFS that is not strongly connected on $\R$ and has basic classes of height greater than 1}

In this subsection, we study the $L^q$-spectrum of the graph-directed self-similar measure $\mu$ defined by the GIFS in Example \ref{E:exam_GIFS_not_str_con_R_2}. For this measure $\mu$, Ngai and Xie \cite[Section 6.3]{Ngai-Xie_2020} have shown that $\mu$ satisfies (EFT). In the following, we give the result and omit the details.

\begin{prop}\label{P:not_str_con_R_2}
Let $\mu=\sum_{i=1}^6 \mu_i$ be a graph-directed self-similar measure defined by a GIFS $G=(V, E)$ in Example \ref{E:exam_GIFS_not_str_con_R_2} together with a probability matric $(p_e)_{e\in E}$. Then $\mu$ satisfies (EFT) with
$\Omega=\{\Omega_i\}_{i=1}^6$
being an EFT-family and there exists a regular basic pair, where $\Omega_i=(0,1)$ for $i=1,\ldots,6$.
Moreover, the GIFS has one basic class of height $2$, and one of height $3$.
\end{prop}

Let $\{S_{e_i}\}_{i=1}^{17}$ be defined as in \eqref{E:exam_str_con_R_2_similitudes}, $(p_e)_{e\in E}$ be a probability matrix, and $\mu$ be a graph-directed self-similar measure defined as in Proposition \ref{P:not_str_con_R_2}. Easily, $\eta=6$,
\begin{eqnarray*}
\begin{aligned}
\sum_{j=k}^{k+2} p_{e_j}=1,\qquad\sum_{j=16}^{17} p_{e_j}=1\quad\mbox{ for }k=1,4,7,10,13,\\
\end{aligned}
\end{eqnarray*}
and
\begin{eqnarray*}
\begin{aligned}
&\mu_1=\sum_{j=1}^3 p_{e_j}\mu_1\circ S^{-1}_{e_j},\quad
&&\mu_2=p_{e_4}\mu_1\circ S^{-1}_{e_4}+\sum_{j=5}^6 p_{e_j}\mu_2\circ S^{-1}_{e_j},\\
&\mu_3=\sum_{j=7}^9 p_{e_i}\mu_3\circ S^{-1}_{e_j},\quad
&&\mu_4=p_{e_{10}}\mu_3\circ S^{-1}_{e_{10}}+\sum_{j=11}^{12} p_{e_j}\mu_4\circ S^{-1}_{e_j},\\
&\mu_5=p_{e_{13}}\mu_3\circ S^{-1}_{e_{13}}+\sum_{j=14}^{15} p_{e_i}\mu_5\circ S^{-1}_{e_j},\quad
&&\mu_6=p_{e_{16}}\mu_1\circ S^{-1}_{e_{16}}+p_{e_{17}}\mu_6\circ S^{-1}_{e_{17}}.
\end{aligned}
\end{eqnarray*}

Define (see Figure \ref{fig:fig2_not_str_conn_R_2})
\begin{eqnarray*}
\begin{aligned}
&B_{1,1}:=\bigcup_{i=1}^2 S_{e_i}(\Omega_1),\quad B_{1,3}:=\bigcup_{i=1}^2S_{e_{i+3}}(\Omega_i),
\quad B_{1,5}:=\bigcup_{i=7}^8 S_{e_i}(\Omega_3),\quad B_{1,7}:=\bigcup_{i=3}^4S_{e_{i+7}}(\Omega_i),\\
&B_{1,9}:=S_{e_{13}}(\Omega_3)\cup S_{e_{14}}(\Omega_5),\qquad B_{1,11}:=S_{e_{16}}(\Omega_1),\qquad B_{1,12}:=S_{e_{17}}(\Omega_6),
\end{aligned}
\end{eqnarray*}
and
\begin{eqnarray*}
B_{1,\ell}:=S_{e_{3\ell/2}}(\Omega_{\ell/2})\quad\mbox{ for }\ell=2,4,6,8,10.
\end{eqnarray*}

Let $\Gamma=\{1,2,\ldots, 12\}$ and ${\bf B}:=\{B_{1,\ell}:\ell\in\Gamma\}$. Then $i\in\mathcal{SC}_i$ and $\Gamma_i=\{2i-1,2i\}$ for $i=1,2,\ldots,6$.
For $k\ge2$, denote
\begin{eqnarray*}
\begin{aligned}
&B_{k,1,1}=S_{e_2^{k-2}e_1}(B_{1,1}),\quad && B_{k,3,1}=S_{e_5^{k-2}e_4}(B_{1,1}),\quad B_{k,5,1}=S_{e_8^{k-2}e_7}(B_{1,5}),\\
& B_{k,7,1}=S_{e_{11}^{k-2}e_{10}}(B_{1,5}),\quad&& B_{k,9,1}=S_{e_{14}^{k-2}e_{13}}(B_{1,5}),\\
&B_{k,\ell,2}=S_{e^{k-1}_{(3\ell+1)/2}}(B_{1,\ell}),\quad &&B_{k,\ell,3}=S_{e^{k-1}_{(3\ell+1)/2}}(B_{1,\ell+1})\quad\mbox{ for }\ell=1,3,5,7,9,\\
\end{aligned}
\end{eqnarray*}
Also, let
\begin{eqnarray*}
\begin{aligned}
&B_{2,\ell,1}=S_{e_{3\ell/2}}(B_{1,\ell-1}),\quad&& B_{2,\ell,2}=S_{e_{3\ell/2}}(B_{1,\ell})\quad\mbox{ for }\ell=2,4,6,8,10\\
&B_{2,\ell,1}=S_{e_{\ell+5}}(B_{1,1}),\quad&& B_{2,\ell,2}=S_{e_{\ell+5}}(B_{1,2})\quad\mbox{ for }\ell=11,12.
\end{aligned}
\end{eqnarray*}
For $\ell\in\Gamma$ define ${\bf P}_{1,\ell}:=\{B_{1,\ell}\}$, and for $\ell=1,3,5,7,9$ and $\ell'=2,4,6,8,10,11,12$, let (see Figure\ref{fig:fig2_not_str_conn_R_2})
\begin{eqnarray*}
{\bf P}_{2,\ell}=\{B_{2,\ell,j}: j=1,2,3\},\qquad{\bf P}_{2,\ell'}=\{B_{2,\ell',j}: j=1,2\}.
\end{eqnarray*}
For $k\ge 0$ and $i=1,2,3,4,5$, we denote
\begin{eqnarray*}
\begin{aligned}
&W_1(k)=\{e_2^j e_1 e_3^{k-j}: j=0,1,\ldots,k\},\quad&&W_2(k)=\{e_5^j e_4 e_3^{k-j}: j=0,1,\ldots,k\},\\
&W_3(k)=\{e_8^j e_7 e_9^{k-j}: j=0,1,\ldots,k\},\qquad&&W_4(k)=\{e_{11}^j e_{10} e_9^{k-j}: j=0,1,\ldots,k\},\\
&W_5(k)=\{e_{14}^j e_{13} e_9^{k-j}: j=0,1,\ldots,k\}.
\end{aligned}
\end{eqnarray*}

\begin{center}
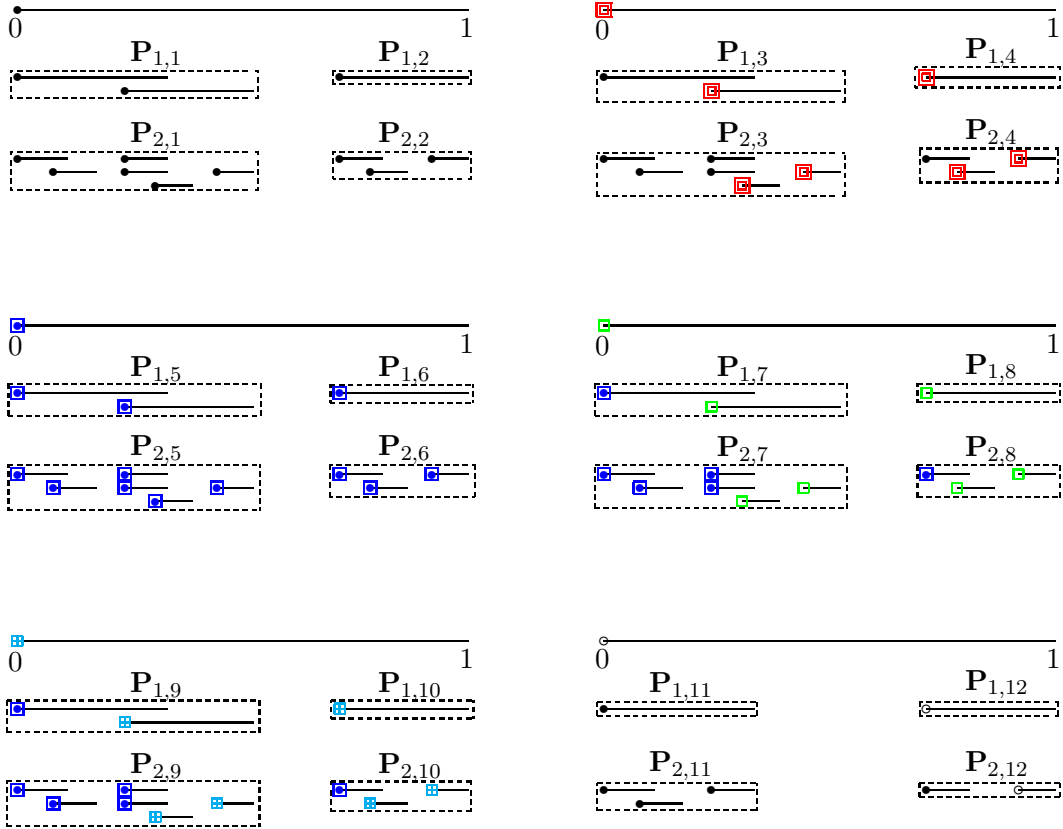
\begin{figure}[h]

\begin{picture}(400,340)
\unitlength=0.06cm
\put(0,190){\line(1,0){100}}
\put(-2,184){$0$}
\put(98,184){$1$}
\put(0,190){\circle*{1.5}}

\put(0,175){\line(1,0){33.33}}
\put(0,175){\circle*{1.5}}
\put(23.81,172){\line(1,0){28.57}}
\put(23.81,172){\circle*{1.5}}
\put(25,178.75){$\mathbf{P}_{1,1}$}
\put(-1.5,170.5){\dashbox(54.88,6){}}

\put(71.43,175){\line(1,0){28.57}}
\put(71.43,175){\circle*{1.5}}
\put(80,178.75){$\mathbf{P}_{1,2}$}
\put(70,173.5){\dashbox(30.57,3){}}

\put(0,157){\line(1,0){11.11}}
\put(0,157){\circle*{1.5}}
\put(7.94,154){\line(1,0){9.52}}
\put(7.94,154){\circle*{1.5}}

\put(23.81,157){\line(1,0){9.52}}
\put(23.81,157){\circle*{1.5}}
\put(23.81,154){\line(1,0){9.52}}
\put(23.81,154){\circle*{1.5}}
\put(30.61,151){\line(1,0){8.163}}
\put(30.61,151){\circle*{1.5}}

\put(44.21,154){\line(1,0){8.163}}
\put(44.21,154){\circle*{1.5}}
\put(-1.5,150){\dashbox(54.873,8.5){}}
\put(25,160.75){$\mathbf{P}_{2,1}$}

\put(71.43,157){\line(1,0){9.52}}
\put(71.43,157){\circle*{1.5}}
\put(78.23,154){\line(1,0){8.163}}
\put(78.23,154){\circle*{1.5}}

\put(91.84,157){\line(1,0){8.163}}
\put(91.84,157){\circle*{1.5}}
\put(70,152.5){\dashbox(30.57,6){}}
\put(80,160.75){$\mathbf{P}_{2,2}$}


\put(130,190){\line(1,0){100}}
\put(128,183.5){$0$}
\put(228,184){$1$}
\put(129.25,189.25){{\color{red}\framebox(1.5,1.5){}}}
\put(128.5,188.5){{\color{red}\framebox(3,3){}}}

\put(130,175){\line(1,0){33.33}}
\put(130,175){\circle*{1.5}}
\put(153.81,172){\line(1,0){28.57}}
\put(153.06,171.25){{\color{red}\framebox(1.5,1.5){}}}
\put(152.31,170.5){{\color{red}\framebox(3,3){}}}
\put(155,178.75){$\mathbf{P}_{1,3}$}
\put(128.5,169.5){\dashbox(54.88,7){}}

\put(201.43,175){\line(1,0){28.57}}
\put(200.68,174.25){{\color{red}\framebox(1.5,1.5){}}}
\put(199.93,173.5){{\color{red}\framebox(3,3){}}}

\put(210,179.25){$\mathbf{P}_{1,4}$}
\put(199,172.75){\dashbox(32.07,4.5){}}


\put(130,157){\line(1,0){11.11}}
\put(130,157){\circle*{1.5}}
\put(137.94,154){\line(1,0){9.52}}
\put(137.94,154){\circle*{1.5}}

\put(153.81,157){\line(1,0){9.52}}
\put(153.81,157){\circle*{1.5}}
\put(153.81,154){\line(1,0){9.52}}
\put(153.81,154){\circle*{1.5}}
\put(160.61,151){\line(1,0){8.163}}
\put(159.86,150.25){{\color{red}\framebox(1.5,1.5){}}}
\put(159.11,149.5){{\color{red}\framebox(3,3){}}}

\put(174.21,154){\line(1,0){8.163}}
\put(173.46,153.25){{\color{red}\framebox(1.5,1.5){}}}
\put(172.71,152.5){{\color{red}\framebox(3,3){}}}
\put(128.5,148.75){\dashbox(54.873,9.5){}}
\put(155,160.75){$\mathbf{P}_{2,3}$}

\put(201.43,157){\line(1,0){9.52}}
\put(201.43,157){\circle*{1.5}}
\put(208.23,154){\line(1,0){8.163}}
\put(207.48,153.25){{\color{red}\framebox(1.5,1.5){}}}
\put(206.73,152.5){{\color{red}\framebox(3,3){}}}

\put(221.84,157){\line(1,0){8.163}}
\put(221.09,156.25){{\color{red}\framebox(1.5,1.5){}}}
\put(220.34,155.5){{\color{red}\framebox(3,3){}}}
\put(200,151.75){\dashbox(30.57,7.5){}}
\put(210,161.25){$\mathbf{P}_{2,4}$}


\put(0,120){\line(1,0){100}}
\put(-2,113.5){$0$}
\put(98,114){$1$}
\put(0,120){{\color{blue}\circle*{1.5}}}
\put(-1.25,118.75){{\color{blue}\framebox(2.5,2.5){}}}

\put(0,105){\line(1,0){33.33}}
\put(0,105){{\color{blue}\circle*{1.5}}}
\put(-1.25,103.75){{\color{blue}\framebox(2.5,2.5){}}}
\put(23.81,102){\line(1,0){28.57}}
\put(23.81,102){{\color{blue}\circle*{1.5}}}
\put(22.56,100.75){{\color{blue}\framebox(2.5,2.5){}}}
\put(25,108.75){$\mathbf{P}_{1,5}$}
\put(-2,100){\dashbox(55.88,7){}}

\put(71.43,105){\line(1,0){28.57}}
\put(71.43,105){{\color{blue}\circle*{1.5}}}
\put(70.18,103.75){{\color{blue}\framebox(2.5,2.5){}}}
\put(80,108.75){$\mathbf{P}_{1,6}$}
\put(69.35,102.85){\dashbox(31.57,4){}}

\put(0,87){\line(1,0){11.11}}
\put(0,87){{\color{blue}\circle*{1.5}}}
\put(-1.25,85.75){{\color{blue}\framebox(2.5,2.5){}}}
\put(7.94,84){\line(1,0){9.52}}
\put(7.94,84){{\color{blue}\circle*{1.5}}}
\put(6.69,82.75){{\color{blue}\framebox(2.5,2.5){}}}

\put(23.81,87){\line(1,0){9.52}}
\put(23.81,87){{\color{blue}\circle*{1.5}}}
\put(22.56,85.75){{\color{blue}\framebox(2.5,2.5){}}}
\put(23.81,84){\line(1,0){9.52}}
\put(23.81,84){{\color{blue}\circle*{1.5}}}
\put(22.56,82.75){{\color{blue}\framebox(2.5,2.5){}}}
\put(30.61,81){\line(1,0){8.163}}
\put(30.61,81){{\color{blue}\circle*{1.5}}}
\put(29.36,79.75){{\color{blue}\framebox(2.5,2.5){}}}

\put(44.21,84){\line(1,0){8.163}}
\put(44.21,84){{\color{blue}\circle*{1.5}}}
\put(42.96,82.75){{\color{blue}\framebox(2.5,2.5){}}}
\put(-2,79.05){\dashbox(55.873,10){}}
\put(25,91.25){$\mathbf{P}_{2,5}$}

\put(71.43,87){\line(1,0){9.52}}
\put(71.43,87){{\color{blue}\circle*{1.5}}}
\put(70.18,85.75){{\color{blue}\framebox(2.5,2.5){}}}
\put(78.23,84){\line(1,0){8.163}}
\put(78.23,84){{\color{blue}\circle*{1.5}}}
\put(76.98,82.75){{\color{blue}\framebox(2.5,2.5){}}}
\put(91.84,87){\line(1,0){8.163}}
\put(91.84,87){{\color{blue}\circle*{1.5}}}
\put(90.59,85.75){{\color{blue}\framebox(2.5,2.5){}}}
\put(69.25,82.05){\dashbox(32.07,7){}}
\put(80,91.25){$\mathbf{P}_{2,6}$}


\put(130,120){\line(1,0){100}}
\put(128,113.5){$0$}
\put(228,114){$1$}
\put(129,119){{\color{green}\framebox(2,2){}}}

\put(130,105){\line(1,0){33.33}}
\put(130,105){{\color{blue}\circle*{1.5}}}
\put(128.75,103.75){{\color{blue}\framebox(2.5,2.5){}}}
\put(153.81,102){\line(1,0){28.57}}
\put(152.81,101){{\color{green}\framebox(2,2){}}}
\put(155,108.75){$\mathbf{P}_{1,7}$}
\put(128,100){\dashbox(55.88,7){}}

\put(201.43,105){\line(1,0){28.57}}
\put(200.43,104){{\color{green}\framebox(2,2){}}}
\put(210,109.25){$\mathbf{P}_{1,8}$}
\put(199.5,103){\dashbox(31.57,4){}}


\put(130,87){\line(1,0){11.11}}
\put(130,87){{\color{blue}\circle*{1.5}}}
\put(128.75,85.75){{\color{blue}\framebox(2.5,2.5){}}}
\put(137.94,84){\line(1,0){9.52}}
\put(137.94,84){{\color{blue}\circle*{1.5}}}
\put(136.69,82.75){{\color{blue}\framebox(2.5,2.5){}}}

\put(153.81,87){\line(1,0){9.52}}
\put(153.81,87){{\color{blue}\circle*{1.5}}}
\put(152.56,85.75){{\color{blue}\framebox(2.5,2.5){}}}
\put(153.81,84){\line(1,0){9.52}}
\put(153.81,84){{\color{blue}\circle*{1.5}}}
\put(152.56,82.75){{\color{blue}\framebox(2.5,2.5){}}}
\put(160.61,81){\line(1,0){8.163}}
\put(159.61,80){{\color{green}\framebox(2,2){}}}

\put(174.21,84){\line(1,0){8.163}}
\put(173.21,83){{\color{green}\framebox(2,2){}}}
\put(128,79.5){\dashbox(55.873,9.5){}}
\put(155,91){$\mathbf{P}_{2,7}$}

\put(201.43,87){\line(1,0){9.52}}
\put(201.43,87){{\color{blue}\circle*{1.5}}}
\put(200.18,85.75){{\color{blue}\framebox(2.5,2.5){}}}
\put(208.23,84){\line(1,0){8.163}}
\put(207.23,83){{\color{green}\framebox(2,2){}}}

\put(221.84,87){\line(1,0){8.163}}
\put(220.84,86){{\color{green}\framebox(2,2){}}}
\put(199.5,82){\dashbox(31.57,7){}}
\put(210,91){$\mathbf{P}_{2,8}$}


\put(0,50){\line(1,0){100}}
\put(-2,43.5){$0$}
\put(98,44){$1$}
\put(-1,49){{\color{cyan}\framebox(2,2){}}}
\put(-1,50){{\color{cyan}\line(1,0){2}}}
\put(0,49){{\color{cyan}\line(0,1){2}}}

\put(0,35){\line(1,0){33.33}}
\put(0,35){{\color{blue}\circle*{1.5}}}
\put(-1.25,33.75){{\color{blue}\framebox(2.5,2.5){}}}
\put(23.81,32){\line(1,0){28.57}}
\put(22.81,31){{\color{cyan}\framebox(2,2){}}}
\put(22.81,32){{\color{cyan}\line(1,0){2}}}
\put(23.81,31){{\color{cyan}\line(0,1){2}}}
\put(25,38.75){$\mathbf{P}_{1,9}$}
\put(-2.25,29.85){\dashbox(55.88,7){}}

\put(71.43,35){\line(1,0){28.57}}
\put(70.43,34){{\color{cyan}\framebox(2,2){}}}
\put(70.43,35){{\color{cyan}\line(1,0){2}}}
\put(71.43,34){{\color{cyan}\line(0,1){2}}}
\put(80,38.75){$\mathbf{P}_{1,10}$}
\put(69.5,32.85){\dashbox(31.57,4){}}

\put(0,17){\line(1,0){11.11}}
\put(0,17){{\color{blue}\circle*{1.5}}}
\put(-1.25,15.75){{\color{blue}\framebox(2.5,2.5){}}}

\put(7.94,14){\line(1,0){9.52}}
\put(7.94,14){{\color{blue}\circle*{1.5}}}
\put(6.69,12.75){{\color{blue}\framebox(2.5,2.5){}}}

\put(23.81,17){\line(1,0){9.52}}
\put(23.81,17){{\color{blue}\circle*{1.5}}}
\put(22.56,15.75){{\color{blue}\framebox(2.5,2.5){}}}

\put(23.81,14){\line(1,0){9.52}}
\put(23.81,14){{\color{blue}\circle*{1.5}}}
\put(22.56,12.75){{\color{blue}\framebox(2.5,2.5){}}}

\put(30.61,11){\line(1,0){8.163}}
\put(29.61,10){{\color{cyan}\framebox(2,2){}}}
\put(29.61,11){{\color{cyan}\line(1,0){2}}}
\put(30.61,10){{\color{cyan}\line(0,1){2}}}

\put(44.21,14){\line(1,0){8.163}}
\put(43.21,13){{\color{cyan}\framebox(2,2){}}}
\put(43.21,14){{\color{cyan}\line(1,0){2}}}
\put(44.21,13){{\color{cyan}\line(0,1){2}}}
\put(-2.25,9){\dashbox(55.873,10){}}
\put(25,20.75){$\mathbf{P}_{2,9}$}

\put(71.43,17){\line(1,0){9.52}}
\put(71.43,17){{\color{blue}\circle*{1.5}}}
\put(70.18,15.75){{\color{blue}\framebox(2.5,2.5){}}}

\put(78.23,14){\line(1,0){8.163}}
\put(77.23,13){{\color{cyan}\framebox(2,2){}}}
\put(77.23,14){{\color{cyan}\line(1,0){2}}}
\put(78.23,13){{\color{cyan}\line(0,1){2}}}

\put(91.84,17){\line(1,0){8.163}}
\put(90.84,16){{\color{cyan}\framebox(2,2){}}}
\put(90.84,17){{\color{cyan}\line(1,0){2}}}
\put(91.84,16){{\color{cyan}\line(0,1){2}}}
\put(69.5,12.35){\dashbox(31.07,6.5){}}
\put(80,20.75){$\mathbf{P}_{2,10}$}


\put(130,50){\line(1,0){100}}
\put(128,44){$0$}
\put(228,44){$1$}
\put(130,50){\circle{1.5}}

\put(130,35){\line(1,0){33.33}}
\put(130,35){\circle*{1.5}}
\put(140,38.75){$\mathbf{P}_{1,11}$}
\put(128.5,33.5){\dashbox(35.33,3){}}

\put(201.43,35){\line(1,0){28.57}}
\put(201.43,35){\circle{1.5}}
\put(210,39.25){$\mathbf{P}_{1,12}$}
\put(200,33.5){\dashbox(30.57,3){}}

\put(130,17){\line(1,0){11.11}}
\put(130,17){\circle*{1.5}}
\put(137.94,14){\line(1,0){9.52}}
\put(137.94,14){\circle*{1.5}}

\put(153.81,17){\line(1,0){9.52}}
\put(153.81,17){\circle*{1.5}}
\put(128.5,12.5){\dashbox(35.33,6){}}
\put(140,21.25){$\mathbf{P}_{2,11}$}
%
%

\put(201.43,17){\line(1,0){9.52}}
\put(201.43,17){\circle*{1.5}}

\put(221.84,17){\line(1,0){8.163}}
\put(221.84,17){\circle{1.5}}
\put(199.93,15.5){\dashbox(31.07,3){}}
\put(210,21.25){$\mathbf{P}_{2,12}$}
\end{picture}

\caption{$\mu$-partitions $\mathbf{P}_{k,\ell}$ of $B_{1,\ell}$ for the GIFS defined in \eqref{E:exam_str_con_R_2_similitudes}. The figure is drawn with $\rho=1/3$ and $r=2/7$.}
\label{fig:fig2_not_str_conn_R_2}
\end{figure}
\end{center}

\begin{lem}
Let $\{S_{e_j}\}_{j=1}^{17}$ be defined as in \eqref{E:exam_str_con_R_2_similitudes}. Then
\begin{enumerate}
\item[(a)]
$S_{e_ie_3}=S_{e_{i+1}e_i}$ and $S_{e_{i'}e_9}=S_{e_{i'+1}e_{i'}}$ for $i=1,4$ and $i'=7,10,13$.
Moreover, $S_{\bf e}=S_{\bf e'}$, for any ${\bf e}, {\bf e'}\in W_i(k)$ and $i=1,2,3,4,5$.
\item[(b)] For any $k\ge 1$,
\begin{eqnarray*}
S_{e_2^ke_1}(\Omega_1)=S_{e_5^ke_4}(\Omega_1)=S_{e_8^ke_7}(\Omega_3)=S_{e_{14}^k e_{13}}(\Omega_3)=S_{e_{11}^k e_{10}}(\Omega_3).
\end{eqnarray*}
\end{enumerate}
\end{lem}

\begin{lem}\label{L:mu_equi_not_str_conn_R_2}
Assume the hypotheses of Example \ref{E:exam_GIFS_not_str_con_R_2}, and let $w_i(k)$ be defined as in \eqref{defi:w_i(k)}. Then
\begin{enumerate}
\item[(a)] For $j=1,2$,
\begin{eqnarray*}
\begin{aligned}
&\mu_i|_{B_{2,2i,j}}=p_{e_{3i}}\mu_i|_{B_{1,j+2i-2}}\circ S^{-1}_{e_{3i}}\quad\mbox{ for }i=1,2,3,4,5,\\
&\mu_6|_{B_{2,11,j}}=p_{e_{16}}\mu_1|_{B_{1,j}}\circ S^{-1}_{e_{16}},\qquad
\mu_6|_{B_{2,12,j}}=p_{e_{17}}\mu_6|_{B_{1,j}}\circ S^{-1}_{e_{17}},
\end{aligned}
\end{eqnarray*}
\item[(b)] For $k\ge2$,
\begin{eqnarray*}
\begin{aligned}
\mu_i|_{B_{k,2i-1,1}}=&w_i(k-2)\mu_1|_{B_{1,1}}\circ S^{-1}_{e^{k-2}_{3i-1}e_{3i-2}} \quad\mbox{ for }i=1,2,\\
\mu_i|_{B_{k,2i-1,1}}=&w_i(k-2)\mu_3|_{B_{1,5}}\circ S^{-1}_{e^{k-2}_{3i-1}e_{3i-2}} \quad\mbox{ for }i=3,4,5;
\end{aligned}
\end{eqnarray*}
\item[(c)] For $k\ge2$,
\begin{eqnarray*}
\begin{aligned}
&\mu_i|_{B_{k,2i-1,2}}=w_i(k-2)\mu_1|_{B_{1,2}}\circ S^{-1}_{e_{3i-1}^{k-2}e_{3i-2}}+p_{e_{3i-1}^{k-1}}\mu_i|_{B_{1,1}}\circ S^{-1}_{e_2^{k-1}}\quad\mbox{ for }i=1,2,\\
&\mu_i|_{B_{k,2i-1,2}}=w_i(k-2)\mu_3|_{B_{1,6}}\circ S^{-1}_{e_{3i-1}^{k-2}e_{3i-2}}+p_{e_{3i-1}^{k-1}}\mu_i|_{B_{1,2i-1}}\circ S^{-1}_{e_{3i-1}^{k-1}}\quad\mbox{for }i=3,4,5;
\end{aligned}
\end{eqnarray*}
\item[(d)] For $k\ge2$ and $i=1,2,3,4,5$,
\begin{eqnarray*}
\mu_i|_{B_{k,2i-1,3}}=p_{e_{3i-1}^{k-1}} \mu_1|_{B_{1,2i}}\circ S^{-1}_{e_{3i-1}^{k-1}}.
\end{eqnarray*}
\end{enumerate}
\end{lem}

It follows from Lemma \ref{L:mu_equi_not_str_conn_R_2} that
${\bf P}'_{2,\ell}=\{B_{2,\ell,j}: j=1,3\}$, ${\bf P}^\ast_{2,\ell}=\{B_{2,\ell,2}\}$,
${\bf P}'_{2,\ell'}={\bf P}_{2,\ell'}$ and ${\bf P}^\ast_{2,\ell'}=\emptyset$
for $\ell=1,3,5,7,9$, and $\ell'=2,4,6,8,10,11,12$. For $k\ge3$ and $\ell=1,3,5,7,9$, define
\begin{eqnarray*}
{\bf P}_{k,\ell}:={\bf P}'_{k-1,1}\cup\{B_{k,\ell,j}: j=1,2,3\}.
\end{eqnarray*}
By Lemma\ref{L:mu_equi_not_str_conn_R_2} (b,c,d), we get
\begin{eqnarray*}
{\bf P}'_{k,\ell}={\bf P}'_{k-1,\ell}\cup \{B_{k,\ell,j}: j=1,3\}\quad\mbox{and}
\quad{\bf P}^\ast_{k,\ell}=\{B_{k,\ell,2}\}\quad\mbox{for }\ell=1,3,5,7,9.
\end{eqnarray*}
Hence $\Gamma'_1=\{2\}$, $\Gamma^\ast_1=\{1\}$, $\Gamma'_2=\{4\}$, $\Gamma^\ast_2=\{3\}$, $\Gamma'_3=\{6\}$, $\Gamma^\ast_3=\{5\}$, $\Gamma'_4=\{8\}$, $\Gamma^\ast_4=\{7\}$, $\Gamma'_5=\{10\}$, $\Gamma^\ast_5=\{9\}$, $\Gamma'_6=\{11,12\}$, $\Gamma^\ast_6=\emptyset$, $\kappa_\ell=2$ and $\kappa_{\ell'}=\infty$ for $\ell=1,3,5,7,9$ and $\ell'=2,4,6,8,10,11,12$. Also,
$G'_{k,\ell}=\{1,3\}$, $G^\ast_{k,\ell}=\emptyset$ for $k\ge2$ and $\ell=1,3,5,7,9$, and $G'_{2,\ell}=\{1,2\}$, $G^\ast_{2,\ell}=\emptyset$ for $\ell=2,4,6,8,10,11,12$.

\begin{lem}
Let $w(k,\ell,j), c(k,\ell,j), e(k,\ell,j)$ and $\rho_{e(k,\ell,j)}$ be defined as in \eqref{eq:relationship_between_mui_and_B1c}. Then
\begin{enumerate}
\item[(a)] for $k\ge2$,
\begin{eqnarray*}
\begin{aligned}
&w(k,\ell,1)=w_\ell(k-2),\quad \rho_{e(k,\ell,1)}=\rho r^{k-2}\quad\mbox{ for }\ell=1,3,5,7,9;\\
&c(k,\ell,1)=1,\quad c(k,\ell',1)=5\quad\mbox{ for }\ell=1,3\mbox{ and }\ell'=5,7,9;\\
&w(k,\ell,3)=p_{e_{(3\ell+1)/2}^{k-1}},\quad c(k,\ell,3)=\ell+1,\quad \rho_{e(k,\ell,3)}=r^{k-1}\quad\mbox{ for }\ell=1,3,5,7,9;\\
\end{aligned}
\end{eqnarray*}
\item[(b)] for $j=1,2$,
\begin{eqnarray*}
\begin{aligned}
&w(2,\ell,j)=p_{e_{3\ell/2}},\quad c(2,\ell,j)=j+\ell-2,\quad \rho_{e(2,\ell,j)}=r\mbox{ for }\ell=2,4,6,8,10;\\
&w(2,\ell,j)=p_{e_{\ell+5}},\quad c(2,\ell,j)=j,\quad\mbox{ for }\ell=11,12;\\
&\rho_{e(2,11,j)}=\rho,\quad \rho_{e(2,12,j)}=r.
\end{aligned}
\end{eqnarray*}
\end{enumerate}
\end{lem}

Using the above results, we can express the vector-valued equations \eqref{eq:f_ell_'_1} and \eqref{eq:f_ell_ast_1} as follows
\begin{eqnarray*}
\begin{aligned}
f_\ell(x)=&\sum_{k=2}^\infty\Big(w_{(\ell+1)/2}^q(k-2)(\rho r^{k-2})^{-\alpha_{(\ell+1)/2}}f_1(x+\ln(\rho r^{k-2}))\\
&\qquad+(p^q_{e_{(3\ell+1)/2}}r^{-\alpha_{(\ell+1)/2}})^{k-1}f_{\ell+1}(x+\ln(r^{k-1}))\Big)\\
&+z_\ell^{(\alpha_{(\ell+1)/2})}(x)-z_{\ell,\infty}^{(\alpha_{(\ell+1)/2})}(x)\qquad\mbox{ for }\ell=1,3,\\
\end{aligned}
\end{eqnarray*}
\begin{eqnarray*}
\begin{aligned}
f_\ell(x)=&\sum_{k=2}^\infty\Big(w_{(\ell+1)/2}^q(k-2)(\rho r^{k-2})^{-\alpha_{(\ell+1)/2}}f_5(x+\ln(\rho r^{k-2}))\\
&\qquad+(p^q_{e_{(3\ell+1)/2}}r^{-\alpha_{(\ell+1)/2}})^{k-1}f_{\ell+1}(x+\ln(r^{k-1}))\Big)\\
&+z_\ell^{(\alpha_{(\ell+1)/2})}(x)-z_{\ell,\infty}^{(\alpha_{(\ell+1)/2})}(x)\qquad\mbox{ for }\ell=5,7,9,\\
\end{aligned}
\end{eqnarray*}
\begin{eqnarray*}
\begin{aligned}
f_\ell(x)=&p^q_{e_{3\ell/2}}r^{-\alpha_{\ell/2}}\sum_{j=\ell-1}^\ell f_j(x+\ln r)+z_\ell^{(\alpha_{\ell/2})}(x)\quad\mbox{ for }\ell=2,4,6,8,10,\\
\end{aligned}
\end{eqnarray*}
\begin{eqnarray*}
f_{11}(x)=&p^q_{e_{16}}\rho^{-\alpha_6}\sum_{j=1}^{2} f_j(x+\ln \rho)+z_{11}^{(\alpha_6)}(x),
\end{eqnarray*}
\begin{eqnarray*}
f_{12}(x)=&p^q_{e_{17}}r^{-\alpha_6}\sum_{j=11}^{12} f_j(x+\ln r)+z_{12}^{(\alpha_6)}(x),
\end{eqnarray*}
with $z_\ell^{(\alpha_i)}(x)=E_\ell^{(\alpha_i)}(e^{-x})$ for $\ell\in\Gamma_i$ and $i=1,2,3,4,5,6$,
and $z_{\ell',\infty}^{(\alpha_{(\ell'+1)/2})}(x)=E_{\ell',\infty}^{(\alpha_{(\ell'+1)/2})}(e^{-x})$ for $\ell'=1,3,5,7,9$.
For $m=1,2,\ldots,6$, $i\in\mathcal{SC}_m$ and $\ell,\ell'\in\Gamma_i$, let $\mu^{(\alpha_m)}_{\ell\ell'}$ be the discrete measures defined as in \eqref{defi:muij_on_kuangjia}.
Then, for $k\ge 2$,
\begin{eqnarray*}
\begin{aligned}
&\mu^{(\alpha_{(\ell+1)/2})}_{\ell1}(-\ln(\rho r^{k-2}))=w_{(\ell+1)/2}^q(k-2)(\rho r^{k-2})^{-\alpha_{(\ell+1)/2}}\qquad\mbox{ for }\ell=1,3,\\
&\mu^{(\alpha_{(\ell+1)/2})}_{\ell5}(-\ln(\rho r^{k-2}))=w_{(\ell+1)/2}^q(k-2)(\rho r^{k-2})^{-\alpha_{(\ell+1)/2}}\qquad\mbox{ for }\ell=5,7,9,\\
&\mu^{(\alpha_{(\ell+1)/2})}_{\ell \ell+1}(-\ln(r^{k-1}))=(p^q_{e_{(3\ell+1)/2}}r^{-\alpha_{(\ell+1)/2}})^{k-1}\qquad\mbox{ for }\ell=1,3,5,7,9,\\
\end{aligned}
\end{eqnarray*}
and
\begin{eqnarray*}
\begin{aligned}
&\mu^{(\alpha_{\ell/2})}_{\ell\ell'}(-\ln r)=p^q_{e_{(3\ell)/2}}r^{-\alpha_{\ell/2}}\quad\mbox{ for }\ell=2,4,6,8,10,\mbox{ and }\ell'=\ell-1,\ell,\\
&\mu^{(\alpha_6)}_{11 \ell'}(-\ln \rho)=p^q_{e_{16}}\rho^{-\alpha_6}\quad\mbox{ for }\ell'=1,2,\\
&\mu^{(\alpha_6)}_{12 \ell'}(-\ln r)=p^q_{e_{17}}r^{-\alpha_6}\quad\mbox{ for }\ell'=11,12.
\end{aligned}
\end{eqnarray*}
Moreover,
\begin{eqnarray*}
{\bf M}_\alpha(\infty)=
\left(
  \begin{array}{cccccccccccc}
    \eta_1 & \vartheta_1 & 0 & 0 & 0 & 0 & 0 & 0 & 0 & 0 & 0 & 0 \\
    \iota_1 & \iota_1 & 0 & 0 & 0 & 0 & 0 & 0 & 0 & 0 & 0 & 0 \\
    \eta_2 & 0 & 0 & \vartheta_2 & 0 & 0 & 0 & 0 & 0 & 0 & 0 & 0 \\
    0 & 0 & \iota_2 & \iota_2 & 0 & 0 & 0 & 0 & 0 & 0 & 0 & 0 \\
    0 & 0 & 0 & 0 & \eta_3 & \vartheta_3 & 0 & 0 & 0 & 0 & 0 & 0 \\
    0 & 0 & 0 & 0 & \iota_3 & \iota_3 & 0 & 0 & 0 & 0 & 0 & 0 \\
    0 & 0 & 0 & 0 & \eta_4 & 0 & 0 & \vartheta_4 & 0 & 0 & 0 & 0 \\
    0 & 0 & 0 & 0 & 0 & 0 & \iota_4 & \iota_4 & 0 & 0 & 0 & 0 \\
    0 & 0 & 0 & 0 & \eta_5 & 0 & 0 & 0 & 0 & \vartheta_5 & 0 & 0 \\
    0 & 0 & 0 & 0 & 0 & 0 & 0 & 0 & \iota_5 & \iota_5 & 0 & 0 \\
    \iota_6 & \iota_6 & 0 & 0 & 0 & 0 & 0 & 0 & 0 & 0 & 0 & 0 \\
    0 & 0 & 0 & 0 & 0 & 0 & 0 & 0 & 0 & 0 & \iota_7 & \iota_7 \\
  \end{array}
\right),
\end{eqnarray*}
where
$\eta_i=\sum_{k=0}^\infty w_i^q(k)(\rho r^k)^{-\alpha_i}$,
$\vartheta_i=\frac{p^q_{e_{3i-1}}r^{-\alpha_i}}{1-p^q_{e_{3i-1}}r^{-\alpha_i}}$,
$\iota_i=p^q_{e_{3i}}r^{-\alpha_i}$, and $\iota_j=p^q_{e_{i+10}} r^{-\alpha_6}$ for $i=1,2,3,4,5$ and $j=6,7$.

In the following, we firstly show that $D_\ell$ is open for $\ell=1,3,5,7,9$, and then prove that the error terms
$z_\ell^{(\alpha_i)}(x)=o(e^{-\epsilon x})$ for $i=1,2,3,4,5,6$, $\ell\in\Gamma_i$,
and $z_{\ell',\infty}^{(\alpha_{(\ell'+1)/2})}(x)=o(e^{-\epsilon x})$ for $\ell'=1,3,5,7,9$ as $x\to\infty$.

\begin{prop}\label{P:not_str_con_R_2_D_ell_open}
For $q\ge 0$ and $\ell=1,3,5,7,9$, $D_\ell$ is open.
\end{prop}

\begin{prop}
For any $k\ge N+1$, $\ell=1,5$ and $\ell'=2,4,6,8,10$,
\begin{eqnarray*}
\Phi_\ell^{(\alpha_{(\ell+1)/2})}(h/{\rho r^{k-2}})\le 1\quad\mbox{ and }\quad\Phi_{\ell'}^{(\alpha_{\ell'/2})}(h/{r^{k-1}})\le 1.
\end{eqnarray*}
\end{prop}

\begin{prop}\label{P:not_str_con_R_2_error}
For $q\ge 0$, assume that $\alpha_m\in D_\ell$ for $m=1,2,\ldots,6$, $i\in\mathcal{SC}_m$ and $\ell\in\Gamma_i$. Then there exists some $\epsilon>0$ such that
\begin{enumerate}
\item[(a)]
$\sum\limits_{k=N+1}^\infty\Big(w_{(\ell+1)/2}^q(k-2)(\rho r^{k-2})^{-\alpha_1}\Phi_1^{(\alpha_1)}(h/{\rho r^{k-2}})$\\
    $\quad+(p^q_{e_{(3\ell+1)/2}}r^{-\alpha_{(\ell+1)/2}})^{k-1}\Phi_{\ell+1}^{(\alpha_{(\ell+1)/2})}(h/{r^{k-1}})\Big)=o(h^{\epsilon/2})$
    for $\ell=1,3$;
\item[(b)] $\sum\limits_{k=N+1}^\infty\Big(w_{(\ell+1)/2}^q(k-2)(\rho r^{k-2})^{-\alpha_3}\Phi_5^{(\alpha_3)}(h/{\rho r^{k-2}})$\\
    $\quad+(p^q_{e_{(3\ell+1)/2}}r^{-\alpha_{(\ell+1)/2}})^{k-1}\Phi_{\ell+1}^{(\alpha_{(\ell+1)/2})}(h/{r^{k-1}})\Big)=o(h^{\epsilon/2})$ for $\ell=5,7,9$;
\item[(c)] $h^{-(1+\alpha_{(\ell+1)/2})}\Big(\sum\limits_{k=2}^N e_{k,\ell}+\int_{B_{N,\ell,2}}\mu(B_h(x))^q\,dx\Big)=o(h^{\epsilon/2})$
           for $\ell=1,3,5,7,9$;
\item[(d)] $h^{-(1+\alpha_{\ell/2})}e_{\ell,2}=o(h^\epsilon/2)$ for $\ell=2,4,6,8,10$, and $h^{-(1+\alpha_6)}e_{\ell,2}=o(h^\epsilon/2)$ for $\ell=11,12$.
\end{enumerate}
\end{prop}

\begin{proof}[Proof of Corollary \ref{C:cor_GIFS_not_str_con_R_2}]
It follows from Theorem \ref{T:main_thm_not_str_con} and Proposition \ref{P:not_str_con_R_2_error} that $\tau(q)=\alpha$.

The proof of the differentially of $\tau_q$ is similar to that of Corollary \ref{C:cor_GIFS_str_con_R2}. We omit the details.
\end{proof}

\subsection{A GIFS that is not strongly connected on $\R^2$}

In this subsection, we study the graph-directed self-similar measure $\mu$ defined by the GIFS in Example \ref{E:exam_GIFS_not_str_con_R2} and compute the $L^q$-spectrum of $\mu$.

\begin{prop}\label{P:prop_GIFS_not_strong_con_R2}
Let $\mu=\sum_{i=1}^2\mu_i$ be the graph-directed self-similar measure defined by a GIFS in Example \ref{E:exam_GIFS_not_str_con_R2} together with a probability matrix $(p_e)_{e\in E}$. Then $\mu$ satisfies (EFT) with
$\Omega=\{\Omega_i\}_{i=1}^2$ being an EFT-family, where $\Omega_1=\bigcup_{x\in(0,1)}(0,1)\times (x,1-x)$ and $\Omega_2=(2,3)\times(0,1)$. Moreover, there exists a weakly regular basic pair.
\end{prop}

Let $\{S_{e_i}\}_{i=1}^7$ be defined as in \eqref{E:exam_str_con_R2_similitudes}, $(p_e)_{e\in E}$ be a probability matrix, and $\mu$ be a graph-directed self-similar measure defined as in Proposition \ref{P:prop_GIFS_not_strong_con_R2}.
It is easy to see that $\eta=2$, $\sum_{i=1}^3 p_{e_i}=1$, $\sum_{i=4}^7 p_{e_i}=1$,
$\mu_1=p_{e_1}\mu_2\circ S_{e_1}^{-1}+\sum_{j=2}^3 p_{e_j}\mu_1\circ S_{e_j}^{-1}$,
and $\mu_2=\sum_{j=4}^7 p_{e_j}\mu_2\circ S_{e_j}^{-1}$.

Define (see Figure \ref{F:fig_not_str_con_R2_first_iteration})
\begin{eqnarray}\label{eq:not_str_con_R2_B_ell}
\begin{aligned}
&B_{1,1}:=S_{e_1}(\Omega_2),\qquad &&B_{1,\ell}:=S_{e_\ell}(\Omega_1)\quad\mbox{ for }\ell=2,3,\\
&B_{1,4}:=\bigcup_{j=4}^5 S_{e_j}(\Omega_2),\qquad &&B_{1,\ell}:=S_{e_{\ell+1}}(\Omega_2)\quad\mbox{ for }\ell=5,6.
\end{aligned}
\end{eqnarray}

Let $\Gamma=\{1,2,3,4,5,6\}$ and ${\bf B}:=\{B_{1,\ell}: \ell\in\Gamma\}$. Then $i\in\mathcal{SC}_i$ for $i=1,2$, $\Gamma_1=\{1,2,3\}$ and $\Gamma_2=\{4,5,6\}$.
Denote
\begin{eqnarray*}
\begin{aligned}
B_{2,1,j}:=S_{e_1}(B_{1,j+3}), B_{2,\ell,j}:=S_{e_\ell}(B_{1,j}), B_{2,\ell',j}:=S_{e_{\ell'+1}}(B_{1,j+3})\mbox{ for }j=1,2,3, \ell=2,3\mbox{ and }\ell'=5,6,
\end{aligned}
\end{eqnarray*}
and
\begin{eqnarray*}
B_{k,4,1}:= S_{e_5^{k-2}e_4}(B_{1,4}),\quad B_{k,4,2}:= S_{e_5^{k-2}e_4}(B_{1,6}),\quad B_{k,4,j}:= S_{e_5^{k-1}}(B_{1,j+1})\mbox{ for }k\ge2\mbox{ and }j=3,4,5.
\end{eqnarray*}
Define ${\bf P}_{1,\ell}:=\{B_{1,\ell}\}$ for $\ell\in\Gamma$, and let
\begin{eqnarray*}
{\bf P}_{2,4}=\{B_{2,4,j}: j=1,2,3,4,5\}\qquad{\bf P}_{2,\ell}=\{B_{2,\ell,j}: j=1,2,3\}\quad\mbox{ for }\ell\in\Gamma\backslash\{4\}.
\end{eqnarray*}

For $k\ge0$, define
$$w_k:=p_{e_4}\sum_{j=0}^k p_{e_5^j} p_{e_6^{k-j}}.$$

\begin{lem}\label{L:GIFS_not_strong_con_R2_1}
Let $\{S_{e_i}\}_{i=1}^7$ be defined as in Example \ref{E:exam_GIFS_not_str_con_R2}. Then $S_{e_4e_6}=S_{e_5e_4}$
and for $k\ge0$ and any $\e, \e'\in W_k:=\{e_5^je_4e_6^{k-j}: j=0,1,\ldots,k\}$, $S_{\e}=S_{\e'}$.
\end{lem}

\begin{lem}\label{L:GIFS_not_strong_con_R2_2}
Let $\{B_{1,\ell}\}_{\ell=1}^6$ be defined as in \eqref{eq:not_str_con_R2_B_ell}.
Then
\begin{enumerate}
\item[(a)] $\mu_1|_{B_{2,1,j}}= p_{e_1}\mu_2|_{B_{1,j+3}}\circ S^{-1}_{e_1}$ for $j=1,2,3$;
\item[(b)] $\mu_1|_{B_{2,\ell,j}}= p_{e_j}\mu_2|_{B_{1,\ell}}\circ S^{-1}_{e_j}$ for $\ell=2,3$ and $j=1,2,3$;
\item[(c)] $\mu_2|_{B_{2,\ell,j}}= p_{e_{\ell+1}}\mu_2|_{B_{1,j+3}}\circ S^{-1}_{e_{\ell+1}}$ for $\ell=5,6$ and $j=1,2,3$;
\item[(d)] for $k\ge2$,
\begin{eqnarray*}
\begin{aligned}
&\mu_2|_{B_{k,4,1}}=w_{k-2}\mu_2|_{B_{1,4}}\circ S^{-1}_{e_5^{k-2} e_4},\quad
\mu_2|_{B_{k,4,2}}=w_{k-2}\mu_2|_{B_{1,6}}\circ S^{-1}_{e_5^{k-2} e_4},\\
&\mu_2|_{B_{k,4,3}}= w_{k-2}\mu_2|_{B_{1,5}}\circ S^{-1}_{e_5^{k-1} e_4}
+p_{e_5^{k-1}}\cdot \mu_2|_{B_{1,4}}\circ S^{-1}_{e_5^{k-1}},\\
&\mu_2|_{B_{k,4,j}}=p_{e_5^{k-1}}\mu_2|_{B_{1,j+1}}\circ S^{-1}_{e_5^{k-1}}\quad \mbox{ for } j=4,5.
\end{aligned}
\end{eqnarray*}
\end{enumerate}
\end{lem}

\begin{center}
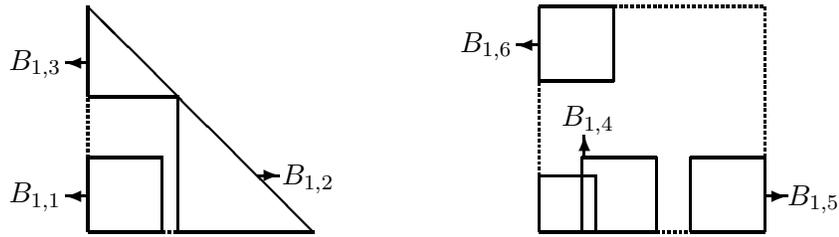
\begin{figure}[h]

\begin{picture}(280,120)
\unitlength=0.30cm

\thicklines


\put(0,6.5){\line(0,1){4}}
\put(4,0.5){\line(1,0){6}}
\put(4,0.5){\line(0,1){6}}
\put(10,0.5){\line(-1,1){10}}

\put(0,6.5){\line(1,0){4}}
\put(0,0.5){\line(0,1){3.3}}
\put(0,0.5){\line(1,0){3.3}}
\put(3.3,0.5){\line(0,1){3.3}}
\put(0,3.8){\line(1,0){3.3}}
\multiput(3.3,0.5)(0.25,0){3}{\line(1,0){0.125}}
\multiput(0,3.8)(0,0.25){12}{\line(0,1){0.125}}

\put(0,2.1){\vector(-1,0){1}}
\put(-3.5,1.8){$B_{1,1}$}
\put(7.5,3){\vector(1,0){1}}
\put(8.6,2.6){$B_{1,2}$}
\put(0,8){\vector(-1,0){1}}
\put(-3.5,7.7){$B_{1,3}$}

\put(20,0.5){\line(1,0){2.5}}
\put(20,0.5){\line(0,1){2.5}}
\put(22.5,0.5){\line(0,1){2.5}}
\put(20,3){\line(1,0){2.5}}

\put(21.9,0.5){\line(0,1){3.3}}
\put(21.9,0.5){\line(1,0){3.3}}
\put(25.2,3.8){\line(-1,0){3.3}}
\put(25.2,3.8){\line(0,-1){3.3}}

\put(26.7,0.5){\line(0,1){3.3}}
\put(26.7,0.5){\line(1,0){3.3}}
\put(26.7,3.8){\line(1,0){3.3}}
\put(30,0.5){\line(0,1){3.3}}

\put(20,7.2){\line(0,1){3.3}}
\put(20,7.2){\line(1,0){3.3}}
\put(23.3,10.5){\line(0,-1){3.3}}
\put(23.3,10.5){\line(-1,0){3.3}}

\multiput(25.2,0.5)(0.25,0){6}{\line(1,0){0.125}}
\multiput(23.3,10.5)(0.25,0){27}{\line(1,0){0.125}}
\multiput(30,3.85)(0,0.25){27}{\line(0,1){0.125}}
\multiput(20,3)(0,0.25){19}{\line(0,1){0.125}}

\put(22,3.8){\vector(0,1){1}}
\put(21,5.2){$B_{1,4}$}
\put(30,2.1){\vector(1,0){1}}
\put(31,1.7){$B_{1,5}$}
\put(20,8.8){\vector(-1,0){1}}
\put(16.5,8.5){$B_{1,6}$}

\end{picture}

\caption{The first iteration of the GIFS defined in Example \ref{E:exam_GIFS_not_str_con_R2}.}
\label{F:fig_not_str_con_R2_first_iteration}
\end{figure}
\end{center}

\begin{center}
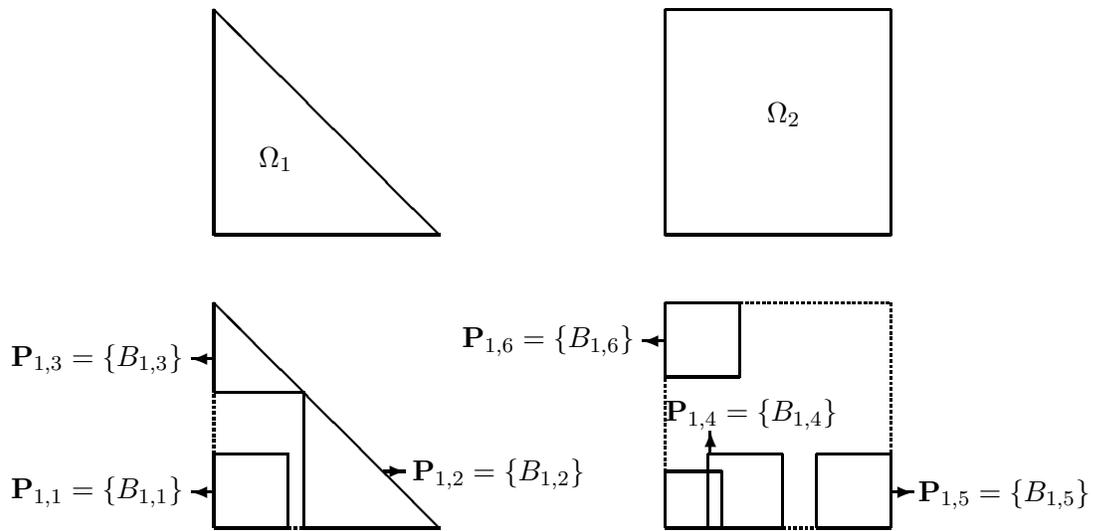
\begin{figure}[h]

\begin{picture}(280,220)
\unitlength=0.30cm

\thicklines

\put(0,13.5){\line(1,0){10}}
\put(0,13.5){\line(0,1){10}}
\put(10,13.5){\line(-1,1){10}}
\put(2,16.5){$\Omega_1$}

\put(20,13.5){\line(1,0){10}}
\put(20,13.5){\line(0,1){10}}
\put(30,13.5){\line(0,1){10}}
\put(20,23.5){\line(1,0){10}}
\put(24.5,18.5){$\Omega_2$}


\put(0,6.5){\line(0,1){4}}
\put(4,0.5){\line(1,0){6}}
\put(4,0.5){\line(0,1){6}}
\put(10,0.5){\line(-1,1){10}}

\put(0,6.5){\line(1,0){4}}
\put(0,0.5){\line(0,1){3.3}}
\put(0,0.5){\line(1,0){3.3}}
\put(3.3,0.5){\line(0,1){3.3}}
\put(0,3.8){\line(1,0){3.3}}
\multiput(3.3,0.5)(0.25,0){3}{\line(1,0){0.125}}
\multiput(0,3.8)(0,0.25){12}{\line(0,1){0.125}}

\put(0,2.1){\vector(-1,0){1}}
\put(-9,1.8){${\bf P}_{1,1}=\{B_{1,1}\}$}
\put(7.5,3){\vector(1,0){1}}
\put(8.8,2.6){${\bf P}_{1,2}=\{B_{1,2}\}$}
\put(0,8){\vector(-1,0){1}}
\put(-9,7.7){${\bf P}_{1,3}=\{B_{1,3}\}$}

\put(20,0.5){\line(1,0){2.5}}
\put(20,0.5){\line(0,1){2.5}}
\put(22.5,0.5){\line(0,1){2.5}}
\put(20,3){\line(1,0){2.5}}

\put(21.9,0.5){\line(0,1){3.3}}
\put(21.9,0.5){\line(1,0){3.3}}
\put(25.2,3.8){\line(-1,0){3.3}}
\put(25.2,3.8){\line(0,-1){3.3}}

\put(26.7,0.5){\line(0,1){3.3}}
\put(26.7,0.5){\line(1,0){3.3}}
\put(26.7,3.8){\line(1,0){3.3}}
\put(30,0.5){\line(0,1){3.3}}

\put(20,7.2){\line(0,1){3.3}}
\put(20,7.2){\line(1,0){3.3}}
\put(23.3,10.5){\line(0,-1){3.3}}
\put(23.3,10.5){\line(-1,0){3.3}}

\multiput(25.2,0.5)(0.25,0){6}{\line(1,0){0.125}}
\multiput(23.3,10.5)(0.25,0){27}{\line(1,0){0.125}}
\multiput(30,3.85)(0,0.25){27}{\line(0,1){0.125}}
\multiput(20,3)(0,0.25){19}{\line(0,1){0.125}}

\put(22,3.8){\vector(0,1){1}}
\put(20,5.25){${\bf P}_{1,4}=\{B_{1,4}\}$}
\put(30,2.1){\vector(1,0){1}}
\put(31.2,1.7){${\bf P}_{1,5}=\{B_{1,5}\}$}
\put(20,8.8){\vector(-1,0){1}}
\put(11,8.5){${\bf P}_{1,6}=\{B_{1,6}\}$}

\end{picture}

\caption{$\mu$-partitions ${\bf P}_{k,\ell}$ of $B_{1,\ell}$ for the GIFS defined in Example \ref{E:exam_GIFS_not_str_con_R2}.}
\label{F:fig1_not_str_con_R2_mu_partition}
\end{figure}
\end{center}

\begin{center}
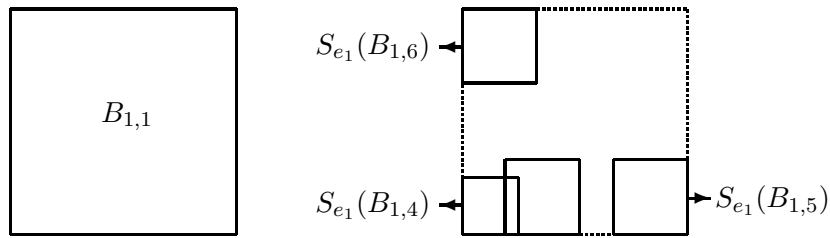
\begin{figure}[h]

\begin{picture}(280,100)
\unitlength=0.30cm

\thicklines

\put(0,0.5){\line(1,0){10}}
\put(0,0.5){\line(0,1){10}}
\put(10,0.5){\line(0,1){10}}
\put(0,10.5){\line(1,0){10}}
\put(4,5.5){$B_{1,1}$}

\put(20,0.5){\line(1,0){2.5}}
\put(20,0.5){\line(0,1){2.5}}
\put(22.5,0.5){\line(0,1){2.5}}
\put(20,3){\line(1,0){2.5}}

\put(21.9,0.5){\line(0,1){3.3}}
\put(21.9,0.5){\line(1,0){3.3}}
\put(25.2,3.8){\line(-1,0){3.3}}
\put(25.2,3.8){\line(0,-1){3.3}}

\put(26.7,0.5){\line(0,1){3.3}}
\put(26.7,0.5){\line(1,0){3.3}}
\put(26.7,3.8){\line(1,0){3.3}}
\put(30,0.5){\line(0,1){3.3}}

\put(20,7.2){\line(0,1){3.3}}
\put(20,7.2){\line(1,0){3.3}}
\put(23.3,10.5){\line(0,-1){3.3}}
\put(23.3,10.5){\line(-1,0){3.3}}

\multiput(25.2,0.5)(0.25,0){6}{\line(1,0){0.125}}
\multiput(23.3,10.5)(0.25,0){27}{\line(1,0){0.125}}
\multiput(30,3.85)(0,0.25){27}{\line(0,1){0.125}}
\multiput(20,3)(0,0.25){19}{\line(0,1){0.125}}

\put(20,1.8){\vector(-1,0){1}}
\put(13.5,1.5){$S_{e_1}(B_{1,4})$}
\put(30,2.1){\vector(1,0){1}}
\put(31.2,1.7){$S_{e_1}(B_{1,5})$}
\put(20,8.8){\vector(-1,0){1}}
\put(13.5,8.5){$S_{e_1}(B_{1,6})$}

\end{picture}

\caption{$\mu$-partitions ${\bf P}_{2,1}$.}
\label{F:fig1_not_str_con_R2_mu_partition_P_21}
\end{figure}
\end{center}

\begin{center}
\begin{figure}[h]

\begin{picture}(280,100)
\unitlength=0.30cm

\thicklines

\put(0,0.5){\line(1,0){10}}
\put(0,0.5){\line(0,1){10}}
\put(10,0.5){\line(-1,1){10}}
\put(3,3.5){$B_{1,2}$}

\put(20,6.5){\line(0,1){4}}
\put(24,0.5){\line(1,0){6}}
\put(24,0.5){\line(0,1){6}}
\put(30,0.5){\line(-1,1){10}}

\put(20,6.5){\line(1,0){4}}
\put(20,0.5){\line(0,1){3.3}}
\put(20,0.5){\line(1,0){3.3}}
\put(23.3,0.5){\line(0,1){3.3}}
\put(20,3.8){\line(1,0){3.3}}
\multiput(23.3,0.5)(0.25,0){3}{\line(1,0){0.125}}
\multiput(20,3.8)(0,0.25){12}{\line(0,1){0.125}}

\put(20,2.1){\vector(-1,0){1}}
\put(13.5,1.8){$S_{e_2}(B_{1,1})$}
\put(27.5,3){\vector(1,0){1}}
\put(28.8,2.6){$S_{e_2}(B_{1,2})$}
\put(20,8){\vector(-1,0){1}}
\put(13.5,7.7){$S_{e_2}(B_{1,3})$}

\end{picture}

\caption{$\mu$-partitions ${\bf P}_{2,2}$.}
\label{F:fig1_not_str_con_R2_mu_partition_P_22}
\end{figure}
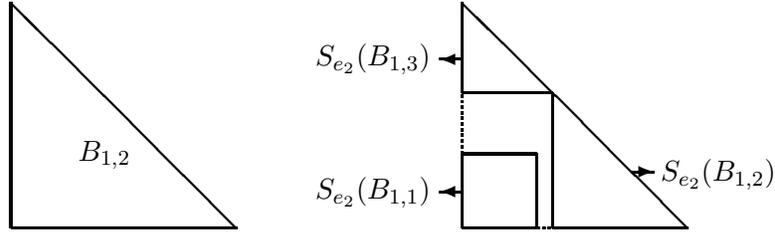
\end{center}

\begin{center}
\begin{figure}[h]
\begin{picture}(280,100)
\unitlength=0.12cm
\thicklines

\put(-5,5){\line(1,0){15}}
\put(-5,5){\line(0,1){15}}
\put(-5,20){\line(1,0){15}}
\put(10,5){\line(0,1){15}}

\put(10,5){\line(1,0){16.25}}
\put(5.25,26){\line(1,0){21}}
\put(5.25,26){\line(0,-1){21}}
\put(26.25,26){\line(0,-1){21}}
\put(-5,12.5){\vector(-1,0){2.5}}
\put(-13,11){$B_{1,4}$}

\put(49.5,4.5){\dashbox(8.6875,6.25){}}
\put(49.5,8){\vector(-1,0){2.5}}
\put(33.5,7){$S_{e_4}(B_{1,4})$}

\put(50,8.75){\line(1,0){3.75}}
\put(53.75,5){\line(0,1){3.75}}
\put(50,5){\line(1,0){3.75}}
\put(50,5){\line(0,1){3.75}}

\put(52.4375,5){\line(0,1){5.25}}
\put(52.4375,5){\line(1,0){5.25}}
\put(57.6875,5){\line(0,1){5.25}}
\put(52.4375,10.25){\line(1,0){5.25}}

\put(49.5,14.25){\dashbox(6.25,6.25){}}
\put(49.5,17.5){\vector(-1,0){2.5}}
\put(33.5,17){$S_{e_4}(B_{1,6})$}

\put(50,14.75){\line(1,0){5.25}}
\put(50,14.75){\line(0,1){5.25}}
\put(50,20){\line(1,0){5.25}}
\put(55.25,14.75){\line(0,1){5.25}}

\put(59.25,4.5){\dashbox(11.7625,8.35){}}
\put(64.5,5){\vector(0,-1){3.5}}
\put(62,-1){$S_{e_5}(B_{1,4})$}

\put(59.75,5){\line(0,1){5.25}}
\put(59.75,5){\line(1,0){5.25}}
\put(59.75,10.25){\line(1,0){5.25}}
\put(65,5){\line(0,1){5.25}}

\put(63.1625,5){\line(0,1){7.35}}
\put(63.1625,5){\line(1,0){7.35}}
\put(63.1625,12.35){\line(1,0){7.35}}
\put(70.5125,5){\line(0,1){7.35}}

\put(72.9,4.5){\dashbox(8.35,8.35){}}
\put(81.25,8){\vector(1,0){2.5}}
\put(84,7){$S_{e_5}(B_{1,5})$}

\put(73.4,5){\line(0,1){7.35}}
\put(73.4,5){\line(1,0){7.35}}
\put(73.4,12.35){\line(1,0){7.35}}
\put(80.75,5){\line(0,1){7.35}}

\put(59.25,18.15){\dashbox(8.35,8.35){}}
\put(59.25,23){\vector(-1,0){2.5}}
\put(44,22){$S_{e_5}(B_{1,6})$}

\put(59.75,18.65){\line(0,1){7.35}}
\put(59.75,18.65){\line(1,0){7.35}}
\put(67.1,26){\line(-1,0){7.35}}
\put(67.1,26){\line(0,-1){7.35}}

\multiput(50,8.75)(0,0.5){12}{\line(0,1){0.25}}
\multiput(57.6875,5)(0.5,0){5}{\line(1,0){0.25}}
\multiput(55.25,20)(0.5,0){9}{\line(1,0){0.25}}
\multiput(70.5125,5)(0.5,0){6}{\line(1,0){0.25}}
\multiput(80.75,12.35)(0,0.5){28}{\line(0,1){0.25}}
\multiput(67.1,26)(0.5,0){28}{\line(1,0){0.25}}

\end{picture}
\caption{$\mu$-partition ${\bf P}_{2,4}$.}
\label{F:fig1_not_str_con_R2_mu_partition_P_24}
\end{figure}
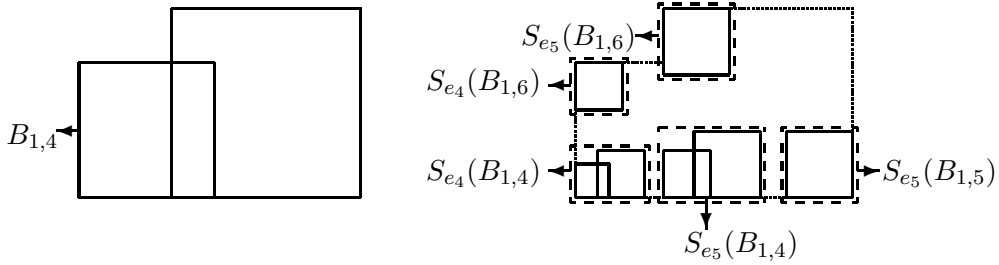
\end{center}

Using Lemma \ref{L:GIFS_not_strong_con_R2_2}, we obtain
\begin{eqnarray*}
\begin{aligned}
&{\bf P}'_{2,\ell}=\{B_{2,\ell,j}: j=1,2,3\},\qquad {\bf P}^\ast_{2,\ell}=\emptyset\quad\mbox{ for }\ell=1,2,3,5,6,\\
&{\bf P}'_{2,4}=\{B_{2,4,j}: j=1,2,4,5\},\qquad {\bf P}^\ast_{k,4}=\{B_{k,4,3}\}.
\end{aligned}
\end{eqnarray*}
For $k\ge 3$, define
\begin{eqnarray*}
{\bf P}_{k,4}:={\bf P}'_{k-1,4}\cup\{B_{k,4,j}: j=1,2,3,4,5\}.
\end{eqnarray*}
Using Lemma \ref{L:GIFS_not_strong_con_R2_2}(d), we get
${\bf P}'_{k,4}=\{B_{k,4,j}: j=1,2,4,5\}$, and ${\bf P}_{k,4}^*=\{B_{k,4,3}\}$.
Hence $\Gamma'_1=\{1,2,3\}$, $\Gamma^\ast_1=\emptyset$, $\Gamma'_2=\{5,6\}$, $\Gamma^\ast_2=\{4\}$, $\kappa_4=\infty$ and $\kappa_\ell=2$ for $\ell=1,2,4,5$. Also, $G'_{2,\ell}=\{1,2,3\}$, $G^\ast_{2,\ell}=\emptyset$ for $\ell=1,2,4,5$,
and $G'_{k,3}=\{1,2,4,5\}$, $G^\ast_{k,3}=\{3\}$ for $k\ge 2$.
\begin{lem}
Let $w(k,\ell,j), c(k,\ell,j), e(k,\ell,j)$ and $\rho_{e(k,\ell,j)}$ be defined as in \eqref{eq:relationship_between_mui_and_B1c}. Then
\begin{enumerate}
\item[(a)] for $j=1,2,3$,
\begin{eqnarray*}
\begin{aligned}
&w(2,1,j)=p_{e_1},\ c(2,1,j)=j+3,\ \rho_{e(2,1,j)}=s^2;\\
&w(2,\ell,j)=p_{e_j},\ c(2,\ell,j)=\ell,\ \rho_{e(2,\ell,1)}=s^2, \ \rho_{e(2,\ell,2)}=t^2,\ \rho_{e(2,\ell,3)}=(1-t)^2
\mbox{ for }\ell=2,3;\\
&w(2,\ell,j)=p_{e_{\ell+1}},\ c(2,\ell,j)=j+3,\ \rho_{e(2,\ell,j)}=r^2\quad\mbox{ for }\ell=5,6;\\
\end{aligned}
\end{eqnarray*}
\item[(b)] for $k\ge2$,
\begin{eqnarray*}
\begin{aligned}
&w(k,4,1)=w_{k-2},\quad\ c(k,4,1)=4,\quad\ \rho_{e(k,4,1)}=(\rho r^{k-2})^2;\\
&w(k,4,2)=w_{k-2},\quad\ c(k,4,2)=6,\quad\ \rho_{e(k,4,2)}=(\rho r^{k-2})^2;\\
&w(k,4,j)=p_{e^{k-1}_5},\quad\ c(k,4,j)=j+1,\quad\ \rho_{e(k,4,j)}=r^{2(k-1)}\quad\mbox{ for }j=4,5.
\end{aligned}
\end{eqnarray*}
\end{enumerate}
\end{lem}

Based on the above results, we can express the vector-valued equations \eqref{eq:f_ell_'_1} and \eqref{eq:f_ell_ast_1} as follows
\begin{eqnarray*}
\begin{aligned}
f_1(x)=&p^q_{e_1} s^{-2\alpha_1}\sum_{j=4}^6 f_j(x+\ln s)+ z_1^{(\alpha_1)}(x),\\
f_2(x)=&p^q_{e_2} t^{-2\alpha_1}\sum_{j=1}^3 f_j(x+\ln t)+ z_2^{(\alpha_1)}(x),\\
f_3(x)=&p^q_{e_3} (1-t)^{-2\alpha_1}\sum_{j=1}^3 f_j(x+\ln (1-t))+ z_3^{(\alpha_1)}(x),\\
f_4(x)=&\sum_{k=2}^\infty \Big(w_{k-2}^q (\rho r^{k-2})^{-2\alpha_2}\sum_{j=4,6}f_j(x+\ln (\rho r^{k-2}))
+(p^q_{e_5} r^{-2\alpha_2})^{k-1}\sum_{j=5}^6f_j(x+\ln (r^{k-1}))\Big)\\
&+z_4^{(\alpha_2)}(x)-z_{4,\infty}^{(\alpha_2)}(x),\\
f_\ell(x)=&p^q_{e_{\ell+1}} r^{-2\alpha_2}\sum_{j=4}^6 f_j(x+\ln r)+ z_\ell^{(\alpha_2)}(x)\quad\mbox{ for }\ell=5,6.
\end{aligned}
\end{eqnarray*}
with $z_{4,\infty}^{(\alpha_2)}(x)=E_{4,\infty}^{(\alpha_2)}(e^{-x})$ and $z_\ell^{(\alpha_i)}(x)=E_\ell^{(\alpha_i)}(e^{-x})$ for $\ell\in\Gamma_i$, $i\in\mathcal{SC}_m$ and $m=1,2$.

For $m=1,2$, $i\in\mathcal{SC}_m$ and $\ell,\ell'\in\Gamma_i$, let $\mu^{(\alpha_m)}_{\ell\ell'}$ be the discrete measures defined as in \eqref{defi:muij_on_kuangjia}. Then
\begin{eqnarray*}
\begin{aligned}
&\mu_{1\ell'}^{(\alpha_1)}(-\ln s)=p^q_{e_1} s^{-2\alpha_1}\quad\mbox{ for }\ell'=4,5,6,\\
&\mu_{2\ell'}^{(\alpha_1)}(-\ln t)=p^q_{e_2} t^{-2\alpha_1}\quad\mbox{ for }\ell'=1,2,3,\\
&\mu_{3\ell'}^{(\alpha_1)}(-\ln (1-t))=p^q_{e_3} (1-t)^{-2\alpha_1}\quad\mbox{ for }\ell'=1,2,3,\\
&\mu_{4\ell'}^{(\alpha_2)}(-\ln (\rho r^{k-2}))=w^q_{k-2}(\rho r^{k-2})^{-2\alpha_2}\quad\mbox{ for }\ell'=4,6\mbox{ and }k\ge2,\\
&\mu_{4\ell'}^{(\alpha_2)}(-\ln (r^{k-1}))=(p_{e_5}^{q}r^{-2\alpha_2})^{k-1}\quad\mbox{ for }\ell'=5,6\mbox{ and }k\ge2,\\
&\mu_{\ell \ell'}^{(\alpha_2)}(-\ln r)=p_{e_{\ell+1}}^q r^{-2\alpha_2}\quad\mbox{ for }\ell=5,6\mbox{ and }\ell'=4,5,6.
\end{aligned}
\end{eqnarray*}
Moreover,
\begin{eqnarray*}
{\bf M}(\alpha; \infty)=
\left(
  \begin{array}{cccccc}
    0 & 0 & 0 & a_1 & a_1 & a_1 \\
    a_2 & a_2 & a_2 & 0 & 0 & 0 \\
    a_3 & a_3 & a_3 & 0 & 0 & 0 \\
    0 & 0 & 0 & a_4 & a_5 & a_4+a_5\\
    0 & 0 & 0 & a_6 & a_6 & a_6 \\
    0 & 0 & 0 & a_7 & a_7 & a_7 \\
  \end{array}
\right),
\end{eqnarray*}
where
\begin{eqnarray*}
\begin{aligned}
&a_1=p^q_{e_1} s^{-2\alpha_1},\quad &&a_2=p^q_{e_2} t^{-2\alpha_1},\quad &&a_3=p^q_{e_3} (1-t)^{-2\alpha_1},\\
&a_4=\sum_{k=0}^\infty w^q_k(\rho r^k)^{-2\alpha_2},\quad&& a_5=\sum_{k=1}^\infty (p_{e_5}^q r^{-2\alpha_2})^k,
\quad &&a_\ell=p^q_{e_\ell} r^{-2\alpha_2}\quad \mbox{ for }\ell=6,7.
\end{aligned}
\end{eqnarray*}

In the following, we firstly show that $D_4$ is open, and then prove that the error terms $z_{4,\infty}^{(\alpha_2)}(x)=o(e^{-\epsilon x})$ and $z_\ell^{(\alpha_i)}(x)=o(e^{-\epsilon x})$ as $x\to\infty$, it is equivalent to that $E_{4,\infty}^{(\alpha_2)}(h)=o(h^\epsilon)$ and $E_\ell^{(\alpha_i)}(h)=o(h^\epsilon)$ as $h\to\infty$ for $\ell\in\Gamma_i$, $i\in\mathcal{SC}_m$ and $m=1,2$. The proofs are the same as that of Propositions \ref{P:str_con_R_D1_open}-\ref{P:str_con_R2_error_esti}, and we omit the details.

\begin{lem}\label{P:not_str_con_R2_Phi_46_56_le1}
For $\ell=4,6$, $\ell'=5,6$ and any $k\ge N+1$,
$$\Phi_\ell^{(\alpha_2)}(h/{\rho r^{k-2}})\le 1 \quad\mbox{ and }\quad
\Phi_{\ell'}^{(\alpha_2)}(h/{r^{k-1}})\le 1.$$
\end{lem}
\begin{lem}\label{P:not_str_con_R2_D4_open}
For $q\ge0$, $D_4$ is open.
\end{lem}

\begin{prop}\label{P:not_str_con_R2_est}
For $q\ge 0$, assume that $\alpha_i\in D_\ell$ for $\ell\in\Gamma_i$, $i\in\mathcal{SC}_m$ and $m=1,2$,
there exists some $\epsilon>0$ such that
\begin{enumerate}
\item[(a)] $\sum\limits_{k=N+1}^\infty \Big(w_{k-2}^q (\rho r^{k-2})^{-2\alpha_2}
\sum\limits_{j=4,6}\Phi_j^{(\alpha_2)}(h/{\rho r^{k-2}})
+(p_{e_5}^q r^{-2\alpha_2})^{k-1}\sum\limits_{j=5}^6\Phi_j^{(\alpha_2)}(h/{r^{k-1}})\Big)=o(h^{\epsilon/2})$;
\item[(b)]
$h^{-(2+\alpha_2)}\Big(\sum\limits_{k=2}^N e_{k,4}+\int_{B_{N,4,3}}\mu(B_h({\bf x}))^q\,d{\bf x}\Big)=o(h^{\epsilon/2})$;
\item[(c)] $h^{-(2+\alpha_1)} e_{2,\ell}=o(h^{\epsilon/2})$ for $\ell=1,2,3$ and $h^{-(2+\alpha_2)} e_{2,\ell}=o(h^{\epsilon/2})$ for $\ell=5,6$.
\end{enumerate}
\end{prop}
\begin{proof}
(a) The proof is similar to that of Proposition \ref{P:str_con_R2_error_esti}(a).

(b) In order to complete the proof of (b), we need to give the following two facts:
\begin{eqnarray}\label{eq:not_str_con_R2_facts}
w^q_{N-1}\le C h^{\alpha_2+\epsilon}\qquad\mbox{ and }\qquad p^{Nq}_{e_5}\le C h^{\alpha_2+\epsilon}.
\end{eqnarray}

The proof of (b) consists of the following two parts:

\noindent\textit{Part 1. The proof of $\sum_{k=2}^N e_{k,4}=o(h^{2+\alpha_2+\epsilon/2})$.}

It is equivalent to show that for any $2\le k\le N$,
\begin{eqnarray}\label{e:want_to_prove_5}
\begin{aligned}
&\int_{\widehat{B}_{k,4,j}(h)}\mu(B_h({\bf x}))^q\,d{\bf x}=o(h^{2+\alpha_2+\epsilon/2})
\quad\mbox{ for }j=1,2,4,5,\\
&w^q_{k-2}(\rho r^{k-2})^2\int_{\widehat{B}_{1,j}(h/{\rho r^{k-2}})}\mu(B_{h/{\rho r^{k-2}}}({\bf x}))^q\,d{\bf x}=o(h^{2+\alpha_2+\epsilon/2})\quad\mbox{ for }j=4,6,\\
&(p^q_{e_5}r^2)^{k-1}\int_{\widehat{B}_{1,j}(h/{r^{k-1}})}\mu(B_{h/{r^{k-1}}}({\bf x}))^q\,d{\bf x}=o(h^{2+\alpha_2+\epsilon/2})
\quad\mbox{ for }j=5,6.
\end{aligned}
\end{eqnarray}

\begin{center}
\begin{figure}[h]

\begin{picture}(280,150)
\unitlength=1.2cm

\thicklines

\put(2,0.5){\line(1,0){5.2}}
\put(2,0.5){\line(0,1){2.86}}
\put(2,3.36){\line(1,0){1.9}}

\put(3.9,3.8){\line(0,-1){0.44}}
\put(7.2,3.8){\line(-1,0){3.3}}
\put(7.2,3.8){\line(0,-1){3.3}}

\put(2.35,0.85){\line(1,0){4.5}}
\put(2.35,0.85){\line(0,1){2.16}}
\put(2.35,3.01){\line(1,0){1.9}}
\put(4.25,3.01){\line(0,1){0.44}}
\put(4.25,3.45){\line(1,0){2.6}}
\put(6.85,3.45){\line(0,-1){2.6}}

{\color{red}
\put(2.35,0.5){\line(0,1){0.35}}
\put(6.85,0.85){\line(1,0){0.35}}
\put(6.85,3.45){\line(0,1){0.35}}
\put(4.25,3.45){\line(-1,0){0.35}}
\put(3.9,3.36){\line(0,-1){0.35}}
\put(2.35,3.01){\line(-1,0){0.35}}}

\put(2.35,0.68){\vector(-1,0){0.6}}
\put(1.5,0.62){$h$}

\multiput(2.05,0.62)(0.25,0){21}{\line(1,0){0.125}}
\multiput(2.05,0.75)(0.25,0){21}{\line(1,0){0.125}}

\multiput(2.1,0.89)(0.25,0){1}{\line(1,0){0.125}}
\multiput(2.1,1.01)(0.25,0){1}{\line(1,0){0.125}}
\multiput(2.1,1.13)(0.25,0){1}{\line(1,0){0.125}}
\multiput(2.1,1.25)(0.25,0){1}{\line(1,0){0.125}}
\multiput(2.1,1.37)(0.25,0){1}{\line(1,0){0.125}}
\multiput(2.1,1.49)(0.25,0){1}{\line(1,0){0.125}}
\multiput(2.1,1.61)(0.25,0){1}{\line(1,0){0.125}}
\multiput(2.1,1.73)(0.25,0){1}{\line(1,0){0.125}}
\multiput(2.1,1.85)(0.25,0){1}{\line(1,0){0.125}}
\multiput(2.1,1.97)(0.25,0){1}{\line(1,0){0.125}}
\multiput(2.1,2.09)(0.25,0){1}{\line(1,0){0.125}}
\multiput(2.1,2.21)(0.25,0){1}{\line(1,0){0.125}}
\multiput(2.1,2.33)(0.25,0){1}{\line(1,0){0.125}}
\multiput(2.1,2.45)(0.25,0){1}{\line(1,0){0.125}}
\multiput(2.1,2.57)(0.25,0){1}{\line(1,0){0.125}}
\multiput(2.1,2.69)(0.25,0){1}{\line(1,0){0.125}}
\multiput(2.1,2.81)(0.25,0){1}{\line(1,0){0.125}}
\multiput(2.1,2.93)(0.25,0){1}{\line(1,0){0.125}}

\multiput(2.1,3.24)(0.25,0){9}{\line(1,0){0.125}}
\multiput(2.1,3.12)(0.25,0){9}{\line(1,0){0.125}}

\multiput(4.0,3.37)(0.25,0){1}{\line(1,0){0.125}}

\multiput(4.0,3.67)(0.25,0){13}{\line(1,0){0.125}}
\multiput(4.0,3.55)(0.25,0){13}{\line(1,0){0.125}}

\multiput(2.1,2.21)(0.25,0){1}{\line(1,0){0.125}}

\multiput(6.95,0.97)(0.25,0){1}{\line(1,0){0.125}}
\multiput(6.95,1.09)(0.25,0){1}{\line(1,0){0.125}}
\multiput(6.95,1.21)(0.25,0){1}{\line(1,0){0.125}}
\multiput(6.95,1.33)(0.25,0){1}{\line(1,0){0.125}}
\multiput(6.95,1.45)(0.25,0){1}{\line(1,0){0.125}}
\multiput(6.95,1.57)(0.25,0){1}{\line(1,0){0.125}}
\multiput(6.95,1.69)(0.25,0){1}{\line(1,0){0.125}}
\multiput(6.95,1.81)(0.25,0){1}{\line(1,0){0.125}}
\multiput(6.95,1.93)(0.25,0){1}{\line(1,0){0.125}}
\multiput(6.95,2.05)(0.25,0){1}{\line(1,0){0.125}}
\multiput(6.95,2.17)(0.25,0){1}{\line(1,0){0.125}}
\multiput(6.95,2.29)(0.25,0){1}{\line(1,0){0.125}}
\multiput(6.95,2.41)(0.25,0){1}{\line(1,0){0.125}}
\multiput(6.95,2.53)(0.25,0){1}{\line(1,0){0.125}}
\multiput(6.95,2.65)(0.25,0){1}{\line(1,0){0.125}}
\multiput(6.95,2.77)(0.25,0){1}{\line(1,0){0.125}}
\multiput(6.95,2.89)(0.25,0){1}{\line(1,0){0.125}}
\multiput(6.95,3.01)(0.25,0){1}{\line(1,0){0.125}}
\multiput(6.95,3.13)(0.25,0){1}{\line(1,0){0.125}}
\multiput(6.95,3.25)(0.25,0){1}{\line(1,0){0.125}}
\multiput(6.95,3.40)(0.25,0){1}{\line(1,0){0.125}}

\put(4.6,0.5){\vector(0,-1){0.3}}
\put(4,0){$\widehat{U}_{k,4,1,1}(h)$}

\put(7.2,2.3){\vector(1,0){0.3}}
\put(7.55,2.2){$\widehat{U}_{k,4,1,2}(h)$}

\put(5.6,3.8){\vector(0,1){0.3}}
\put(5,4.25){$\widehat{U}_{k,4,1,3}(h)$}

\put(4.07,3.01){\vector(0,-1){0.3}}
\put(3.5,2.5){$\widehat{U}_{k,4,1,4}(h)$}

\put(3.0,3.36){\vector(0,1){0.3}}
\put(2.25,3.75){$\widehat{U}_{k,4,1,5}(h)$}

\put(2,1.76){\vector(-1,0){0.3}}
\put(0.3,1.65){$\widehat{U}_{k,4,1,6}(h)$}

\end{picture}

\caption{The middle part and shaded region are $\widetilde{B}_{k,4,1}(h)$ and $\widehat{B}_{k,4,1}(h)=\bigcup_{j=1}^6\widehat{U}_{k,4,1,j}(h)$.}
\label{F:fig1_not_str_con_R2_1}
\end{figure}
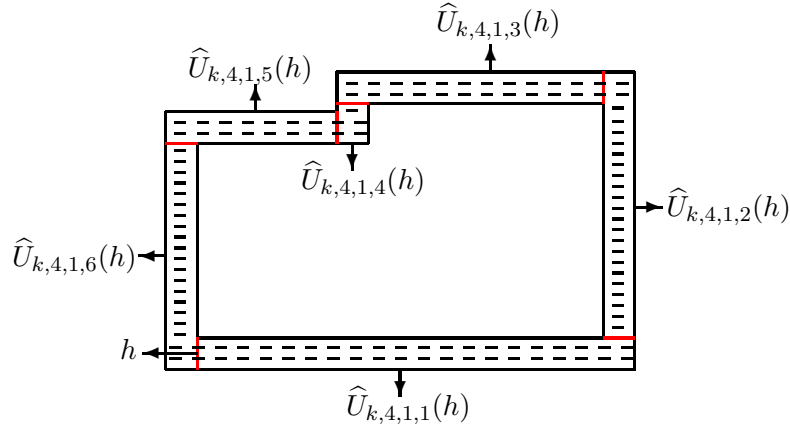
\end{center}

The proofs of \eqref{e:want_to_prove_5} is similar, we only prove $\int_{\widehat{B}_{k,4,1}(h)}\mu(B_h({\bf x}))^q\,d{\bf x}=o(h^{2+\alpha_2+\epsilon/2})$ as an example. Using the first fact of \ref{eq:not_str_con_R2_facts}, we get (see Figure \ref{F:fig1_not_str_con_R2_1}) for any $2\le k\le N$,
\begin{eqnarray*}
\begin{aligned}
&\int_{\widehat{B}_{k,4,1}(h)}\mu(B_h({\bf x}))^q\,d{\bf x}\\
\le&\int_{\widehat{U}_{k,4,1,1}(h)}\mu(B_{(\rho+r-\rho r)\rho r^{k-2}+h}(S_{e_5^{k-2}e_4e_5}(3,0)))^q\,d{\bf x}
+\int_{\widehat{U}_{k,4,1,2}(h)}\mu(B_{\rho r^{k-1}+h}(S_{e_5^{k-2}e_4e_5}(3,1)))^q\,d{\bf x}\\
&+\int_{\widehat{U}_{k,4,1,3}(h)}\mu(B_{\rho r^{k-1}+h}(S_{e_5^{k-2}e_4e_5}(2,1)))^q\,d{\bf x}\\
&+\int_{\widehat{U}_{k,4,1,4}(h)}\mu(B_{(r-\rho)\rho r^{k-2}+2h}(S_{e_5^{k-2}e_4e_5}(2,1)+(0,-h)))^q\,d{\bf x}\\
&+\int_{\widehat{U}_{k,4,1,5}(h)}\mu(B_{\rho^2(1-r)r^{k-2}+2h}(S_{e_5^{k-2}e^2_4}(2,1)))^q\,d{\bf x}
+\int_{\widehat{U}_{k,4,1,6}(h)}\mu(B_{\rho^2 r^{k-2}+h}(S_{e_5^{k-2}e_4^2}(2,0)))^q\,d{\bf x}\\
\le&p_{e_5}^{(k-1)q}p^q_{e_4}\sum_{j=1}^4\int_{\widehat{U}_{k,4,1,j}(h)}1\,d{\bf x}
+p_{e_5}^{(k-2)q}p^{2q}_{e_4}\sum_{j=5}^6\int_{\widehat{U}_{k,4,1,j}(h)}1\,d{\bf x}\\
\le&C\big(p^q_{e_5}(\rho+4r)+2p^q_{e_4}\rho\big)p_{e_5}^{-Nq}p^q_{e_4}r^{1-N}h^{2+\alpha_2+\epsilon},
\end{aligned}
\end{eqnarray*}
and so
\begin{eqnarray*}
h^{-(2+\alpha_2+\epsilon/2)}\int_{\widehat{B}_{k,4,1}(h)}\mu(B_h({\bf x}))^q\,d{\bf x}
\le C\big(p^q_{e_5}(\rho+4r)+2p^q_{e_4}\rho\big)p_{e_5}^{-Nq}p^q_{e_4}r^{1-N}h^{\epsilon/2}.
\end{eqnarray*}

\noindent\textit{Part 2. The proof of $\int_{B_{N,4,3}}\mu(B_h({\bf x}))^q\,dx=o(h^{2+\alpha_2+\epsilon/2})$.}

\begin{center}
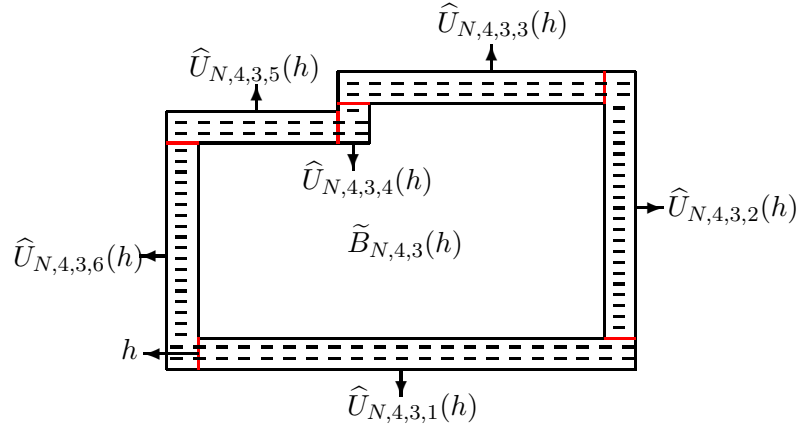
\begin{figure}[h]

\begin{picture}(280,150)
\unitlength=1.2cm

\thicklines

\put(2,0.5){\line(1,0){5.2}}
\put(2,0.5){\line(0,1){2.86}}
\put(2,3.36){\line(1,0){1.9}}

\put(3.9,3.8){\line(0,-1){0.44}}
\put(7.2,3.8){\line(-1,0){3.3}}
\put(7.2,3.8){\line(0,-1){3.3}}

\put(2.35,0.85){\line(1,0){4.5}}
\put(2.35,0.85){\line(0,1){2.16}}
\put(2.35,3.01){\line(1,0){1.9}}
\put(4.25,3.01){\line(0,1){0.44}}
\put(4.25,3.45){\line(1,0){2.6}}
\put(6.85,3.45){\line(0,-1){2.6}}

{\color{red}
\put(2.35,0.5){\line(0,1){0.35}}
\put(6.85,0.85){\line(1,0){0.35}}
\put(6.85,3.45){\line(0,1){0.35}}
\put(4.25,3.45){\line(-1,0){0.35}}
\put(3.9,3.36){\line(0,-1){0.35}}
\put(2.35,3.01){\line(-1,0){0.35}}}

\put(2.35,0.68){\vector(-1,0){0.6}}
\put(1.5,0.62){$h$}

\multiput(2.05,0.62)(0.25,0){21}{\line(1,0){0.125}}
\multiput(2.05,0.75)(0.25,0){21}{\line(1,0){0.125}}

\multiput(2.1,0.89)(0.25,0){1}{\line(1,0){0.125}}
\multiput(2.1,1.01)(0.25,0){1}{\line(1,0){0.125}}
\multiput(2.1,1.13)(0.25,0){1}{\line(1,0){0.125}}
\multiput(2.1,1.25)(0.25,0){1}{\line(1,0){0.125}}
\multiput(2.1,1.37)(0.25,0){1}{\line(1,0){0.125}}
\multiput(2.1,1.49)(0.25,0){1}{\line(1,0){0.125}}
\multiput(2.1,1.61)(0.25,0){1}{\line(1,0){0.125}}
\multiput(2.1,1.73)(0.25,0){1}{\line(1,0){0.125}}
\multiput(2.1,1.85)(0.25,0){1}{\line(1,0){0.125}}
\multiput(2.1,1.97)(0.25,0){1}{\line(1,0){0.125}}
\multiput(2.1,2.09)(0.25,0){1}{\line(1,0){0.125}}
\multiput(2.1,2.21)(0.25,0){1}{\line(1,0){0.125}}
\multiput(2.1,2.33)(0.25,0){1}{\line(1,0){0.125}}
\multiput(2.1,2.45)(0.25,0){1}{\line(1,0){0.125}}
\multiput(2.1,2.57)(0.25,0){1}{\line(1,0){0.125}}
\multiput(2.1,2.69)(0.25,0){1}{\line(1,0){0.125}}
\multiput(2.1,2.81)(0.25,0){1}{\line(1,0){0.125}}
\multiput(2.1,2.93)(0.25,0){1}{\line(1,0){0.125}}

\multiput(2.1,3.24)(0.25,0){9}{\line(1,0){0.125}}
\multiput(2.1,3.12)(0.25,0){9}{\line(1,0){0.125}}

\multiput(4.0,3.37)(0.25,0){1}{\line(1,0){0.125}}

\multiput(4.0,3.67)(0.25,0){13}{\line(1,0){0.125}}
\multiput(4.0,3.55)(0.25,0){13}{\line(1,0){0.125}}

\multiput(2.1,2.21)(0.25,0){1}{\line(1,0){0.125}}

\multiput(6.95,0.97)(0.25,0){1}{\line(1,0){0.125}}
\multiput(6.95,1.09)(0.25,0){1}{\line(1,0){0.125}}
\multiput(6.95,1.21)(0.25,0){1}{\line(1,0){0.125}}
\multiput(6.95,1.33)(0.25,0){1}{\line(1,0){0.125}}
\multiput(6.95,1.45)(0.25,0){1}{\line(1,0){0.125}}
\multiput(6.95,1.57)(0.25,0){1}{\line(1,0){0.125}}
\multiput(6.95,1.69)(0.25,0){1}{\line(1,0){0.125}}
\multiput(6.95,1.81)(0.25,0){1}{\line(1,0){0.125}}
\multiput(6.95,1.93)(0.25,0){1}{\line(1,0){0.125}}
\multiput(6.95,2.05)(0.25,0){1}{\line(1,0){0.125}}
\multiput(6.95,2.17)(0.25,0){1}{\line(1,0){0.125}}
\multiput(6.95,2.29)(0.25,0){1}{\line(1,0){0.125}}
\multiput(6.95,2.41)(0.25,0){1}{\line(1,0){0.125}}
\multiput(6.95,2.53)(0.25,0){1}{\line(1,0){0.125}}
\multiput(6.95,2.65)(0.25,0){1}{\line(1,0){0.125}}
\multiput(6.95,2.77)(0.25,0){1}{\line(1,0){0.125}}
\multiput(6.95,2.89)(0.25,0){1}{\line(1,0){0.125}}
\multiput(6.95,3.01)(0.25,0){1}{\line(1,0){0.125}}
\multiput(6.95,3.13)(0.25,0){1}{\line(1,0){0.125}}
\multiput(6.95,3.25)(0.25,0){1}{\line(1,0){0.125}}
\multiput(6.95,3.40)(0.25,0){1}{\line(1,0){0.125}}

\put(4.6,0.5){\vector(0,-1){0.3}}
\put(4,0){$\widehat{U}_{N,4,3,1}(h)$}

\put(7.2,2.3){\vector(1,0){0.3}}
\put(7.55,2.2){$\widehat{U}_{N,4,3,2}(h)$}

\put(5.6,3.8){\vector(0,1){0.3}}
\put(5,4.25){$\widehat{U}_{N,4,3,3}(h)$}

\put(4.07,3.01){\vector(0,-1){0.3}}
\put(3.5,2.5){$\widehat{U}_{N,4,3,4}(h)$}

\put(3.0,3.36){\vector(0,1){0.3}}
\put(2.25,3.75){$\widehat{U}_{N,4,3,5}(h)$}

\put(2,1.76){\vector(-1,0){0.3}}
\put(0.3,1.65){$\widehat{U}_{N,4,3,6}(h)$}

\put(4,1.8){$\widetilde{B}_{N,4,3}(h)$}
\end{picture}

\caption{The middle part and shaded region are $\widetilde{B}_{N,4,3}(h)$ and $\widehat{B}_{N,4,3}(h)=\bigcup_{j=1}^6\widehat{U}_{N,4,3,j}(h)$.}
\label{F:fig1_not_str_con_R2_7}
\end{figure}
\end{center}

From Figure \ref{F:fig1_not_str_con_R2_7}, we get
\begin{eqnarray}\label{e:GIFS_not_strong_con_R2_error_b_part2_1}
\begin{aligned}
\int_{B_{N,4,3}}\mu(B_h({\bf x}))^q\,d{\bf x}
=\sum_{j=1}^6\int_{\widehat{U}_{N,4,3,j}(h)}\mu(B_h({\bf x}))^q\,d{\bf x}
+\int_{\widetilde{B}_{N,4,3}(h)}\mu(B_h({\bf x}))^q\,d{\bf x}.
\end{aligned}
\end{eqnarray}
It follows from Lemma \ref{L:GIFS_not_strong_con_R2_2}(e) that
\begin{eqnarray*}
\mu_2|_{S_{e_5^{N-1}}(B_{1,4})}= w_{N-2}\cdot \mu_2|_{B_{1,5}}\circ S^{-1}_{e_5^{N-1} e_4}
+p_{e_5^{N-1}}\cdot \mu_2|_{B_{1,4}}\circ S^{-1}_{e_5^{N-1}},
\end{eqnarray*}
and so
$\mu(B_h({\bf x}))=\mu_2(B_h({\bf x}))\le w_{N-2}+p_{e_5^{N-1}}$ for any ${\bf x}\in \widetilde{B}_{N,4,3}(h)$.
Hence
\begin{eqnarray*}
\begin{aligned}
\int_{\widetilde{B}_{N,4,3}(h)}\mu(B_h({\bf x})^q\,d{\bf x}
\le&(w_{N-2}+p^{N-1}_{e_5})^q((\rho r^{N-1})^2+r^{2N}).
\end{aligned}
\end{eqnarray*}
Combining this with \ref{e:GIFS_not_strong_con_R2_error_b_part2_1}, we have
\begin{eqnarray*}
\begin{aligned}
&\int_{B_{N,4,3}}\mu(B_h({\bf x}))^q\,d{\bf x}\\
\le&\int_{\widehat{U}_{N,4,3,1}(h)}\mu(B_{(\rho+r-\rho r)r^{N-1}+h}(S_{e_5^N}(3,0)))^q\,d{\bf x}
+\int_{\widehat{U}_{N,4,3,2}(h)}\mu(B_{r^N+h}(S_{e_5^N}(3,1)))^q\,d{\bf x}\\
&+\int_{\widehat{U}_{N,4,3,3}(h)}\mu(B_{r^N+h}(S_{e_5^N}(2,1)))^q\,d{\bf x}
+\int_{\widehat{U}_{N,4,3,4}(h)}\mu(B_{(r-\rho)\cdot r^{N-1}+2h}(S_{e_5^N}(3,1)+(0,-h)))^q\,d{\bf x}\\
&+\int_{\widehat{U}_{N,4,3,5}(h)}\mu(B_{\rho(1-r)\cdot r^{N-1}+2h}(S_{e_5^{N-1}e_4}(2,1)))^q\,d{\bf x}
+\int_{\widehat{U}_{N,4,3,6}(h)}\mu(B_{\rho r^{N-1}+h}(S_{e_5^{N-1}e_4}(2,0)))^q\,d{\bf x}\\
&+(w_{N-2}+p^{N-1}_{e_5})^q((\rho r^{N-1})^2+r^{2N})\\
\le&p^{Nq}_{e_5}\big((\rho+4r)+2p^{-q}_{e_5}p^q_{e_4}\rho\big)r^{N-1} h
+C((p_{e_5}+p_{e_6})^{1-N}+p^{-1}_{e_5})^q h^{2+\alpha_2+\epsilon}\\
\le&C\Big(r^{-1}\big((\rho+4r)+2p^{-q}_{e_5}p^q_{e_4}\rho\big)
+\big((p_{e_5}+p_{e_6})^{1-N}+p^{-1}_{e_5}\big)^q \Big) h^{2+\alpha_2+\epsilon}\\
\end{aligned}
\end{eqnarray*}
and hence
\begin{eqnarray*}
h^{-(2+\alpha_2+\epsilon/2)}\int_{B_{N,4,3}}\mu(B_h({\bf x}))^q\,d{\bf x}
\le C\Big(r^{-1}\big((\rho+4r)+2p^{-q}_{e_5}p^q_{e_4}\rho\big)
+\big((p_{e_5}+p_{e_6})^{1-N}+p^{-1}_{e_5}\big)^q \Big)h^{\epsilon/2}.
\end{eqnarray*}

(c) The proof is similar to that of (b).
\end{proof}

\begin{proof}[Proof of Corollary \ref{C:exam_GIFS_not_str_con_R2}]
By Theorem \ref{T:main_thm_not_str_con} and Proposition \ref{P:not_str_con_R2_est}, we have $\tau(q)=\alpha$.

The proof of the differentially of $\tau_q$ is similar to that of Corollary \ref{C:cor_GIFS_str_con_R2}.
\end{proof}


\begin{thebibliography}{11}

\bibitem{Cawley-Mauldin_1992}
R. Cawley and R. D. Mauldin, Multifractal decompositions of Moran fractals,
{\em Adv. Math.} {\bf 92} (1992), 196--236.

\bibitem{Das-Ngai_2004}
M. Das and S.-M. Ngai, Graph-directed iterated function systems with overlaps, {\em Indiana Univ. Math. J.} {\bf 53} (2004), 109--134.

\bibitem{Deng-Ngai}
G. Deng and S.-M. Ngai, Differentiability of $L^q$-spectrum and multifractal decomposition by using infinite graph-directed IFSs, \textit{Adv. Math.} \textbf{311} (2017), 190--237.

\bibitem{Edgar-Mauldin_1992}
G. A. Edgar and R. D. Mauldin,
Multifractal decompositions of digraph
recursive fractals, {\em Proc. London Math. Soc.} {\bf 65} (1992), 604--628.

\bibitem{Feng_2005} D. J. Feng, The limited Rademacher functions and Bernoulli
convolutions associated with Pisot numbers, {\em Adv. Math.} {\bf
195} (2005), 24--101.

\bibitem{Feng_2007}
D. J. Feng, Gibbs properties of self-conformal measures and the multifractal formalism, {\em Ergodic Theory Dynam. Systems} {\bf 27} (2007), 787--812.

\bibitem{Feng-Lau_2009} D. J. Feng and K.-S. Lau, Multifractal
formalism for self-similar measures with weak separation condition,
{\it J. Math. Pures Appl.} {\bf 92} (2009), 407--428.

\bibitem{Frisch-Parisi_1985}
U. Frisch and G. Parisi,
On the singularity structure of fully
developed turbulence, in {\em Turbulence and predictability in geophysical
fluid dynamics and climate dynamics} (eds M. Ghil, R. Benzi and G. Parisi)
(North-Holland, Amsterdam-New York, 1985), 84--88.

\bibitem{Halsey-Jensen-Kadanoff-Procaccia-Shraiman_1986}
T. C. Halsey, M. H. Jensen, L. P. Kadanoff,
I. Procaccia, and B. I. Shraiman,
Fractal measures and their singularities: the
characterization of strange sets, {\em Phys. Rev. A} {\bf 33}
(1986), 1141--1151.

\bibitem{Hambly-Nyberg_2003}
B. M. Hambly and S. O. G. Nyberg, Finitely ramified graph-directed fractals, sepctral asymptotics and the multidimensional
renewal theorem, {\em Proc. Edinb. Math. Soc.} {\bf 46} (2003), 1--34.

\bibitem{Kessebohmer-Niemann_2022_1}
M. Kesseb\"ohmer and A. Niemann, Spectral dimensions of Kre\u{\i}n-Feller operators and $L^q$-spectra,
{\em Adv. Math. } {\bf 399} (2022), Paper No. 108253.


\bibitem{Kessebohmer-Niemann_2022}
M. Kesseb\"ohmer and A. Niemann, Spectral dimensions of Kre\u{\i}n-Feller operators in higher dimensions, 10.48550/arXiv.2202.05247.


\bibitem{Lau_1992}
K.-S. Lau,
Fractal measures and mean $p$-variations, {\em J. Funct. Anal.} {\bf 108} (1992), 427--457.


\bibitem{Lau-Ngai_1998}
K.-S. Lau and S.-M. Ngai,
$L^q$-spectrum of the Bernoulli convolution associated with the golden ratio, {\em Studia Math.} {\bf 131} (1998), 225--251.

\bibitem{Lau-Ngai_1999}
K.-S. Lau and S.-M. Ngai,
Multifractal measures and a weak separation condition, {\em Adv Math.} {\bf 141} (1999), 45--96.

\bibitem{Lau-Ngai_2000}
K.-S. Lau and S.-M. Ngai,
Second-order self-similar identities and multifractal decompositions, {\em Indiana Univ. Math. J.} {\bf 49} (2000), 925--972.

\bibitem{Lau-Wang-Chu_1995}
K.-S. Lau, J. Wang, and C.-H. Chu,
Vector-valued Choquet-Deny theorem, renewal equation and self-similar
measures, {\em Studia Math.} {\bf 117} (1995), 1--28.


\bibitem{Mauldin-Williams_1988} R. D. Mauldin and S. C. Williams, Hausdorff dimension in graph directed constructions,
{\em Trans. Amer.  Math. Soc.} {\bf 309} (1988), 811--829.

\bibitem{Ngai_2011}
S.-M. Ngai, Spectral asymptotics of Laplacians associated with one-dimensional iterated function systems with overlaps,
{\em Canad. J. Math.} {\bf 63} (2011), 648--688.

\bibitem{Ngai-Tang-Xie_2018}
S.-M. Ngai, W. Tang, and Y. Xie, Spectral asymptotics of one-dimensional fractal Laplacians in the absence of second-order identities, {\em Discrete Contin. Dyn. Syst.} {\bf 38} (2018), 1849--1887.


\bibitem{Ngai-Xie_2019} S.-M. Ngai and Y. Xie, $L^q$-spectrum of self-similar measures with overlaps in the absence of
second-order identities, \textit{J. Aust. Math. Soc.} \textbf{106} (2019), 56--103.

\bibitem{Ngai-Xie_2020} S.-M. Ngai and Y. Xie, Spectral asymptotics of Laplacians related to one-dimensional
graph-directed self-similar measures with overlaps, \textit{Ark. Mat.} \textbf{58} (2020), 393--435.


\bibitem{Wang_1997}
J. L. Wang, The open set condition for graph directed self-similar sets, \textit{Random Comput. Dynam.} \textbf{5} (1997), 283--305.

\end{thebibliography}
\end{document}